\documentclass[12pt]{amsart}

\usepackage{macros}
\standardmargins\standardrenews\standardmakes\standardlabeling
\colorcommentstrue

\PassOptionsToPackage{hyphens}{url}
\usepackage[breaklinks,hidelinks]{hyperref}
\usepackage{color}
\usepackage[usenames,dvipsnames]{xcolor}
\usepackage{pgf, tikz}
\usetikzlibrary{arrows, patterns}

\newcommand{\ego}{I}
\newcommand{\me}{me}
\newcommand{\my}{my}

\newcommand{\He}{He}
\newcommand{\he}{he}
\newcommand{\him}{him}
\newcommand{\his}{his}

\renewcommand{\SS}{\mathcal S}

\newcommand{\statement}{sentence}
\newcommand{\statements}{sentences}

\newcommand{\LQ}{`}
\newcommand{\RQ}{{\textrm'}}
\newcommand{\mq}[1]{\text{`$#1$'}}

\newcommand{\List}{\text{List}}

\newcommand{\bq}{<}

\newcommand{\lv}{\lb}
\newcommand{\rv}{\rb}

\renewcommand{\N}{\textbf{\textup{N}}}

\renewcommand{\K}{\textbf{\textup{K}}}
\renewcommand{\C}{\textbf{\textup{C}}}
\newcommand{\GC}{\textbf{\textup{GC}}}

\newcommand{\ST}{\textbf{\textup{ST}}}
\newcommand{\SCT}{\textbf{\textup{SCT}}}
\newcommand{\GST}{\textbf{\textup{GST}}}
\newcommand{\GDST}{\textbf{\textup{GDST}}}
\newcommand{\ZF}{\textbf{\textup{ZF}}}

\newcommand{\ZFC}{\textbf{\textup{ZFC}}}
\newcommand{\NBG}{\textbf{\textup{NBG}}}
\newcommand{\KP}{\textbf{\textup{KP}}}
\newcommand{\GKP}{\textbf{\textup{GKP}}}
\newcommand{\KPC}{\textbf{\textup{KPC}}}

\newcommand{\NT}{\textbf{\textup{NT}}}
\newcommand{\GNT}{\textbf{\textup{GNT}}}
\newcommand{\PA}{\textbf{\textup{PA}}}

\newcommand{\EFA}{\textbf{\textup{EFA}}}
\newcommand{\PRA}{\textbf{\textup{PRA}}}

\newcommand{\elt}{\text{elt}}
\newcommand{\fst}{\text{fst}}
\newcommand{\reclang}{\NT!}
\newcommand{\lang}{L}

\newcommand{\ex}{\downarrow}

\newcommand{\philang}[1]{\phi_\lang\{#1\}}
\newcommand{\fA}[1]{f_A\{#1\}}
\newcommand{\alphaA}[1]{\alpha_A\{#1\}}
\newcommand{\double}[2]{#1\{#2\}}

\newcommand{\LNC}{\textbf{\textup{LNC}}}

\theoremstyle{theorem}

\makedefinition{thesis}{Thesis}

\theoremstyle{definition}

\makedefinition{intpict}{Intuitive Picture}
\makedefinition{lnc}{Contest Definition}
\makedefinition{lncs}{Contest Definition}
\makedefinition{axpred}{Axioms of Predicativity}
\makedefinition{axiom}{Axiom}
\makedefinition{principle}{Principle}
\makedefinition{pcrit}{Preciseness Criterion}

\draftfalse

\begin{document}
\title[Naming the largest number]{Naming the largest number \\ \tiny{Exploring the boundary between mathematics and the philosophy of mathematics}}

\authordavid


\begin{abstract}
What is the largest number accessible to the human imagination? The question is neither entirely mathematical nor entirely philosophical. Mathematical formulations of the problem fall into two classes: those that fail to fully capture the spirit of the problem, and those that turn it back into a philosophical problem.
\end{abstract}
\maketitle

%

\section{Introduction}
\label{sectionintro}
\begin{quote}
``You have fifteen seconds. Using standard math notation, English words, or both, name a single whole number--not an infinity--on a blank index card. Be precise enough for any reasonable modern mathematician to determine exactly what number you've named, by consulting only your card and, if necessary, the published literature. Are you ready? Get set. Go.'' -- Scott Aaronson, ``Who can Name the Bigger Number?'' \footnote{\url{http://www.scottaaronson.com/writings/bignumbers.html}}\\~\\
``The game of describing the largest integer, when played by experts, lapses into a hopeless argument over legitimacy.'' -- Joel Spencer, ``Large numbers and unprovable theorems'' \cite{Spencer}
\end{quote}~

This paper addresses the question: what is the best strategy in a ``name the largest number'' contest such as the one described above? Of course, the answer depends on what the rules of the contest are. Scott Aaronson's idea of letting a ``reasonable modern mathematician'' judge the entries might sound fair at first, but the answer to the question of whether a given entry is a precise description of a number may depend on exactly which mathematician you ask. It turns out that this dependence is not incidental but fundamental, and slightly different standards for what counts as a valid entry give rise to radically different strategies for winning the game.

One way to make precise the notion of a ``precise description of a number'' is to fix a \emph{precise language}, i.e. a language whose syntax and semantics are both precisely defined, for describing numbers.\Footnote{\ego\ warn the reader that the phrase `precise language' is not standard terminology. While \ego\ will attempt to use standard terminology wherever possible, in some cases this is not possible; for instance, the reason \ego\ use the phrase `precise language' is that the more standard phrase `formal language' tends to refer only to syntax and not to semantics. The concept of a formal language is amenable to mathematical definition, whereas the concept of a precise language is not unless it is restricted in some way.} 
A major goal of mathematics in the nineteenth and early twentieth centuries was to create precise languages capable of expressing mathematical thought. According to one view, which I will call the standard mathematicians' view, this goal was accomplished by the introduction in the early twentieth century of the language of (first-order,\Footnote{In this paper \ego\ may seem to have an unjustified preference for first-order languages, as opposed to things like second-order languages and languages that allow plural quantification. However, any second-order language can be interpreted as a first-order language by the expedient of viewing the ``properties'' in the original language as a type of object that one is allowed to quantify over. Although statements involving plural quantification are more difficult to straightforwardly translate into first-order form, it seems to me that in all cases where the original sentence is not ambiguous, it is possible to do so, and that moreover the issues that concern philosophers regarding plural quantification have little do to with the issues considered in this paper.} classical,\Footnote{Here `classical' means ``based on classical logic'' -- in particular, every statement in a classical language is either true or false.} Zermelo--Fraenkel) \emph{set theory}, which is allegedly capable of codifying all mathematical facts. Another important language developed in the late nineteenth century is (first-order, classical, Peano) \emph{arithmetic} or \emph{number theory}. Roughly speaking, set theory is the language in which the fundamental objects of discourse are sets and the fundamental relation is membership, and number theory is the language in which the fundamental objects of discourse are the natural numbers $0,1,2,3,\ldots$ (which we will hereafter refer to simply as ``numbers'') and the fundamental operations are addition and multiplication. The syntaxes of both set theory and number theory have precise definitions, which are given in any good textbook on mathematical logic. However, it is open to philosophical debate whether the languages can be assigned precise semantics; see Sections \ref{subsectionplan}, \ref{sectionnumberrealism}, and \ref{sectionsetrealism} for further discussion. Here \ego\ will attempt to take a neutral position with respect to this debate, and merely analyze the consequences of any particular set of beliefs regarding the preciseness of mathematical languages. In what follows, \ego\ will denote the languages of set theory and number theory by the labels \ST\ and \NT, respectively.\Footnote{The languages of set theory and number theory are sometimes denoted by the labels \ZFC\ and \PA, respectively, as shorthand for `Zermelo--Fraenkel set theory with Choice' and `Peano Arithmetic'. These labels are also used to denote standard axiom systems associated with these languages. However, in order to emphasize the distinction between a language and an axiom system associated with that language, \ego\ reserve the labels \ZFC\ and \PA\ for the axiom systems and \ego\ use the alternate labels \ST\ and \NT\ for the languages. \ego\ chose new labels not based on mathematicians' names in order to suggest that the languages are somehow less ``historically contingent'' than the axiom systems, which \ego\ believe to be the case.} It is well-known that \NT\ has significantly less expressive power than \ST\ in the sense that there is an algorithm for translating \statements\ of \NT\ into \ST\ (which claims to preserve meaning or at least truth-values), but not vice-versa.


If we have in mind a precise language, then we can define a precise largest number contest by declaring that an entry to the contest is valid if and only if it is a term\Footnote{The word `term' is a technical term in mathematical logic whose closest non-technical analogue is something like ``a singular noun phrase that includes a definite article''. The key thing about terms is that each term is intended to refer to a unique object.} of the language that names a number. We can write this in semi-formal mathematical language as follows:
\begin{lncs}
\label{lnc1}
Let $\lang$ be a precise language. The \emph{largest number contest of $\lang$}, denoted $\LNC(\lang)$ is the contest in which a valid entry is a term of $\lang$ that succeeds at referring to a number, and the score of the entry is the number it refers to. (In all of our contests, the winner is the valid entry with the highest score.) We will denote the score of an entry $\phi$ by $[\phi]$.
\end{lncs}

To many modern mathematicians, the largest number contest of \ST\ will seem like the obvious ``first choice'' for a largest number contest. On the other hand, many people are uncertain as to whether set theory is really a precise language, but are more confident in number theory; the largest number contest of \NT\ may be a more natural ``first choice'' for such people.\Footnote{There is the minor technical issue that the terms of these languages have much less expressive power than the languages themselves; this can be remedied by adding new rules for constructing terms that bring the language's terms on a par with the rest of the language in terms of expressive power, without increasing the expressive power of the language as a whole. See Appendix \ref{appendixPT} for details.} Each of these contests has the disadvantage that it is not possible to determine whether any given entry is valid in an algorithmic way. In fact, due to Tarski's theorem on the undefinability of truth (see Section \ref{sectionphilosophical}) it is not even possible to formalize the notion of a ``valid entry'' in the precise language being used for the contest (namely \ST\ or \NT). Despite this lack of formalizability, it appears that the question of which entries are valid is still legitimate and precise, assuming that the language used to define the contest has in fact been given precise semantics.\Footnote{Some mathematicians trained under the set theory framework may object that any question that cannot be translated into \ST\ should be suspected of being imprecise. However, this view is subject to the counterargument that any reason to believe that \ST\ is a precise language is also a reason to believe that the question of which entries in \LNC(\ST) are valid is precise; on the other hand, if \ST\ is not precise, then it is not clear why one would use it as a standard of precision. Some further discussion may be found in Section \ref{sectionsetrealism}.} So it makes sense to ask what the best strategy to use is, and in fact it is not too hard to find a strategy for the contest which is close to being optimal. This strategy is based on the following theorem:\Footnote{The proofs of all theorems and corollaries, as well as the explanation of the axiom system they are understood to be theorems of, are given in Appendices \ref{appendixPT}-\ref{appendixproofs}, where more general versions of the theorems are also stated.}

\begin{theorem}[Cf. Theorem \ref{theoremrecstratB}]
\label{theoremrecstrat}
Let $\lang$ be either set theory or number theory. Then for each number $n$, the phrase `the largest number that can be represented in $\lang$ by a term consisting of $n$ symbols or fewer'\Footnote{Cf. \cite[Figs. 9-10]{Spencer}, \url{http://web.mit.edu/arayo/www/bignums.html}} can be translated into $\lang$ as a term consisting of $O(n)$ symbols.\Footnote{Recall that $O(n)$ is shorthand for `at most a constant times $n$'.} This term, which we will denote by $\philang{n}$, can be written explicitly.
\end{theorem}

Our strategy for the largest number contest of $\lang$ can now be given as follows (assuming that $\lang$ is precise): first compute the largest number $n$ such that the term $\philang{n}$ can fit onto the ``index card'' Scott Aaronson gives us to write our entry on (small handwriting will come in handy here), and then write $\philang{n}$ on the card.

This strategy, which we will call the \emph{brute force strategy} of $\lang$, is not necessarily optimal, but it is \emph{nearly optimal} in the sense that no strategy can beat it too badly. To be precise, by the definition of $\philang{n}$, any entry that scores higher than $\philang{n}$ necessarily uses at least $n+1$ symbols to do so, which is not too much less than the $O(n)$ symbols needed by the brute force strategy. So if there is a better strategy, the difference is not overwhelming in the sense that if one player plays with the brute force strategy and another player plays with a better strategy, who wins may be determined more by who has more resources to express their answer (for example, small handwriting) than it is by who is using the better strategy. (This is the reason for the name ``brute force strategy''.) The question of who wins also starts to depend on the specific details of the formal definition of $\lang$; for example, augmenting \ST\ with the primitive ability to write `$x\notin y$' (``the set $x$ is not a member of the set $y$'') as shorthand for `$\neg x\in y$' (``it is not true that $x$ is a member of $y$''), thus saving a symbol, may make a crucial difference to the outcome. \ego\ will take the point of view that any improvement that depends on such details is not a significant improvement, and thus \ego\ will consider the ``near-optimality'' possessed by the brute force strategy to be essentially the same as optimality. (This is a standard practice in many areas of mathematics and computer science.)

An obvious objection is that the largest number contest of $\lang$, though precise, does not stay true to the ``spirit'' of being a ``name the largest number'' contest. For example, a description of the term $\philang{n}$ in terms of $n$ could perhaps be made to fit on an index card,\Footnote{If this paper counts as ``published literature'' then one could simply write `$\philang{10^{100}}$' together with the appropriate citation.} and then one could simply write `and my number is the number denoted by the term $\philang{10^{100}}$'.\Footnote{Alternatively, one could simply write the original English phrase `the largest number that can be represented in $\lang$ by a term consisting of $10^{100}$ symbols or fewer' on the index card, as is done for example in \cite{Spencer}. However, the point is that the fact that Theorem \ref{theoremrecstrat} specifies a procedure for (eventually) translating this phrase into $\lang$ might be thought of as adding some ``legitimacy'' to the indirect strategy, as described in the next paragraph.} We will call this the \emph{indirect strategy}, and we note that the entry it produces is not a valid entry to $\LNC(\lang)$, since it is an English description of a number rather than a term of the language $\lang$. Since $10^{100}$ is much larger than the number of symbols that any physically realistic entry could have, in a largest number contest in which the entry produced by the indirect strategy is considered valid, it wins against any entry which is valid in $\LNC(\lang)$. However, it seems that the indirect strategy should be allowed if we follow Aaronson's principle that a player only needs to submit an entry that is ``precise enough for any reasonable modern mathematician to determine exactly what number [the player has] named'' (assuming, of course, that $\lang$ is a precise language to begin with).

The objection becomes more poignant if we consider the fact that the hypothetical entry of the previous paragraph is not only perfectly precise, but also has the property that it can be converted into a valid entry to $\LNC(\lang)$ given a sufficient (albeit physically unrealistically large) amount of time and effort. Concentrating on this part of the objection, we can imagine a second contest in which the entries are not required to be terms of the language $\lang$, but are allowed to be merely \emph{descriptions} of terms. The crucial question now becomes what sorts of descriptions of these terms are allowed. One possibility would be to allow these descriptions to be themselves terms of the language $\lang$, or perhaps terms of a different precise language that has the ability to talk about strings of symbols.\Footnote{It is well-known that the theory of strings of symbols can be translated into both \ST\ and \NT.} However, it is not hard to see that the ``brute force strategy'' described above can be trivially modified to deal with this new game: instead of writing $\philang{n}$, write `$\philang{m}$, where $m = \cl{\philang{n}}$'\Footnote{Here $\cl{\philang{n}}$ means to insert the numeral referring to the number $\philang{n}$ into the string.} for an appropriate value of $n$. Things get more interesting if we look at other ways of describing strings. For now at least, let us take the conservative viewpoint that in order to count for the purposes of the contest, a description of a string should contain enough information that it is possible to use it to algorithmically generate the string it describes, although this may take an arbitrarily long time. Semi-formally:
\begin{lncs}
\label{lnc1.1}
Let $\lang$ be a precise language. The \emph{largest number contest of $\lang$ allowing indirect entries}, denoted $\LNC^*(\lang)$, is the contest in which a valid entry is an algorithm (written in some fixed programming language the details of which are irrelevant\Footnote{Note that a programming language is a particular type of precise language: a language for precisely specifying mathematical algorithms. In this paper a \emph{program} is understood to be a description of an algorithm in a programming language.}) that succeeds at (eventually) computing a term of $\lang$ that refers to a number. If the algorithm fails to terminate by returning a string as its output, or if the string it returns is not a term of $\lang$ referring to a number, then the entry is considered invalid. The score of a valid entry is the number that its output refers to.
\end{lncs}
Taking our cue from the strategy that ``won'' $\LNC(\lang)$, a natural strategy for $\LNC^*(\lang)$ is to first find an algorithm that computes a very large number $n$, and then to combine it with a second algorithm that computes the string $\philang{n}$ given $n$ as input. However, this leaves open the question of how to write an algorithm that computes a very large number. It is worth listing this as a separate largest number contest:
\begin{lnc}
\label{lnc2}
The \emph{Busy Beaver contest}\Footnote{The Busy Beaver contest is inspired by the \emph{Busy Beaver function}. Given a number $n$ as input, the Busy Beaver function returns the largest number that is the output of some algorithm that can be written in at most $n$ symbols (in some fixed programming language). Equivalently, the value of the Busy Beaver function at $n$ is the largest number computed by any possible entry to the Busy Beaver contest written in at most $n$ symbols. (Technical note: The classical Busy Beaver function is defined not in terms of the number of symbols in a description of an algorithm, but rather in terms of the \emph{number of states} in a (one-tape, two-symbol) \emph{Turing machine representation} of the algorithm. However, the length of a program encoding an algorithm is a much more natural metric of the algorithm's complexity than the number of states in a Turing machine representation, and \ego\ have no qualms with saying that the ``true'' Busy Beaver function is the one described above.)} is the contest in which a valid entry is an algorithm that succeeds at computing a number, and the score of the entry is its output.
\end{lnc}

It is clear that any strategy for the Busy Beaver contest whose score is $n$ can be converted into a strategy for $\LNC^*(\lang)$ whose score is $[\philang{n}]$, and that if a strategy is optimal or nearly optimal for the Busy Beaver contest, then its converted version is nearly optimal for $\LNC^*(\lang)$. Thus, we can avoid dealing with $\LNC^*(\lang)$ directly and instead analyze the Busy Beaver contest in its place.

\begin{remark*}
Any strategy for the Busy Beaver contest can also be converted into a strategy for $\LNC^*(\NT)$ or $\LNC^*(\ST)$ with the same score as the original strategy, via the fact that phrases of the form `the output of the algorithm $\alpha$, if it exists' can be translated into \NT\ and \ST. However, \NT\ and \ST\ are capable of naming much larger numbers, and scores in the Busy Beaver contest will in general be much lower than scores in the largest number contests of \NT\ and \ST.
\end{remark*}


It turns out that the best strategy for the Busy Beaver contest is very different from the best strategy for $\LNC(\ST)$ or $\LNC(\NT)$. To describe this strategy, we need the notion of an \emph{axiom system} for a precise language. An axiom system $A$ for a precise language $\lang$, also known as an \emph{axiomatization} of $\lang$, is a list of \emph{axioms} and \emph{inference rules}. An axiom is a \statement\ of $\lang$ which is either considered obviously true, or merely hypothesized to be true,\Footnote{A major difference between the mathematical and philosophical communities is that mathematicians generally view axioms as ``hypotheses'' while philosophers usually view axioms as ``obvious or self-evident truths''. Furthermore, while some mathematical hypotheses are considered likely to be true, others are introduced purely for the sake of argument.} while an inference rule is a syntactic rule for creating new \statements\ of $\lang$ from previous ones, with the understanding that the inference rule is supposed to formalize some method of reasoning, in the sense that the previous \statements\ represent ``prior beliefs'' and the new \statements\ represent ``deductions'' or ``conclusions formed using the reasoning method''. A list of claims of $\lang$ is said to be a \emph{syntactic proof in $A$} if each claim on the list is either an axiom or can be created from the previous ones using some inference rule. We use the phrase `syntactic proof' to distinguish between a syntactic proof and a \emph{semantic proof} of a claim, which is an argument that convinces us, or would convince us, that the claim is true. An axiom system aspires for its syntactic proofs to be also semantic proofs; when this is the case we call the axiom system \emph{valid}. The standard axiomatizations of \ST\ and \NT\ are the Zermelo--Fraenkel axiom system for set theory and the Peano axiom system for number theory, which we will denote by \ZFC\ and \PA, respectively (for ``Zermelo--Fraenkel set theory with Choice'' and ``Peano Arithmetic'', respectively).

In what follows we consider a pair $(\lang,A)$ where $\lang$ is a precise language and $A$ is a valid axiom system for $\lang$. In order to prove results about $(\lang,A)$ we must make some additional assumptions about the structure of $\lang$; specifically, we will assume that $\lang$ is a \emph{proposition/term language} as defined in Appendix \ref{appendixPT}. Note that both \ST\ and \NT\ are proposition/term languages in this sense.


\begin{theorem}[Cf. Theorem \ref{theoremaxiomstratB}]
\label{theoremaxiomstrat}
Let $A$ be a valid axiomatization of a precise proposition/term language $\lang$ such that all \statements\ of the form `the algorithm $\alpha$ eventually halts', where $\alpha$ is a concrete algorithm, can be translated into $\lang$. Then for each $n$, there is a program of length $O(\log(n))$ encoding an algorithm that succeeds at computing $\fA{n}$, the largest number which is the output of an algorithm that can be proven to halt via a proof in $A$ of length at most $n$. This program, which we will denote by $\alphaA{n}$, can be written explicitly.
\end{theorem}

\begin{remark*}
Note that the length of the program $\alphaA{n}$ is only $O(\log(n))$ rather than $O(n)$. The reason for this is that $\alphaA{n}$ only needs to describe $n$ in sufficient detail that it can be computed (for example by giving its binary or decimal representation), rather than needing to contain $n$ symbols to start with.
\end{remark*}

Our strategy for the Busy Beaver contest can now be given succinctly as follows: submit the program $\alphaA{10^{100}}$. This is feasible because a program of length $O(\log(10^{100}))$ is ordinarily short enough that it can be physically implemented.\Footnote{Of course, this depends on the constant implied by the $O(\cdot)$ notation, but it appears not to be too large.} Although theoretically one may submit the program $\alphaA{n}$ for much larger values of $n$, it turns out that it is not worth it to do so because gains from this sort of change are overwhelmed by much larger easy improvements elsewhere, such as replacing $(\lang,A)$ by the metasystem $(\lang+1,A+1)$ described below.

Unlike the brute force strategy, the above strategy, which we will call the \emph{axiomatic strategy} of $A$, is not nearly optimal. In fact, there is no humanly understandable nearly optimal strategy, since any nearly optimal entry will appear to us as a nearly random sequence of symbols; see Appendix \ref{appendixmoreBBC} for details. However, if the language $\lang$ and the axiom system $A$ are chosen wisely, the strategy can have a different type of optimality, namely \emph{relative epistemological optimality}. The idea is that even though there may be many possible valid entries that would win against the axiomatic strategy if they were submitted as competitors, humanity may never be able to determine which of these entries are in fact valid. So if you are playing against human opponents that prefer to only submit entries that they believe are valid, then with the axiomatic strategy you may have a reasonable shot at winning.

Suppose that Alice and Bob are human players, and that Alice has just submitted an entry $\alphaA{10^{100}}$ using the axiomatic strategy, while Bob is going to submit an entry $\alpha_0$ that he believes is valid. Presumably, this means that Bob has executed (either internally or externally\Footnote{An example of external execution of reasoning is a computer-assisted proof. Alternatively, we could say that a human who comes to believe the conclusion of a computer-assisted proof is applying internal reasoning using the belief that the external world (and in particular the computer performing the calculations used in the proof) operates according to the laws of physics. However, according to this second way of looking at things the chain of reasoning cannot properly be described as mathematical or strictly deductive. Thus the second way of looking at things makes the situation more difficult to analyze, which is why \ego\ stick to considering a computed-assisted proof as external reasoning.}) some chain of reasoning whose conclusion is that $\alpha_0$ is valid. Although it is theoretically possible that this chain of reasoning cannot be formalized in any axiom system, it seems plausible that in practice it usually can be, and that in fact there is usually an axiom system capable of formalizing Bob's reasoning which is not too much more powerful than an axiom system that Bob explicitly endorses. Moreover, arguments that can be formalized in one axiom system also tend to be able to be formalized in more powerful axiom systems, so as long as the axiom system that Bob's reasoning can be formalized in is less powerful than Alice's preselected axiom system $A$, then Bob's reasoning can probably also be formalized in $A$ as well. Finally, if we formalize Bob's reasoning in the axiom system $A$, then we may expect that the formalization of Bob's reasoning will not be too much longer than a formalization 
of his reasoning in English (or another natural language); all that is needed is that the ratio between the two lengths will not be more than (for example) $10^{50}$. On the other hand, due to physical limitations that apply regardless of whether Bob's reasoning is internal or external, 
it is reasonable to expect that we can formalize his reasoning in English using no more than (for example) $10^{50}$ symbols, and thus that we can formalize his reasoning in the axiom system $A$ using no more than $10^{100}$ symbols. If this turns out to be the case, then $\alpha_0$ can be proven to halt in $A$ using reasoning encoded in no more than $10^{100}$ symbols. Thus the output of $\alpha_0$ is at most $\fA{10^{100}}$, which is the output of Alice's entry $\alphaA{10^{100}}$. So Alice will win.\Footnote{The philosophical argument given in this paragraph is significant not for its rigor, which is practically non-existent, but because it seems to mesh well with what we expect. It seems to \me\ that if largest number contests were actually held, many of them would work in exactly this way.} \Footnote{It is also interesting to ask what might happen if Bob bases his beliefs on probabilistic or heuristic reasoning. For example, many of the arguments for accepting the validity of strong axiomatizations of set theory (i.e. large cardinal axioms) are based on such reasoning. Recently there has been much interest in the difficult question of how to formalize probabilistic reasoning about mathematical concepts (see e.g. \cite{GBCST} and the references therein), but it is not clear to me that any of the frameworks proposed so far are capable of modeling (for example) the standard heuristic arguments in favor of measurable cardinals (see \cite{Maddy1,Maddy2}). Thus, if Bob bases his belief that his entry is valid on probabilistic reasoning, it may be necessary to find an ``intermediate axiom'' which Bob would also accept on the basis of the same probabilistic reasoning, which implies that Bob's entry is valid (in the axiomatic framework Bob already accepts), and which is a consequence of the axiom system $A$ chosen by Alice, in order for the argument to apply.}

In order to make the analysis fair, \ego\ should also address the question of why Alice might believe that her entry is valid. Its validity cannot be proven in the axiom system $A$ in a physically realistic timeframe, since if it could then Alice would win if she played against herself, which makes no sense. However, it turns out that there is a precise language and corresponding axiom system not too much more powerful than the pair $(\lang,A)$ capable of formalizing the reasoning used in the proof of Theorem \ref{theoremaxiomstrat}: namely, the pair $(\lang+1,A+1)$, where
\begin{itemize}
\item $\lang+1$ is the \emph{metalanguage} of $\lang$, a precise proposition/term language capable of expressing the notions of truth and reference with respect to $\lang$ in the sense that claims of the form `$\phi$ is a true \statement\ of $\lang$' and terms of the form `the object denoted by the $\lang$-term $t$' can be straightforwardly translated into $\lang+1$; and
\item $A+1$ is the \emph{metasystem} of $A$, a natural axiomatization of $\lang+1$ which formalizes the assumption that the axiom system $A$ is valid as well as standard beliefs and methods of reasoning about strings of symbols.
\end{itemize}
(See Appendices \ref{appendixPT} and \ref{appendixaxiomatizations} for more details.) Thus, Alice may come to believe that her entry is valid using reasoning that can be formalized in the axiom system $A+1$.

In what follows, we will metonymically refer to a pair $(\lang,A)$, where $\lang$ is a precise language and $A$ is an axiomatization of it, as an axiom system, and to the pair $(\lang+1,A+1)$ as the metasystem of $(\lang,A)$.

\begin{corollary}[Cf. Corollary \ref{corollaryaxiomstratB}]
\label{corollaryaxiomstrat}
Let $(L_1,A_1)$ and $(L_2,A_2)$ be valid axiom systems with the translation property hypothesized in Theorem \ref{theoremaxiomstrat}, and suppose that $(L_1,A_1)$ is at least as powerful as the metasystem $(L_2+1,A_2+1)$ in the sense that all proofs of $(L_2+1,A_2+1)$ can be translated into $(L_1,A_1)$. Then in the Busy Beaver contest, the axiomatic strategy of $(L_1,A_1)$ wins agains the axiomatic strategy of $(L_2,A_2)$ (assuming both axiom systems as well as the algorithm for translating one into the other are small enough that they could be physically implemented).
\end{corollary}

Consequently, the key question to ask about the contestants to the Busy Beaver contest is: what is the weakest axiom system capable of formalizing their reasoning? If Alice uses reasoning incapable of being formalized in an axiom system $A$ capable of formalizing Bob's reasoning, this may lead her to believe that the axiom system $A$ is valid, in which case she can win simply by playing its axiomatic strategy. If both players' logic is best described in terms of the same axiomatic framework, then we are led to the consideration of another type of largest number contest.

A common defect in all of the largest number contests considered above is that there is no clear way for human judges to determine whether any given entry is valid or not. Even if a term $\phi$ refers to a number (and is therefore a valid entry to the first contest), there may be no proof that $\phi$ refers to a number in any axiom system that the judges accept as valid. However, the optimal strategies for the largest number contests of \ST\ and \NT, i.e. the brute force strategies of \ST\ and \NT, are relatively uncontroversial (since the proof of Theorem \ref{theoremrecstrat} may be considered valid by the judges) so the issue does not matter too much for those contests. In the Busy Beaver contest the situation is different. Theoretically a player could run the algorithm described in \his\ entry to show that it halts, but by design this would take an extremely large amount of time. The contest may be modified by adding the requirement that \emph{contestants must submit, as an addendum to their entry, a proof that their entry is valid}. Then whatever reasoning convinces the contestants of the validity of their entry could also be shown to the judges.

This new requirement brings us back around almost full circle: like the question of what constitutes a precise description of a number, the question of what constitutes a valid proof that an entry is valid is subject to philosophical debate, unless we describe a precise criterion, such as validity in a fixed axiom system, for what counts as a proof. But if we introduce such a structure, then as before it turns out that the best strategy is a brute force strategy:

\begin{lncs}
\label{lnc3}
Let $(\lang,A)$ be as in Theorem \ref{theoremaxiomstrat}. The \emph{Busy Beaver proof contest of $(\lang,A)$} is the contest in which a valid entry is a pair of the form $(\alpha,\rho)$, where $\alpha$ is an algorithm and $\rho$ is a proof in $A$ that $\alpha$ succeeds at computing a number, and the score of the entry $(\alpha,\rho)$ is the output of $\alpha$.
\end{lncs}


\begin{theorem}[Cf. Theorem \ref{theoremrecstrat2B}]
\label{theoremrecstrat2}
Let $\lang$ be either set theory or number theory, and correspondingly let $A$ be either \ZFC\ or \PA.\Footnote{The proof of this theorem is also valid for a much broader class of languages and axiom systems; see Appendices \ref{appendixPT}-\ref{appendixproofs} for details.} Then for each $n$, the $n$th case of Theorem \ref{theoremaxiomstrat} can be (syntactically\Footnote{If the axiom system $(\lang,A)$ is valid, then this syntactic proof is also a semantic proof.}) proven in $A$ using reasoning encoded in at most $O(n^2\log(n))$ characters.\Footnote{This number is possibly somewhat sensitive to the exact definition of a ``proof in the axiom system $A$''. Though many definitions of the word ``proof'' exist in the literature, this theorem is based on the definition of ``proof'' described in Appendix \ref{appendixaxiomatizations}, which is intended to closely model the notion of ``proof'' in standard mathematical practice, including automated proof checkers.} The proof in $A$, which we will denote by $\double{\rho_A}{n}$, can be written explicitly.
\end{theorem}
The brute force strategy for the Busy Beaver proof contest of $(L,A)$ can then be given as follows: First compute the largest number $n$ such that the $\double{\rho_A}{n}$ is short enough to be submitted as an addendum to the entry $\alphaA{n}$, and then submit the pair $(\alphaA{n},\double{\rho_A}{n})$ as an entry. (Note that the length of $\double{\rho_A}{n}$ is the limiting factor here; $\alphaA{n}$ is much shorter.) Like the brute force strategies of the largest number contests of \ST\ and \NT, this strategy is nearly optimal in the sense that no other valid strategy can beat it too badly. However, in this case the near-optimality is a little bit worse in the sense that the number of symbols needed is $O(n^2\log(n))$ rather than $O(n)$. It turns out that if we consider proofs that incorporate ``shorthand'' corresponding to the common mathematical practice of giving ``definitions'', then the number of symbols needed can be reduced to $O(n\log(n))$.

Like the largest number contests of \ST\ and \NT, an objection can be raised to the various Busy Beaver proof contests on the grounds that they are not ``real'' largest number contests: for any such contest, it is easy to describe entries that ``should be'' accepted as valid since their corresponding proofs use methods of reasoning just a little bit outside of the axiomatic framework that has been fixed for the judging, which are easily seen to win against all valid entries by a large margin. However, we seem to have exhausted all our ``possible moves'' for getting away from the problem (though we will see later that this is not quite true), so perhaps it is best to just admit defeat: there is no good way to formalize the largest number contest.

To be specific, an important thesis of this paper is the following:
\begin{thesis}
\label{mainthesis}
For any precise largest number contest, the best strategy is either a brute force strategy or an axiomatic strategy.
\end{thesis}
\ego\ will not argue for this thesis other than by pointing to the lack of counterexamples despite encountering a broad range of precise largest number contests. The real test will be to see whether a counterexample will be found in the future.

In contests where the optimal strategy is a brute force strategy, there is always the feeling that the spirit of the contest has been lost since it is easy to give descriptions of numbers that are just as precise as the implicit descriptions of numbers provided by valid entries in the contest, that nevertheless win against all legitimate entries. Moreover, the fact that victory in the contest is tied to merely being able to physically write down a certain sequence of symbols, when we all know in advance exactly what those symbols will be, makes a mockery of what was supposed to be a test of wits. On the other hand, in contests where the best strategy is axiomatic, it is impossible to determine which entries are valid in an algorithmic way, and this problem is most conspicuous exactly when we look at the strategies that appear to have the best chance of winning, namely the axiomatic strategies of the most powerful axiom systems for the most expressive languages. In general the plausibility of an axiom system is inversely proportional to its power, so this leads us to the uncomfortable situation where the strategies which are the best if valid are also those which are the least likely to be valid, and the most difficult to justify.\\

{\bf Acknowledgements.} The author thanks Lance and Joseph Simmons for comments on an earlier version of this paper, and Scott Aaronson for writing about the question from which this paper arose. The author was supported in part by EPSRC Programme Grant EP/J018260/1.

\tableofcontents

\section{Philosophical largest number contests}
\label{sectionphilosophical}

One reaction to the difficulties of formalizing the largest number problem is to eschew formalization and instead consider a different, more ``philosophical'', approach to the challenge of implementing a largest number contest. Namely, we can take literally the idea of judging the entries according to the standard of whether a ``reasonable modern mathematician'' would find them to be precise. The fact that different mathematicians may give different answers to this question can be dealt with by specifying in advance who will judge the contest. Alternatively, we may divide mathematicians into broad classes based on their philosophy of mathematics, and allow each philosophical school to elect a panel of judges to represent it. Each entry could then spark a philosophical debate within the panel as to whether or not it is valid. Semi-formally:

\begin{lncs}
\label{lnc4}
Let $X$ be a school of thought in the philosophy of mathematics. Then the \emph{largest number contest run by adherents of $X$} is the contest in which an entry consists of a natural language description of a number together with a philosophical argument defending the premise that the description in fact uniquely identifies a number. The entry is valid if the adherents of $X$ agree with the argument, and the winner is whichever entry the adherents of $X$ believe describes the largest number.
\end{lncs}

\begin{remark*}
In practice, the matter of deciding which entry describes the largest number is much more straightforward than the matter of deciding which entries are valid.
\end{remark*}

Obviously, largest number contests run by philosophers are not precise in the same way that the contests of the previous section are. Thus, instead of immediately describing what the optimal strategy is, \ego\ will discuss possible strategies for this contest for various schools of thought over the next few sections. In particular, \ego\ will argue that in most cases Thesis \ref{mainthesis} applies to this contest as well, in the following sense: With enough ``nudging'', it is possible to convince the adherents of any school of thought that their largest number contest can be reduced to a largest number contest in which either a brute force strategy or an axiomatic strategy is optimal. The idea is that there are only a few ``free variables'' in the interpretation of a philosophy that need to be pinned down before its largest number contest becomes a precise largest number contest.

Actually, most philosophers will find the idea that a philosophical largest number contest can be won with a brute force strategy to be abhorrent, since it appears to mean that our powers of description are dependent on our ability to physically instantiate large descriptions of numbers, rather than merely to precisely describe such descriptions. This means that in practice philosophical largest number contests are usually won by axiomatic strategies. Of course, an axiom system is also the sort of thing that it is possible to give philosophical arguments for and against. This leads to the following family of contests:

\begin{lncs}
\label{lnc5}
Let $X$ be a school of thought in the philosophy of mathematics, and let $\lang$ be a language that adherents of $X$ believe is precise. Then the \emph{strongest axiom contest for $\lang$ run by adherents of $X$} is the contest in which an entry consists of an axiomatization of $\lang$ (given either informally or as a list of axioms and inference rules) together with a philosophical argument defending the premise that the axiom system is valid. The entry is valid if the adherents of $X$ agree with the argument, and the winner is whichever entry the adherents of $X$ believe describes the most powerful axiom system.
\end{lncs}

Again, \ego\ will describe possible strategies for various strongest axiom contests over the next few sections. Note that the largest number and strongest axiom contests will have to be treated as different contests at first, even though \ego\ will argue that usually the largest number contest reduces to the strongest axiom contest.

Without any information about what school of thought is running the contest, it is not possible to be very precise about what constitutes a good strategy for either of these types of contests. Nevertheless, it turns out that there is a broad \emph{class} of strategies that tend to be good strategies for the first type of contest. Namely, one type of possible entry to a philosophical largest number contest is a description of a (hopefully) precise language $\lang$ for naming numbers together with a term of $\lang$ which is supposed to refer to a number. Once the language $\lang$ is fixed, then the philosophical largest number contest reduces to the largest number contest of $\lang$, a precise largest number contest. However, the appropriate strategic response to this depends on what kind of language $\lang$ is. In Appendix \ref{appendixPT} \ego\ will introduce a class of languages called \emph{proposition/term languages}, or \emph{PT languages}, which have a regular structure designed to be compatible with the ideas in this paper. This structure is somewhat different from, and more general than, the usual structure of the languages considered in mathematical logic. It appears to more closely resemble the structure of natural language.

\subsection{Classical languages}
For the purposes of this paper, a \emph{classical} language is a proposition/term language that is assumed to satisfy\Footnote{Here, \ego\ use the phrase `assumed to satisfy' to mark assumptions made by the users of a language without passing judgement on whether those assumptions are accurate or not. In general, those who disagree with the implicit premises of a language will not be considered to be valid users of the language, though they can still talk about the language.} the \emph{law of excluded middle}, which essentially claims that every \statement\ of the language is either true or false (more precisely, every \statement\ whose terms all have referents). We will call languages that in fact satisfy the law of excluded middle \emph{definite languages}; thus, a classical language is one which speakers of the language treat as definite. The notion of definiteness is closely related to the concept of preciseness, but in what follows \ego\ will argue that it is important to treat them as separate notions, by describing indefinite languages that are in certain respects just as precise as classical definite ones, but which nevertheless cannot be understood sensibly as definite languages.

If $\lang$ is a classical language, then by an appropriate generalization of Theorem \ref{theoremrecstrat} (see Theorem \ref{theoremrecstratB}), it has a brute force strategy. Now if $\lang$ is part of an entry to a philosophical largest number contest, then the obvious way to complete the entry is by using the brute force strategy of $\lang$. The following theorem shows that this \emph{linguistic strategy} has the property that it increases in strength as the expressive power of the language increases:
\begin{theorem}[Cf. Theorem \ref{theoremlingstratB}]
\label{theoremlingstrat}
Let $L_1$ and $L_2$ be two precise proposition/term languages, at least one of which is classical, and suppose that $L_1$ can reproduce the metalanguage $L_2 + 1$, i.e. that the metalanguage $L_2 + 1$ can be translated into $L_1$. Then for any $n$, the phrase `the largest number that can be represented in $L_2$ using at most $n$ symbols' can be translated into $L_1$ as a term of length $O(\log(n))$ which succeeds at naming a number. Thus, in any philosophical largest number contest in which both strategies are considered valid, the linguistic strategy of $L_1$ will win against the linguistic strategy of $L_2$.
\end{theorem}

For example, it follows from well-known facts\Footnote{Specifically, the fact that set theory asserts that there is a set of all numbers, together with standard methods for formalizing in \ST\ the notions of truth and reference for a language whose domain of quantification is a fixed set.} that set theory \ST\ can reproduce the metalanguage of number theory \NT, so Theorem \ref{theoremlingstrat} shows that the brute force strategy of \ST\ wins against the brute force strategy of \NT. As another example, one way to phrase Tarski's theorem on the undefinability of truth (see \cite{Tarski} and Theorem \ref{theoremtarski}) is to say that no classical language is capable of reproducing its own metalanguage, and this is reflected in Theorem \ref{theoremlingstrat} in the fact that no language can describe a larger number than the largest number it describes.

\subsection{Selfmeta languages}
Despite Tarski's theorem, it turns out to be possible to create proposition/term languages that are capable of reproducing their own metalanguages, as long as they are not required to be classical. We will call a proposition/term language which is capable of reproducing its own metalanguage \emph{selfmeta}. The most well-known examples of selfmeta languages are \emph{Kripke extensions} of classical languages (see Section \ref{sectionkripke}), which are essentially selfmeta ``by definition'' in the sense that they have primitive notation for expressing their own concept of truth. However, we will see in the next two sections that there are other natural examples of selfmeta languages.

This leads to the question: what is the best strategy for the largest number contest of a selfmeta language (in the sense of Definition \ref{lnc1})? It turns out to be an axiomatic strategy, due to the following theorem:

\begin{theorem}[Cf. Theorem \ref{theoremselfmetaB}]
\label{theoremselfmeta}
Let $\lang_*$ be a precise selfmeta language, and let $A$ be a valid axiomatization of a precise proposition/term language $\lang$ with at least as much expressive power as $\lang_*$ ($L = \lang_*$ is allowed). Then for each $n$, the phrase `the largest number named by any term $t$ such that $A$ proves that $t$ succeeds at naming a number using reasoning encoded in less than $n$ symbols' can be translated into $\lang_*$ as a term of length $O(\log(n))$ which succeeds at naming a number. This term, which we will denote by $\double{\psi_{\lang_*,A}}{10^{100}}$, can be written explicitly. The proof of this theorem (for fixed $\lang_*,L,A$) can be formalized in the axiom system $A+1$.
\end{theorem}

For any axiom system $(L,A)$, we define the \emph{axiomatic strategy} of $A$ for the largest number contest of $\lang_*$ to be to submit the term $\double{\psi_{\lang_*,A}}{10^{100}}$ given in the above theorem. A player using the axiomatic strategy of $A$ strategy will win against any player submitting an entry whose validity can be proven in $A$, and by Theorem \ref{theoremselfmeta}, the validity of the axiomatic strategy  of $A$ can be proven in $A+1$. This leads to the following corollary:

\begin{corollary}[Cf. Corollary \ref{corollaryselfmetaB}]
\label{corollaryselfmeta}
Let $\lang_*$ be a selfmeta precise language, and let $(L_1,A_1)$ and $(L_2,A_2)$ be valid axiom systems with at least as much expressive power as $\lang_*$. Suppose in addition that the axiom system $(L_1,A_1)$ is at least as powerful as the metasystem $(L_2+1,A_2+1)$. Then in the largest number contest of $\lang_*$, the axiomatic strategy of $A_1$ wins against the axiomatic strategy of $A_2$.
\end{corollary}

Note that Theorem \ref{theoremselfmeta} shows that a philosophical largest number contest can be ``hijacked'' if the judges of the contest are lax enough. Namely, the adherents of a school of thought $X$ may decide that they believe that a certain selfmeta language $\lang_*$ is precise, and then they may decide to declare that an entry to their contest is valid if and only if it is a term of $\lang_*$ that succeeds at naming a number. If the judges do not impose an additional requirement that the entry must be accompanied by a semantic proof that the entry names a number, then the contest can be won by adherents of other schools of thought that admit more expressive languages and correspondingly more powerful axiom systems. If these other philosophers' languages are valid, or at least valid enough that their corresponding axiomatic strategies are valid, then they will win the contest.

It is natural for the adherents of $X$ to prevent this ``hijacking'' by stipulating that entries to their largest number contest must be accompanied by semantic proofs that they name numbers. If they impose this requirement, then their largest number contest reduces to the strongest axiom contest of the selfmeta language $\lang_*$ run by adherents of $X$. In particular, this method of judging entries does not trivialize the largest number contest, so philosophers may find it acceptable. In fact, \ego\ will argue that any coherent position on the philosophy of mathematics can eventually be formalized into a stable selfmeta language, causing its largest number contest to reduce to the strongest axiom contest of that language:

\begin{thesis}
\label{thesis2}
Any coherent position on the philosophy of mathematics can be formalized into a ``maximal'' selfmeta language which adherents of the philosophy should accept as precise.
\end{thesis}

By contrast, the largest \emph{axiom} contest can never stabilize in this way. There are two reasons for this. First of all, anyone who accepts the validity of an axiom system $A$ starts to use reasoning that cannot be formalized in $A$, but can be reasonably formalized in the metasystem $A+1$. Thus it is natural to ask \him\ whether \he\ accepts the validity of $A+1$. If not, then it becomes difficult for \him\ to justify \his\ own reasoning as valid, and if so, then the axiomatic strategy of $A+1$ beats the axiomatic strategy of $A$ and thus the contest has not stabilized.\Footnote{This is analogous to the fact that any strategy in a largest number contest can be beaten by a modification of itself which says to add 1 to the number described by the original strategy, if such a modification is valid.} Secondly, the largest number \emph{proof} contest of a selfmeta language with respect to any fixed axiomatization has a ``brute force'' strategy, so it will be avoided by philosophers who believe in ``mind over matter'':

\begin{lncs}
The \emph{largest number proof contest} of an axiomatization $A$ of a precise proposition/term language $\lang$ is the contest in which a valid entry is a pair of the form $(\psi,\rho)$, where $\psi$ is a term in $\lang$ and $\rho$ is a proof in $A$ that $\psi$ succeeds at naming a number, and the score of the entry $(\psi,\rho)$ is the number named by $\psi$.
\end{lncs}

\begin{theorem}
\label{theoremrecstrat3}
Let $A$ be an axiomatization of a selfmeta language $\lang = \lang_*$. Then for each $n$, the $n$th case of Theorem \ref{theoremselfmeta} can be proven in $A$ using reasoning encoded in at most $O(n)$ symbols. The proof in $A$, which we will denote by $\double{\rho_{L,A}}{n}$, can be given explicitly.
\end{theorem}

The brute force strategy for the largest number proof contest of $(L,A)$ can now be given as follows: First compute the largest number $n$ such that $\double{\rho_{L,A}}{n}$ can be submitted as an addendum to the entry $\double{\psi_{L,A}}{n}$ (without exceeding physical or imposed limits), and then submit the pair $(\double{\psi_{L,A}}{n},\double{\rho_{L,A}}{n})$ as an entry. As usual, this strategy is nearly optimal in the sense that no other valid strategy can beat it too badly.

Note that if $\lang$ is a selfmeta language and $A$ is an axiomatization of $\lang$, then $A+1$ can be viewed either as an axiomatization of $\lang+1$ or an axiomatization of $\lang$. Since $\lang$ is capable of reproducing $\lang+1$, both languages have equal expressive power and thus axiomatizations of either one can be viewed as axiomatizations of the other.\\

\subsection{Plan for the next few sections}
\label{subsectionplan}
In what follows, \ego\ will attempt to describe entries to the largest number and/or strongest axiom contests corresponding to various schools of thought in the philosophy of mathematics. For clarity \ego\ introduce the following terms:
\begin{itemize}
\item \emph{Set realism} is the view that it is possible to talk precisely about sets, understood in the classical sense according to which some sets are undescribable even in principle.
\item \emph{Descriptionalism} is the opposite view, that all mathematical objects should be describable at least in principle, even if it would take an extremely large (but finite) amount of time or resources to describe them.
\item \emph{Number realism} is the view that it is possible to talk precisely about natural numbers, understood in the classical sense according to which large numbers are inherently no different from small ones.
\item \emph{Ultrafinitism} is the opposite view, that the classical concept of a ``natural number'' is incoherent from a philosophical point of view.
\end{itemize}
Obviously, these labels are very broad, and it is possible to for schools of thought to get much more specific about their beliefs. In particular, the category of ``descriptionalist number realism'' includes the schools of thought traditionally known as ``constructivism'', ``intuitionism'', and ``predicativism''. However, \ego\ will find it convenient to avoid talking too much about about the relationship between the philosophical stances that one might take in response to a largest number contest and pre-existing philosophical schools, though \ego\ will mention similarities where they exist.

\begin{remark*}
The label `realism' should not be taken to indicate that adherents of either set realism or number realism are required to subscribe to \emph{philosophical realism}, the view that words and phrases like `the number 5', `the set of odd numbers', and `greenness' refer to ``real things'' that ``exist'' ``independently of the mind''.\Footnote{The scare quotes may allow the reader to guess what \my\ view on the subject is: \ego\ am not sure what it would mean to say that a mathematical entity ``exists'' or ``does not exist''. Certainly mathematical entities are not typical examples of the category of ``things that exist'', but this appears to be a very loosely defined category (even more loose than the category of things that can be precisely talked about), and is there any point in trying to define precise boundaries for it? In the absence of any motivating question, it appears that the answer is no.} A mathematical realist may or may not believe that the entities he talks about ``actually exist''; he only needs to believe that it is possible to talk about them precisely.

Out of the two types of mathematical realism, set realism is more closely connected with philosophical realism, since arguments in favor of set realism often presuppose philosophical realism. On the other hand, there is a tension between the two types of realism in that philosophical realism asserts that e.g. ``greenness'' exists \emph{as greenness}, while set realism asserts that it exists \emph{as the set of all things that are green}.\Footnote{This way of putting things is in some ways an oversimplification, since a set realist may deny that the notion of greenness is well-defined enough for there to be a set of all things that are green. This observation highlights the differences between the set-realist and philosophical-realist perspectives.} This tension is highlighted by the fact that set realists believe that some sets are undefinable and undescribable, whereas it is not clear that a ``property'' or ``universal'' in the sense of philosophical realism could ever be in-principle undescribable. These kinds of considerations may lead a philosophical realist to reject set realism in favor of descriptionalism.
\end{remark*}

Before delving into the details of how \ego\ think set realists and number realists might be able to judge their largest number and strongest axiom contests in a philosophically defensible way, \ego\ will say a few words regarding some general objections to the project \ego\ am undertaking:\\

{\it The pragmatist/formalist/ultrafinitist objection.}
Some mathematicians may object that it does not make sense to talk about the truth or falsity of mathematical \statements\ unless we specify an axiom system according to which alleged proofs or disproofs of these \statements\ are to be judged. There are a few natural objections to this point of view. First of all, the question arises of whether the truth that can be talked about once we have specified an axiom system is conceptually independent of provability in that axiom system, or whether to say that a \statement\ is true simply \emph{means} that it can be proven in the axiom system. If it is conceptually independent, then it is not clear why it cannot be talked about independently of the axiom system. On the other hand, if truth is nothing more than syntactic provability (a view known as \emph{formalism}), then it is not immediately clear what it means to say that a \statement\ is syntactically provable, since the claim that a given \statement\ is syntactically provable is usually considered to be itself a mathematical statement\ (asserting the existence of a syntactic proof), and so it is not immediately clear how such a statement\ could be meaningful unless there is some theory according to which at least some mathematical statements\ are meaningful independent of their provability in any axiom system. However, the notion of syntactic provability can be instead interpreted in the social sense: as the ability of mathematicians to construct a syntactic proof. (In practice, this ability is almost completely hypothetical since mathematicians rely nearly exclusively on informal descriptions of syntactic proofs. However, there appears to be no reason that this should constitute a philosophical problem for formalism.) This raises the question of what constitutes a syntactic proof, and it can again be understood in a social sense: as the question of a proof-checker ``following convention'' would (hypothetically) find an ``error'' or ``non sequitur'' in the alleged proof.

The objection to this view is that our intuition suggests that the question of whether an alleged syntactic proof is in fact a syntactic proof seems to be just as precise as the simpler questions of which it is composed, such as the question of whether or not two given symbols in the alleged syntactic proof are the same symbol, and moreover these simpler questions seem to be perfectly precise. If so, then this principle of ``conservation of precision under finite quantification'' is enough to get us a very rudimentary version of number realism along with a corresponding largest number contest (see Section \ref{subsectionGNR}). However, one may object that the new question is not \emph{quite} as precise as the old one, since if we create machines for checking syntactic proofs, then the probability that two machines will give different answers gets larger as the alleged proof gets longer. Eventually the probability will be large enough that the question of what the ``right'' answer is will no longer be precise (or so the argument goes). The obvious objection to this is that we ordinarily think that the reason that two machines usually give the same answer is \emph{because} there is a ``right'' answer for both of them to give, not the other way around. It is not obvious how to explain this fact  other than by thinking of the machine calculations as approximations of an abstract calculation which remains precise regardless of physical implementability.

If even the principle of conservation of precision under finite quantification is rejected, then there is not much sense in talking about largest number contests, because on such a view it is not clear that the concept of a ``number'' is coherent enough to be worth talking about. Namely, one can talk about various ways of storing information that are classically thought of as ``representing'' numbers, and the stored information can itself be called a ``number'', but there appears to be no principled way to distinguish between valid and invalid representations of numbers, since every ``number'' uncontroversially represents itself but it is difficult to argue that it represents any other kind of number. For example, binary numbers cannot in general be thought of as ``representing'' numbers stored as tally marks, since it is easy (and in fact common practice) to store binary numbers so large that they cannot feasibly be implemented as tally marks. This ``radically relativist'' view about numbers is known as \emph{actualism} or \emph{ultrafinitism}, and it is hard to argue against, though it goes against many people's intuitions.\Footnote{In particular, ultrafinitism denies the intuitively plausible claim that the validity of the standard (feasible) algorithm for comparing the sizes of binary numbers can be thought of in terms of a second (unfeasible) algorithm that converts both numbers to a unary representation (i.e. tally marks) and then compares their raw sizes. Namely, it is intuitively plausible that \emph{what we mean} by saying that the standard algorithm is valid is that it gives the same answer as the second algorithm would. But ultrafinitism denies the coherence of talking about what the second algorithm would do, once the numbers involved are large enough.} It also seems to be difficult to explain the success of mathematical methods using ultrafinitism, though perhaps it is not impossible.\\

{\it The imprecision objection.} Another common objection, perhaps related to the formalist/pragmatist one, is that the question ``is this entry a precise description of a number?'' is not itself a precise question. This is true, but it misses the point that attempts to convince someone that a certain description is precise are really attempts to expand their worldview by introducing new intuitive concepts. Of course there is no precise way of characterizing when we have successfully expanded our worldview (rather than introducing erroneous thinking), but anyone who is not a pure formalist has already expanded their worldview by allowing at least some intuitively intelligible precise concepts, so the idea of further expansion should not be viewed with so much suspicion.\Footnote{One can use the imprecision objection to justify not taking too seriously the question of which philosophical school of thought is ``correct'', on the grounds that this question is not precise, and \ego\ mostly endorse this position.}

\section{Classical vs selfmeta languages}
\label{sectionkripke}

One of the most fundamental questions in the philosophy of language is: what does it mean for a statement\ to be \emph{true}? In general the answer to this question may depend on many things, such as the context in which the statement\ was made.\Footnote{For a simple example, whether or not a statement\ of the form `John did X' is true depends on which ``John'' the speaker is referring to.} 
Additionally, some statements\ are so vague that we would be uncomfortable with saying that the question of whether or not they are true is a precise one. However, the whole point of creating a precise language is that it should be possible to give a precise, context-independent answer to the question of what it would mean for any given \statement\ in the language to be true. Thus, any coherent philosophy should be able to explain precisely what it would mean for any given \statement\ in a language considered precise to be true.

The theory of mathematical logic contains a partial answer to this question in the case of classical languages. Specifically, classical model theory attempts to make precise the notion of truth with respect to a classical language by defining it recursively. For example, a \statement\ of the form `$\forall x \; \phi(x)$' (short for ``for all $x$ [in the category of objects under consideration], $\phi(x)$ is true'') is defined to be true of a given model $M$ if and only if for every object $x$ in $M$, $\phi(x)$ is true of $M$.\Footnote{Strictly speaking, we should say that `$\forall x \; \phi(x)$' is true of $M$ if for every object $a\in M$, $\phi(x)$ is true of $M$ relative to the variable assignment $\lv x=a\rv$.} This definition is \emph{recursive} in the sense that a concept (``truth'') is defined in terms of itself.

The question now becomes: how can a recursive definition be considered to be precise? It certainly gives some information about the concept being defined, namely a way in which it relates to itself, but there is conceivably some imprecision remaining, since the concept has not been cashed out in terms of fundamental precise concepts. We can make recursive definitions more precise by instituting the convention that they should only be considered meaningful in cases where the meaning ultimately derives from concepts outside the recursion. To formalize this convention, we introduce the notion of a ``category generated by production rules'', which is perhaps best illustrated by example. Namely, suppose we have a two-letter alphabet $E$ consisting of the letters `a' and `b'. Then the notion of a \emph{string in $E$} can be defined by the following production rules: the letters `a' and `b' both constitute strings in $E$, and whenever we have a string in $E$, then concatenating it with either `a' or `b' forms another string. These two rules generate the category; any object that cannot be eventually created by these production rules is not a string in $E$.

The concept of production rules can be used to define the concept of a recursive definition as follows: each case of a recursive definition corresponds to a different production rule for the category of ``facts about the recursively defined concept''. For example, the recursive definition of truth considered above could be described in terms of production rules for the category of true \statements\ and the category of false \statements: there would be a production rule stipulating that ``if the \statements\ $\phi(x)$ $(x\in M)$ are all members of the category of true \statements, then `$\forall x \; \phi(x)$' should be added to the category of true \statements'', and another production rule stipulating that ``if at least one of the \statements\ $\phi(x)$ $(x\in M)$ is a member of the category of false \statements, then `$\forall x \; \phi(x)$' should be added to the category of false \statements''.

Does the concept of production rules need to be defined, or is it an irreducible concept? One way we can try to define it is by using the notion of \emph{indexical recursion}, where a sequence of objects $x_0, x_1, x_2,\ldots$ is defined by specifying how to construct each object $x_k$ in terms of the previous objects $x_0,\ldots,x_{k-1}$. Indexical recursion may be viewed as less problematic than general recursion for two reasons. First, the structure of indexically recursive definitions makes it clear that the definition is a reductive one, reducing more complicated concepts to simpler ones as measured by the indices of the objects being considered. Second, with indexically recursive definitions it is clear that the meaning always ultimately derives from outside the recursion, so there is no need for a convention handling the case where this does not happen.

We could attempt to define the notion of a category generated by production rules in terms of indexical recursion as follows: for each number $k$ there is the concept of the ``category generated by stage $k$'', denoted $C_k$. It is defined by indexical recursion as the set of all things that can be created using a single instance of the production rules from elements of $C_j$, $j < k$. We could then define the final generated category to be the union $C = \bigcup_k C_k$. However, this definition does not capture the intuitive notion of a category generated by production rules (at least in the sense we mean here), because the category $C$ is not necessarily closed under the production rules, if there are ``infinitary'' production rules such as the rule considered above for adding universally quantified \statements\ to the collection of true \statements. Thus, it is conceivably possible to add more objects to $C$ using the same production rules, and what we really want to consider is the category that results from continuing to do this indefinitely.

Thus, it is problematic to reduce the notion of a category generated by production rules to the notion of indexical recursion. However, a more mathematically sophisticated variant on this idea based on transfinite induction is arguably more successful (we leave the details to the reader). Nevertheless, for the remainder of this section we will treat production rules as an irreducible concept.

It is worth mentioning another well-known and popular attempt to reduce the concept of production rules to other fundamental concepts: the ``intersective definition'' due to Frege. They consider the scenario where the objects generated by a list of production rules $R$ are known in advance to lie in a given set $S$, and they define the set generated by the list $R$ to be the intersection of all subsets of $S$ that are closed under the production rules of $R$.\Footnote{A set is \emph{closed} under a list of production rules if it is not possible to add new members to the set by applying the production rules.} However, in \my\ view this reduction is neither conceptually valid nor appropriate for the present context, for several reasons summarized in Appendix \ref{appendixintersectivedefinition}, and consequently \ego\ will ignore it in the subsequent discussion.

To summarize, using the concept of recursion, possibly supplemented by the concept of production rules, it is possible to precisely define the notion of truth with respect to a classical language such as \NT\ or \ST, assuming that the relevant category (of numbers or sets) is precise enough to be quantified over. This fact is surely a relief to any philosopher who wishes to maintain that these languages are precise. But there is no need to stop there. The argument from recursion can be modified to handle the notion of truth not only for classical languages, but also for certain selfmeta extensions which we now define:

\begin{definition}
\label{definitionkripke}
The \emph{Kripke extension} of a classical language $\lang$, which we will denote by $\K(L)$, is simply the language $\lang$ augmented with a primitive ability to write phrases of the form `$\phi$ is a true \statement\ of $\K(L)$', where $\phi$ is a term of $\K(L)$, intended to denote a string.
\end{definition}
Let us be clear that the Kripke extension of $\lang$ is very different from the metalanguage of $\lang$. The metalanguage of $\lang$ is denoted $\lang+1$ and has the primitive ability to write phrases of the form `$\phi$ is a true \statement\ of $\lang$', where $\phi$ is a term of $\lang+1$. However, $\lang+1$ does not have the ability to write phrases of the form `$\phi$ is a true \statement\ of $\lang+1$'. The metalanguage of $\lang$ is a classical language, whereas the Kripke extension of $\lang$ is a selfmeta language.

The crucial question about Kripke extensions is: Is the Kripke extension of a precise classical language itself a precise language? Equivalently, is there a precise characterization of what it means for any given \statement\ of a Kripke extension to be true?

Fortunately this question has a fairly simple answer: \emph{If the notion of truth for the original classical language can be defined precisely in terms of production rules, then so can the notion of truth for the Kripke extension}. The reason for this is that one simply has to add the following rules to the list of production rules:
\begin{itemize}
\item ``If $\phi$ is in the category of true \statements, then add `$\phi$ is true' to the category of true \statements, and add `$\phi$ is false' to the category of false \statements.''
\item ``If $\phi$ is in the category of false \statements, then add `$\phi$ is false' to the category of true \statements, and add `$\phi$ is true' to the category of false \statements.''
\end{itemize}
There is nothing unusual about these production rules, in that they have the same sort of structure as the production rules used in the recursive definitions of truth for \NT\ and \ST, so adding them does not appear to lead to any further philosophical problems.

Thus, it appears that if $L$ is a classical language which can be viewed as precise according to the standard picture of the meaning of languages provided by mathematical logic, then the same is true for the Kripke extension $\K(L)$, which is a selfmeta language. It is natural to ask whether the same principle works in reverse, namely whether every precise selfmeta language has a precise classical extension.

In favor of an affirmative answer, one could argue that if $L$ is a precise selfmeta language, then the fact that $L$ is precise already means that any sentence of $L$ either has a truth-value or does not. If so, then we could define the classicization $\C(L)$ to be \NT\ augmented with the ability to talk about whether sentences in $L$ are true, false, or indeterminate. (Note that this presupposes that \NT\ is precise.) It would seem that $\C(L)$ would be a classical language with at least as much expressive power as $L$.

To get a clear picture of the challenges this view faces, it is helpful to analyze why a Kripke extension $\K(L)$ is not already capable of reproducing its classicization $\C(\K(L))$. After all, the expressive power of $\C(\K(L))$ comes from the fact that it can express the claim that a given sentence $\phi$ of $\K(L)$ is true, false, or indeterminate, and $\K(L)$ can already express the claim that $\phi$ is true or false. So the advantage of $\C(\K(L))$ is that it can express the claim that $\phi$ is indeterminate. 

But as a matter of fact $\K(L)$ can already express this claim in a roundabout way: as `$\phi$ is neither true nor false'. However, if we think of this as a sentence of $\K(L)$, then it can never be true. According to the definition of the semantics of $\K(L)$ via production rules, `$\phi$ is neither true nor false' could only be true if `$\phi$ is true' and `$\phi$ is false' are both false, which in turn could only be true if $\phi$ was both false and true, which is impossible.

The paradox here is that while no \statement\ of $\K(L)$ can be truthfully asserted to be indeterminate within $\K(L)$, nevertheless there are \statements\ of $\K(L)$ that are clearly neither true nor false. The standard example is the liar's paradoxical sentence `This sentence is false'. This sentence can be translated into the language \K(\NT) (the Kripke extension of classical number theory) using the standard method of quining. To be precise, there exists a \statement\ $\phi$ of \K(\NT) which is provably equivalent to `$\phi$ is false'. Thus if $\phi$ is true, then $\phi$ is false, and if $\phi$ is false, then $\phi$ is true. Since $\phi$ cannot be both true and false, it follows that $\phi$ is neither true nor false. Moreover, it is easy to formalize this argument to show that according to the production rules $\phi$ will never be added to the category of true \statements\ nor to the category of false \statements.

The resolution of the paradox is that the phrase ``not true'' has a slightly different meaning when we use it inside $\K(L)$ and when we use it outside $\K(L)$. Inside $\K(L)$, asserting that a sentence is not true is the same as asserting that it is false, which is the same as asserting that it is in the category of false \statements. But outside $\K(L)$, asserting that a sentence is not true means only that it will never be added to the category of true \statements.

Thus, the classicization $\C(\K(L))$ has a hidden power that $\K(L)$ does not: it can make ``never'' claims, rather than merely claims about what \emph{is} in a given category. Now it is not clear that the former type of claim is anywhere near as precise as the latter. For one thing, if we view the generation of a category as an ongoing process, then claims that something will never get added to a category are claims about the future, while claims that something is already in a given category are claims about the present.

But the more serious objection is that it is not clear how to define the semantics for a language in which one can make assertions of indeterminacy interpreted in a weak sense: if we proceed in the usual way via production rules for the categories of true and false \statements, then it is not clear what sort of production rule should lead to adding a \statement\ of the form `$\phi$ is indeterminate' to the category of true \statements.\Footnote{More precisely, it is not clear what sort of production rule should lead to adding a \statement\ of the form `$\phi$ is not true' to the category of true \statements; the natural candidate leads to the notion of indeterminacy being interpreted in the strong/impossible sense, not the weak sense.}

In conclusion, selfmeta languages appear to be favored over classical languages from a philosophical point of view, in the sense that they make more natural ``stopping points'' for increasing the expressive power of a language. Beyond this, it is natural to ask specifically which selfmeta languages best correspond to which philosophical positions. For this purpose, it is necessary to consider separately the set and number realist points of view.

\section{Number realism}
\label{sectionnumberrealism}

There are many different points of view that a number realist could take, leading to many different selfmeta languages for our Philosophical Largest Number Contest. Below, they are listed in order of increasing expressive power. We show that each language is capable of reproducing the classicization of the previous language, and thus the languages are ordered in terms of increasing power. Each language has various possible axiomatizations and various possible classical sublanguages whose definitions we leave to the reader.

\subsection{Gradualist number realism}
\label{subsectionGNR}
Though there are many ways to analyze the philosophical divisions that exist between number realists, a natural ``litmus test'' is the following \emph{Principle of Limited Omniscience}, which some number realists find to be intuitively obvious, while others deny its validity:
\begin{principle}[Principle of Limited Omniscience]
Let $\phi(0),\phi(1),\ldots$ be a sequence of \statements\ that are precise enough to have a definite truth-value. Then the \statement\ $\Phi = \LQ \forall n \; \phi(n)$' (i.e. `all of the \statements\ $\phi(0),\phi(1),\ldots$ are true') is precise enough to have a definite truth-value.
\end{principle}
The name comes from constructivist philosophy: some philosophers known as ``constructivists'' take the point of view that to claim that a mathematical \statement\ has a definite truth-value is the same as to claim the (hypothetical) ability to discover what the truth-value is.\Footnote{Some constructivists may instead take the weaker point of view that mathematics is best purified by labelling claims of the form `$X$ has a definite truth-value, but I don't know how to find it' as ``non-mathematical''.} On this view, a mathematician who uses the Principle of Limited Omniscience to conclude that the Goldbach conjecture (the claim that every even number $\geq 4$ can be written as the sum of two primes) is either true or false is claiming the ability to figure out (given enough time and effort, but with a precise plan from the start) whether the Goldbach conjecture is true or false. However, no algorithm is known to be able to succeed at this task. Thus, the mathematician appears to be claiming a kind of limited omniscience, hence the name.

An obvious criticism of this view is that in ordinary reasoning, when one claims that a \statement\ has a definite truth-value, one is not ordinarily claiming to know what the truth-value is, even in a speculative way. For example, after flipping a coin and not looking at the outcome, one may reasonably assert that the claim that the coin has landed on heads has a definite truth-value, even though one may not know which it is. It is not clear why mathematical reasoning should be any different.

Nevertheless, there are still good reasons why someone might not accept the Principle of Limited Omniscience. One reason is the principle that a claim cannot be true unless there is/are some thing/things that makes/make it true. The obvious answer to the question of what kind of things could make the claim $\Phi$ true is that the truths of the claims $\phi(0),\phi(1),\ldots$ could make $\Phi$ true. But it is not clear that it makes sense to talk about infinitely many things at the same time, and if it is impossible to do so, then the previous sentence does not make sense and thus $\Phi$ cannot be made true in this way. (There may be other ways to make $\Phi$ true, such as the existence of valid reasoning of one sort or another demonstrating that $\phi(n)$ is true where $n$ is a free variable used in the reasoning, but it is not at all clear that these other ways will be enough to make up for the loss of power from the ``main'' or ``canonical'' way of making $\Phi$ true. See Appendix \ref{appendiximplication} for further discussion.)

There are a couple of different ways to phrase the same point. Recall that in the previous section, \ego\ argued that future-tense claims about ongoing processes are inherently less precise than present-tense claims. Now if we view the assignment of truth-values to the claims $\phi(0),\phi(1),\ldots$ as an ongoing process, such that at any given point in time only finitely many of the claims have been assigned truth-values, then the claim that all of the \statements\ $\phi(0),\phi(1),\ldots$ will eventually be labelled as true is a future-tense claim, and thus we can coherently reject it on the grounds of imprecision.

Alternatively, if we take the point of view that what really matters is our ability to write an algorithm computing the truth-value of any \statement, then we can reject universal quantification on the grounds that our algorithm could never safely return `true' as the truth-value for a universally quantified claim, since at any point in time the algorithm will only have computed finitely many of the truth-values of the \statements\ $\phi(0),\phi(1),\ldots$. Again, this defect could be partially fixed if we allowed the algorithm to return `true' if there is a proof that a universally quantified \statement\ is true, but this solution is inelegant and in any case many true \statements\ cannot be proven.

These arguments point not only towards rejecting the Principle of Limited Omniscience, but towards rejecting unbounded universal quantification over the category of numbers as a precise concept. \ego\ call \emph{gradualist number realism} the view that unbounded universal quantification should be rejected for these reasons, but bounded universal and unbounded existential quantification should be allowed. The corresponding mathematical theory will be called \emph{gradualist number theory} and denoted \GNT.\Footnote{Gradualist number theory may be compared with \emph{Heyting arithmetic}, a theory which allows unbounded universal quantification but requires its axiomatizations to use it in a more restricted way. To me the fundamental problem with Heyting arithmetic is that its semantics are fundamentally unclear due to the way that implication is treated; see Appendix \ref{appendiximplication} for a more sensible treatment of the notion of implication.} The term ``gradualist'' refers to the fact that a gradualist number realist may hold that the creation of the category of numbers is an ongoing process.


\subsection{Logicist number realism}

If a number realist accepts the Principle of Limited Omniscience, as well as the concept of recursion via production rules described in the previous section, then \he\ must also accept the language \K(\NT), the Kripke extension of classical number theory. One option is to go no farther, and we call this view \emph{logicist number realism}.

\subsection{The language of generated categories}
\label{subsectionGC}

One reason to go beyond \K(\NT) is based on analyzing the intuitions behind the Principle of Limited Omniscience. The fundamental idea is that it works because the category of numbers is ``well-defined'': it includes everything that you get if you start with $0$ and repeatedly apply the successor operation. But this intuition can be extended to other categories with precisely defined generating processes, leading to the following generalization of the Principle of Limited Omniscience:

\begin{principle}[Principle of Precise Categories]
\label{principleGC}
Let $C$ be a category generated by production rules, and suppose that we have a method for associating to each member $x$ of $C$ a \statement\ $\phi(x)$ which is precise enough that it is either true or false. Then the \statement\ $\Phi = \LQ\forall x\in C \; \phi(x)$' (``for all $x$ in $C$, $\phi(x)$ is true'') is either true or false.
\end{principle}

Given Principle \ref{principleGC}, it follows from a trivial argument\Footnote{If $\phi$ is a \statement\ of \K(\NT), then for every $\psi$ in the category of true \statements\ of \K(\NT), either $\phi = \psi$ or $\phi \neq \psi$. Thus by Principle \ref{principleGC}, either $\phi = \psi$ for some such $\psi$, or $\phi \neq \psi$ for all such $\psi$. But in the first case $\phi$ is true, while in the second case $\phi$ is not true.} that every \statement\ of \K(\NT) is either true or not. Thus the classicization \C(\K(\NT)) is also a precise language, and in particular Principle \ref{principleGC} implies that \K(\NT) is not maximal. This raises the question: What kind of language is best capable of formalizing the reasoning of a philosopher who accepts Principle \ref{principleGC}?

Well, it is clear that if we accept Principle \ref{principleGC}, then the notion of a category generated by production rules becomes crucially important. Thus, it becomes useful to develop a convention for describing such categories. For simplicity we consider the case where the categories to be described are subsets of the natural numbers, such as the set of all true sentences of \K(\NT) or similarly constructed categoies. Our convention in this case is as follows: if $\Phi(n,C)$ is a formula with one first-order free variable `$n$' and one second-order free variable `$C$' that only appears positively (i.e. `$m\in C$' can appear but not `$m\notin C$'),\Footnote{The condition that `$C$' appears only positively is necessary so that the resulting category $C_\Phi$ is the same regardless of in what order the production rules are applied.} then the category $C_\Phi$ is generated according to the family of production rules: ``If $\Phi(n,C_\Phi)$ is true, then add $n$ to $C_\Phi$''. For example, if $C_\Phi$ is the category of true statements of \NT\ then $\Phi(n,C)$ would be a disjunct of several statements including ``$n=\LQ\forall m \; \phi(m)\RQ$ and $\forall m \; \LQ\phi(m)\RQ\in C$''.

Using this convention, we define the \emph{language of generated categories} \GC, which is \NT\ augmented with the ability to refer to categories of the form $C_\Phi$, where $\Phi(n,C)$ is a formula in \GC. We need to allow $\Phi$ to range over \GC, and not just \NT, in order to make \GC\ a selfmeta language (we leave the proof that it is selfmeta as an exercise\Footnote{The basic idea is to recursively define a map $\phi(x_1,\ldots,x_n) \mapsto C_\Phi$ from formulas with an arbitrary number of free variables into categories, such that $\phi(a_1,\ldots,a_n)$ is true if and only if $\lb a_1,\ldots,a_n\rb \in C_\Phi$, where $\lb \cdot\rb$ is the standard $n$-tupling function on $\N$.}). However, it has the effect that not all categories $C_\Phi$ that can be written in \GC\ are well-defined. If $\Phi$ refers to \NT, then $C_\Phi$ is well-defined, and if $\Phi$ refers only to categories $C_\Psi$ such that $C_\Psi$ is well-defined, then $C_\Phi$ is well-defined. However, there is no way of knowing in advance whether this is the case or not. The strength of an axiomatization of \GC\ could be measured by how many categories $C_\Phi$ it can prove are well-defined.

\begin{example}
The set of true statements in \K(\NT) is a generated category in the sense described here (we leave the details as an exercise to the reader).
\end{example}

\begin{example}
The set of objects in the minimal model of Kripke--Platek set theory, together with true non-quantified statements about this minimal model, is a generated category.
\end{example}

\begin{example}
The set of computable ordinals is a generated category. It follows that the expression $f(n) = BB_{\alpha(n)}(n)$ defined by Scott Aaronson in his MathOverflow post ``Succinctly naming big numbers: ZFC versus Busy-Beaver''\Footnote{\url{https://mathoverflow.net/questions/34710/succinctly-naming-big-numbers-zfc-versus-busy-beaver?rq=1}} is definable in \GC. It also follows that the Feferman--Schutte ordinal $\Gamma_0$ can be expressed in \GC, meaning that \GC\ is to be considered ``impredicative''.\Footnote{The ordinal $\Gamma_0$ is supposed to be the exact limit of predicative reasoning, which means that predicative reasoning will be best described by a classical language, whose largest number contest therefore has a brute force strategy. Similar remarks apply to this case as to the case of classical set theory discussed in the next section.}
\end{example}

\subsection{Transfinite time and constructive set theory}
\label{subsectionDST}

Having reduced the concept of recursion to the concept of production rules, we can ask whether we can reduce the concept of production rules to something more fundamental. The answer becomes clearer if we ask: in what order do objects get added to a generated category? One could answer that this question is nonsensical, in which case the language of generated categories would be the most fundamental language. However, if we think that the question is meaningful, then we are led to a notion of time more general than finite time, i.e. infinite or ordinal time. This is because for an infinitary production rule like ``if $\phi(n)$ is true for every $n$, then $\forall n \; \phi(n)$ is true'', if each $\phi(n)$ is added to the category of true statements at a different point in time, then $\forall n \; \phi(n)$ can only be added after an infinite amount of time has passed. Ordinal arithmetic gives a way of quantifying infinite amounts of time.

So, adding objects to categories according to production rules is a computation that takes place over ordinal time. It is then perhaps natural to allow set theory based on transfinite time, as long as it does not break the descriptionalist principle that every mathematical object should be describable in principle. We define the universe of constructible sets $\LL$ as follows:

\begin{definition}
\label{definitionconstructibleuniverse}
~
\begin{itemize}
\item At the start of time $\LL$ is the empty set $\emptyset$.
\item At any given ordinal time $\alpha$, we denote the current value of $\LL$ by $L_\alpha$. The successor $L_{\alpha+1}$ is defined as
\[
L_{\alpha+1} = L_\alpha \cup \{ \{x\in L_\alpha : \phi^\alpha(x)\} : \phi\in \ST\},
\]
where $\phi^\alpha(x)$ denotes $\phi(x)$ with its predicates restricted to $L_\alpha$ (so e.g. if $\phi(x)$ is ``$\forall y \; y\in x$'', then $\phi^\alpha(x)$ would be ``$\forall y\in L_\alpha \; y\in x$'').
\item The universe continues expanding without limit as long as the principle that every object must be describable can be adhered to.
\end{itemize}
\end{definition}

Note that each set $\{x\in L_\alpha : \phi^\alpha(x)\}$ can be described perfectly in terms of the formula $\phi$ and the ordinal $\alpha$. Thus, the universe will continue to expand for as long as it is possible to precisely describe ordinals.

It does not make sense to universally quantify over the entire universe of constructible sets, because the amount to which it can expand is defined vaguely. However, it makes sense to existentially quantify over it, since asserting the existence of an ordinal is the same as claiming to have a way of identifying the ordinal. Thus, we define the language of gradualist descriptivist set theory \GDST\ in an analogous manner to gradualist number theory \GNT\ (see Appendix \ref{appendixPT} for details), but quantifying over sets rather than over numbers.

The strength of the language \GDST\ depends on exactly how it is interpreted: which ordinals are describable? One answer is ``the ordinals that can be proven to exist in Kripke--Platek set theory'', and then \GDST\ is weaker than the language of generated categories \GC.

However, with a more powerful axiom system, one can prove that \GDST\ is capable of reproducing the classicization of the language of generated categories \C(\GC). For example, the axiom ``there is an ordinal that cannot be referred to by a term of \GDST'' is sufficiently strong.\Footnote{Indeed, if $f$ is any term in \GDST\ that refers to a method of associating to an ordinal a subset of $\N$, then for each $n$ we can consider the term `$\fst\{\alpha : n\in f(\alpha+1)\}$': by hypothesis, the ordinal $\beta$ coming from our axiom is not equal to this ordinal for any $n$. It follows that for all $n$, either $n\notin f(\alpha+1)$ or $n\in f(\beta)$ for some $\beta\leq \alpha$. So if $f$ is monotone increasing, then $f(\alpha) = f(\alpha+1)$. This implies that \GDST\ is capable of reproducing \C(\GC).}

\section{Set realism}
\label{sectionsetrealism}

The ``official'' philosophy of the mathematical community, to the extent that there is one, is \emph{set realism}, the view that it is possible to talk precisely about ``sets'', where the word `set' is understood in at least one of its two classical senses. This philosophy is enshrined in the convention that a properly ``mathematical'' proof is one that can be formalized in the axiom system \ZFC\ (Zermelo--Fraenkel set theory with Choice), which is the standard axiomatization of the language of classical set theory \ST. However, most mathematicians are not very aware of the philosophical arguments for and against set realism, and perhaps the primary reason for the widespread use of \ZFC\ is simply the largeness of its class of syntactically provable \statements\ -- many mathematicians are pragmatists or formalists, and if we are not interested in philosophical truth and we are just manipulating symbols, then we might as well allow the broadest possible rules for symbol manipulation, as long as the system doesn't become uninteresting due to the discovery of an inconsistency.

There are two different classical ways of interpreting the word `set', corresponding two different types of set realism:\\

{\it The naive view.} One point of view is that the intuitive notion of a ``set'' or ``collection'' is already precise enough that the phrase `the category of all sets whose members are sets, members of members are sets, and so on' uniquely identifies a category of objects, the ``universe of (hereditary) sets''. According to this view, \statements\ in \ST\ are thought of as talking about and quantifying over the members of this category.

A key point brought up in favor of this view is the fact that we ordinarily consider two different lists of objects to be equivalent if they agree on the question of which objects are members of the list. For example, the list `Alice and Bob' is considered to be equivalent to the list `Bob and Alice' even though the word order is different. It is then claimed that this is evidence that we believe that there is a pre-existing object (a ``collection'') that the lists `Alice and Bob' and `Bob and Alice' both identify.

A criticism of this view is that the fact that two terms are considered equivalent does not necessarily imply that there is some pre-existing object that we are thinking of them as identifying. In fact, the fact that we think of their equivalence in terms of the members of the lists rather than in terms of the two lists' identifying a single common object suggests that perhaps there is not such an object.

Another criticism of the naive view is the well-known Russell's paradox. Although it is usually conceived of as a mathematical argument yielding a contradiction in a certain axiom system (known as \emph{naive set theory}), its relevance to philosophy is that there is no obvious reason why `the set of all sets that do not contain themselves' is not a valid description of a set under the intuitive notion of a set. Thus the paradox appears to show that the intuitive notion of a set as a pre-existing object (assuming that there is such an intuitive notion) is inconsistent, and in particular is not a precise notion.

Two ways a set realist can respond to this are (a) to claim that their intuition does not suggest that there is a set of all sets that do not contain themselves or (b) to suggest that there should be a ``limiting principle'' used in conjunction with intuition, for example John von Neumann's criterion that ``A [class] is too large [to be a set] [if and only if] it can be put in one-to-one correspondence with the [class] of all sets'' \cite[footnote 6]{Maddy1}. It seems that there is not much to say about (a); regarding (b), it seems problematic that this ``limiting principle'' was not suggested by intuition but merely by the discovery of paradoxes. In fact, the idea that infinite sets and ``classes'' have ``sizes'' that can be compared by checking for the existence of one-to-one correspondences is itself a modern idea that appears to have only a meager basis in intuition.

Another objection is that the distinction between sets and ``classes'' is conceptually incoherent since ``classes'' seem to have all of the properties we would intuitively expect sets to have. One response to this is that the notion of a ``class'' corresponds to the notion of a \emph{logical} set rather than the notion of a classical set; that is, it is simply a description of a criterion for distinguishing between sets abstracted away from the details of that description. Another possible response is to deny that the notion of ``classes'' is legitimate and to use other heuristics for thinking about the ``size'' of sets rather than von Neumann's.

A final objection is that in order to precisely talk about a category of sets, it is first necessary to precisely define what sorts of objects can be members of those sets. The classical way of answering this is to restrict attention to the class of \emph{hereditary} sets, i.e. sets whose members are sets, whose members are sets, all the way down. However, one might respond that this leads to circularity: the concept of a hereditary set is only well-defined to the extent that the category of possible members of a hereditary set, i.e. the category of hereditary sets, is well-defined.\\

{\it The iterative picture.} In part due to the difficulties outlined above, many set realists today prefer a different conception of the universe of sets, which can be described as follows:

\begin{intpict}[Iterative universe of sets, classical conception]
\label{intclassST}
~
\begin{itemize}
\item[1.] The construction of the universe of sets occurs in ``stages''. At each stage, there is a well-defined collection of sets that exist at that stage, and each set that exists at a stage also exists at all later stages. The collection of all sets that exist at stage $\alpha$ is usually denoted $V_\alpha$.
\item[2.] In the first (or 0th) stage, no sets exist, i.e. $V_0$ is defined to be the empty collection.
\item[3.] After we have reached stage $\alpha$ and consequently defined a collection $V_\alpha$, when passing to the next stage $\alpha+1$ we add to $V_\alpha$ all subsets of $V_\alpha$, i.e. $V_{\alpha+1}$ is defined to be the collection of all subsets of $V_\alpha$.
\item[4.] The process is endless in a strong way: after infinitely many stages have passed, there is another stage after that, and so on: no phenomenon (such as ``infinitely many stages having passed'') is enough to make the process stop. (A process which continues after infinitely many stages have passed is called a \emph{transfinite process}.) The collection of all sets that exist at a ``limit stage'' $\alpha$ is simply the union of the collection of sets that have existed before that stage, i.e. $V_\alpha$ is defined to be the collection of sets that are in some collection of the form $V_\beta$, where $\beta$ is a stage preceding $\alpha$.
\end{itemize}
\end{intpict}

This picture depends on two key concepts: the notion of a ``subset'' and the notion of an ``endless transfinite process''. It may be compared with the picture of the constructible universe given in \6\ref{subsectionDST}, in which each of these concepts is modified somewhat (specifically, ``subset'' becomes  ``logical subset'' and ``endless transfinite process'' becomes ``transfinite process that goes on as long as it can remain descriptionalist'').

The notion of a ``subset'' appears to be at least somewhat more coherent than the naive notion of a ``set''. In particular, it is exempt from criticism based on Russell's paradox, on the incoherence of the distinction between sets and classes, or on the necessity to precisely define a category of objects before one can precisely talk about sets that contain objects from this category. However, one may still object that the intuitive notion of a set or subset is usually formed by example from things like lists of objects or descriptions of categories of objects, but according to standard set realism neither of those things forms a prototypical example of a set.\Footnote{For example, a set realist would normally hold that only ``countably many'' sets can be described.} To put it another way, one might define a subset of a category to be a way of choosing some members of the category for inclusion and others for exclusion, but it is not clear that we can talk about ``choosing'' independent of our methods for making choices, and our methods for making precise choices generally involve descriptions or algorithms of some sort.

Regarding the notion of an ``endless transfinite process'', the obvious criticism is that the intuitive notion of a ``process'' clearly excludes the idea that there could be a new stage after infinitely many stages have passed, so that the notion of a ``transfinite process'' already does not make sense. After all, there is no point in time when infinitely many stages will have passed; infinity is only an abstraction. The main responses to this criticism seem to be: firstly, that dealing with transfinite processes becomes intuitive after you deal with them for a while; and secondly, that many parts of the theory of transfinite processes can be encoded in mathematics which is less controversial (see in particular \6\ref{subsectionDST} above). Note that it is possible to accept the notion of a transfinite process while rejecting the notion of an abstract subset, see \6\ref{subsectionDST}.

In addition, in this paper \ego\ will present a further criticism of the idea of an \emph{endless} process: not that such a process is nonsensical, but that \emph{universal quantification over the outputs of such a process is imprecise or nonsensical}. The reason for this is that at any given point in time, the process is only partially complete and so it makes no sense to quantify over all objects that the process will eventually produce, only the ones that it has already produced. It is possible for a set realist to accept this criticism while remaning a set realist, leading to a view \ego\ will call \emph{gradualist set realism}. However, accepting the criticism leads to the conclusion that the language of classical set theory \ST\ (with unbounded quantifiers) does not have precise semantics.

Finally, let us briefly mention one argument used in favor of set realism which does not itself specify a semantics to be used in the interpretation of \ST: the \emph{indispensability argument} which alleges that since modern mathematics is indispensable for modern science, which can be empirically verified, modern mathematics and in particular set theory has been empirically verified. However, there is little reason to think that this is anything more than a historical accident. Namely, the fact that mathematicians make sure that their arguments can be formalized in \ZFC\ is not necessarily the basis of the conceptual validity of their arguments, which are ordinarily conceived through intuitive reasoning. Verifying that arguments can be formalized in \ZFC, though useful, may be nothing more than a double-check. In particular, there seems to be a consensus among mathematical logicians that a great deal of mathematical reasoning can be formalized in languages and axiom systems much weaker than \ST\ and \ZFC. For example, Harvey Friedman's ``grand conjecture'' claims that ``Every theorem published in the Annals of Mathematics [that can be translated into \NT] can be proved in \EFA''.\Footnote{\url{http://cs.nyu.edu/pipermail/fom/1999-April/003014.html}. In support of his conjecture, Friedman claims that ``in every case where a logician has taken the time to learn the proofs, that logician also [has seen] how to prove the theorem in Peano Arithmetic.''} Here \EFA\ denotes ``Elementary Function Arithmetic'', an axiomatization of \NT\ that is much weaker than the standard axiomatization \PA. Simlarly, Solomon Feferman has claimed that  ``surprisingly meager (in the proof-theoretical sense) predicatively justified systems suffice for the direct formalization of almost all, if not all, scientifically applicable mathematics'' \cite[p.443]{Feferman5}, and Geoffrey Hellman has said that ``whether, indeed, there are \emph{any} serious limitations to predicativist mathematics' power to recover scientifically applicable mathematics remains an important open question'' \cite[p.462]{Hellman}. Here, \emph{predicativism} is a type of number realism that allows for fairly powerful languages and axiom systems.\\

{\it Strongest axiom contest for the language of set theory \ST.} The obvious starting point for the strongest axiom contest for \ST\ is its standard axiomatization \ZFC. For an example of a philosophical defense of \ZFC\ from a set realist perspective, see \cite{Maddy1}. Essentially, \ZFC\ formalizes many common intuitions about the ``universe of sets'', as described either by the naive view or in the iterative picture. However, \ZFC\ does not capture all possible intuitions about set theory, and there is no particular reason for set realists to limit themselves to \ZFC\ rather than trying to axiomatize other intuitions. In particular, in the iterative picture the phrase `all subsets' in (3), and the phrase `no phenomenon is enough to make the process stop' in (4), can always generate more intuitions, which we can then attempt to formalize. 

A straightforward example is the \emph{axiom of inaccessible cardinals}. Roughly speaking, this axiom says that there is some stage $\alpha$ we can stop at such that the resulting universe $V_\alpha$ is closed under the rules for producing sets described by the axiom system \ZFC.\Footnote{The phrase `rules for producing sets described by the axiom system \ZFC' is not taken literally, since the axiom schemas of \ZFC\ (separation and replacement) can only be instantiated for definable classes, and the production rules used to define inaccessible cardinals are taken to apply to all classes (i.e. subsets of $V_\alpha$), regardless of whether or not they are definable.} This axiom can be thought of as a formalization of the principle that ``no phenomenon is enough to make the process stop'', since it hypothesizes that the phenomenon of having used all of the production rules of \ZFC\ as much as we possibly can is not enough to make the process stop.

Though there are more \emph{large cardinal axioms} for which it is possible to give a strong intuitive argument or ``intrinsic justification'' using set realism as a starting assumption,\Footnote{Opinion varies on exactly which axioms can count as ``intrinsically justified''. Some say that axioms as strong as the axiom of remarkable cardinals are justified by intuition \cite[Theorem 3.3]{McCallum}, while \ego\ do not really see how even the axiom of Mahlo cardinals can be made plausible by appealing to intuition (the argument given in \cite{McCallum} is not very convincing since the number of possible values of a second-order variable in $V_\kappa$ increases dramatically as $\kappa$ increases, so there is no clear reason why we would expect anything to stabilize over all possible values of such a variable).} eventually we get to axioms which philosophers have attempted to give ``extrinsic justifications'' based on data other than intuition.\Footnote{Though popular, the labels `intrinsic' and `extrinsic' seem wrong to \me. Namely, human intuition is not something that is intrinsic to any language; it is conceptually independent. In other words, to define a language, it is not necessary to define all possible intuitions that one may have about the language.} For example, when attempting to justify an axiom philosophers may cite as evidence ``confirmation by instances, prediction, providing new proofs of old theorems, unifying new results with old, extending patterns begun in weaker theories, providing powerful new ways of solving old problems, filling a gap in a previously conjectured `false, but natural' proof, explanatory power, and intertheoretic connections'' \cite[\6VII]{Maddy2}.\Footnote{It is interesting to note that almost all of the items on this list are the kind of evidence used by \emph{abductive reasoning}, or reasoning from consequences. If we view intuitive arguments as \emph{deductive reasoning}, then it appears that what is missing is \emph{inductive reasoning}, or reasoning based on expectations. While inductive reasoning is often used in mathematics in the form of probabilistic heuristic arguments (such as the standard argument justifying the conjecture that $\sqrt 2$ has equal proportions of all digits in its decimal expansion), it is interesting to ask whether there are any large cardinal axioms that can be made more plausible by inductive reasoning.}
 While such considerations are certainly relevant in any discussion of the epistemology of mathematics, there does seem to be a relatively clear conceptual distinction between intuitively plausible axioms and hypotheses made plausible in other ways, though in any given case it may be unclear where the line lies. Note that large cardinal axioms whose standard justifications are ``extrinsic'' tend to be far more powerful than those whose standard justifications are intuitive. It appears to be an interesting question whether the same will hold for number realist languages.

Large cardinal axioms, i.e. axioms that hypothesize the existence of stages in the iterative picture with certain properties, are conventionally ordered according to their \emph{consistency strength}. An axiom system $A$ is said to be \emph{consistent} if no claim of the form $P\wedge\neg P$ (i.e. ``the claim $P$ is both true and false'', a \emph{contradiction}) is syntactically provable in $A$. The notion of consistency is particularly important in mathematics because it can be formalized in fairly weak languages such as \NT.\Footnote{From listening to some mathematicians, one could get the impression that this formalizability property means that syntactic consistency is a more primitive notion than truth. One of the goals of this paper is to argue the opposite: that languages and their semantics are more conceptually primitive than axiom systems and their corresponding notions of syntactic provability.} Since no \statement\ can be both true and false, every valid axiom system is consistent, but the converse is not true. The axiom system $A_1$ is said to have \emph{greater\Footnote{In mathematics the word `greater' is often used in an inclusive sense, i.e. as shorthand for `greater or equal'.} consistency strength} than the axiom system $A_2$ if the consistency of $A_2$ can be proven in some ``neutral'' axiom system (known as the \emph{metatheory}) using the hypothesis that $A_1$ is consistent. The selection of a metatheory is important because otherwise all consistent axiom systems would be equally consistent.\Footnote{Axiom systems could still be ranked according to how plausible the philosophical arguments for their consistency are, but the idea of such a ranking has a decidedly non-mathematical flavor. Reference to a metatheory is also important because of \my\ view that sentences of the form `if $P$, then $Q$' can only be thought of as claims or \statements\ (rather than merely intelligible utterances) to the extent that modes of reasoning with respect to which it is being claimed that $Q$ can be inferred from $P$ can be understood from context. This view will be explained further in Appendix \ref{appendiximplication}.} There is nothing to rule out choosing one of the systems $A_1$ or $A_2$ as the metatheory, but usually a ``weaker'' axiom system such as \PA\ or even \PRA\ (Primitive Recursive Arithmetic) is chosen. Note that the ordering of axiom systems by consistency strength is essentially the reverse of the ordering by plausibility: an axiom system $A_1$ with greater consistency strength tends to be less plausible than an axiom system $A_2$ with lesser consistency strength, since if both $A_1$ and the ``neutral'' axiom system are valid, then $A_2$ is consistent and therefore plausibly valid; so the plausibility of $A_1$ feeds into the plausibility of $A_2$ but not vice-versa.

However, the ordering of axiom systems by consistency strength is not the appropriate one to make comparisons between their corresponding axiomatic strategies in the Busy Beaver contest and the largest number contests of selfmeta languages, according to Corollaries \ref{corollaryaxiomstrat} and \ref{corollaryselfmeta}. In fact, no axiomatization of \ST\ is stronger than any other according to the criterion specified there, since the language \ST\ is incapable of reproducing its own metalanguage, and thus if $A_1$ and $A_2$ are axiomatizations of \ST, it is not even possible to ask the question of whether $A_1$ is stronger than the metasystem $A_2 + 1$, since these axiom systems are not formulated in the same language. This is a hint that the classical language \ST\ is not a good language for metamathematics for any form of set realism, a point which \ego\ will argue further below. For now, let us note that a partial response based on classical mathematics is that a slightly stronger notion than consistency suffices to compare axiom systems well enough to determine which axiomatic strategies will do better at the Busy Beaver contest, namely the notion of \emph{$\omega$-consistency}. The notion of $\omega$-consistency depends on the notion of an $\omega$-syntactic proof, which is essentially a syntactic proof in which ``infinite but countable'' chains of reasoning are allowed. An axiom system is said to be $\omega$-consistent if no contradiction is $\omega$-syntactically provable. We have the following:
\begin{theorem}
\label{theoremaxiomstrat3}
Let $A_1$ and $A_2$ be two $\omega$-consistent axiomatizations of \ST, and suppose that $A_1$ can prove that $A_2$ is $\omega$-consistent. Then the axiomatic strategy of $A_1$ in the Busy Beaver contest wins against the axiomatic strategy of $A_2$.
\end{theorem}
In \my\ opinion, this theorem is not very natural because the notion of $\omega$-consistency is not conceptually simple (in comparison to the notion of a language that has the ability to talk about another language's semantics), and because it only applies to the Busy Beaver contest and not to the largest number contests of selfmeta languages, for which other concepts would be needed. However, it does allow the comparison of the axiomatic strategies of various large cardinal axioms. \ego\ leave as an exercise to mathematical logicians the question of whether the standard ordering of large cardinal axioms according to consistency strength is the same as the ordering according to the strength of their axiomatic strategies in the Busy Beaver contest as well as in other contests possessing axiomatic strategies.\\

{\it Largest number contest run by set realists.} Even if we are satisfied with the view that in the end, $\omega$-consistency or some other constructed criterion is to be used for judging the strength of axiom systems rather than their raw ability to formalize reasoning about each other, when we consider the largest \emph{number} contest rather than the largest \emph{axiom} contest, things get more difficult for a semi-naive version of set realism. This is because the most obvious way for a semi-naive set realist to make precise \his\ largest number contest is to reduce it to the largest number contest of the language \ST. But this contest has a brute force strategy! To highlight the absurdity, we ask the set realist if \he\ is really going to reject the entry `the number referred to by the term $\double{\phi_\ST}{10^{100}}$, where $\phi_\ST$ is as in Theorem \ref{theoremrecstrat}' to their contest on the grounds that it is not technically a term of \ST. This is true, but surely the entry is almost as good as a term of \ST, since it can be converted into one given enough time and computational power. Moreover, ordinary mathematical discourse is not directly framed in \ST, but the idea is that there is some method, often partly algorithmic, via which this discourse can be translated into \ST. So the idea that an algorithm that can generate a term of \ST\ is almost as good as a term of \ST\ seems to be embedded in the very idea of interpreting standard mathematical reasoning as reasoning in \ST.

Once the set realist admits that the informal description `my entry is the term $\double{\phi_\ST}{10^{100}}$' is a valid entry, things really start to get interesting. If \he\ allows the use of an algorithm for generating terms of \ST, then \he\ has to address the question of what kind of algorithms are allowed. \He\ could of course require the algorithms to belong to a small class like the class of primitive recursive algorithms, but this would be somewhat arbitrary and would also yield a degenerate largest number contest. Similar remarks apply to the class of algorithms that can be proven to terminate in some fixed axiom system like \ZFC. A better option is to require only a semantic proof that the algorithm described in an entry halts, such as a philosophical argument for some large cardinal axiom together with the stipulation that the axiomatic strategy for the Busy Beaver contest should be used. This would at least gets \him\ away from the realm of degenerate largest number contests and into the realm where axiomatic strategies are appropriate.

However, stopping at algorithms seems to indicate that the key thing about them is the fact that they \emph{can be used to compute} strings in \ST, rather than just the fact that they \emph{uniquely identify} strings in \ST. Such an emphasis seems a little inappropriate for set realists, who are usually uninterested in whether objects can be computed as long as they can be proven to exist. Thus, it seems likely that many set realists will take the point of view that any description that uniquely identifies a string in \ST\ is almost as good as the string it identifies, and that talking about ``the number named by the \ST-term named by the \ST-term $t$'' (where $t$ is a concrete \ST-term) is just as precise as talking about ``the number named by the \ST-term $t$''.

These considerations suggest that a set realist who accepts \ST\ should also accept its metalanguage $\ST+1$, which is capable of transcribing the above terms. Applying the same argument, the set realist should also accept $\ST+2$, and so on, all the way up to the Kripke extension $\K(\ST)$. This reinforces our earlier conclusion in Section \ref{sectionkripke} that any philosopher who accepts the language \ST\ should also accept the language \K(\ST).

However, a set realist need not accept \ST\ as a precise language at all. \He\ may instead say that this analysis shows that the idea of allowing arbitrary quantification over the outputs of a so-called ``endless'' process is incoherent. In other words, whenever we start to talk about the objects that have been created by the process up to some point, then our ability to meaningfully talk about these objects depends on the existence of larger objects delineating truths from falsehoods, meaning that the process continues past the collection of objects that we are currently talking about. In this view, which we will call \emph{gradualist set realism}, it is incoherent to claim that ``for all sets $x$ and $y$, there is a set $\{x,y\}$'', since the truth of such a \statement\ would depend on facts about all sets, and at no time do all sets exist, so at no time can something be true about all of them. Instead, the Axiom of Pair should be thought of merely as a description of a process for creating sets from other sets, but not as a claim that would have been precise if we didn't already know that the process exists.

It turns out to be possible use the principles of gradualist set realism to define a language similar to \ST, which we will denote by \GST\ (for \emph{gradualist set theory}), which does not allow unbounded universal quantification; see Appendix \ref{appendixPT} for details. Bounded universal quantification is allowed, because the quantification ranges over the members of a set which already exists, so that it makes sense to talk about all of the members of the set simultaneously. Note that the language \GST\ is selfmeta in the sense of \6\ref{sectionkripke}.

Although the language \GST\ initially appears to be incapable of formalizing classical mathematics (which uses reasoning that is supposed to be formalizable in \ZFC), it can in fact reproduce a language with respect to which \ZFC\ is usually presumed to be valid: the language with the syntax of \ST\ whose semantics is given by interpreting the word `set' to range over all members of $V_\kappa$, where $\kappa$ is a concretely specified inaccessible cardinal (for example, the first inaccessible cardinal). we will denote this language by $\ST_\kappa$.\Footnote{Note that gradualist set realism appears to be incompatible with classical set realism, since the languages $\ST_\kappa$ all have different semantics and so it appears that the language \ST\ does not have any canonical semantics.} Note that \GST\ must have primitive language for talking about large cardinals with a specified property, such as the first inaccessible cardinal, in order for this to work.

Since \GST\ is selfmeta, there are no obvious philosophical arguments showing that someone who accepts that \GST\ has precise semantics should also accept that some stronger language has precise semantics. Namely, the argument that anyone who accepts a language should also accept its metalanguage is irrelevant, because someone who accepts \GST\ has already accepted the metalanguage of \GST, since it is reproducible in \GST. Thus, \GST\ can be viewed as an ``ultimate ontology'', and can thereby used to convert the imprecise ``philosophical largest number contest'' into a precise largest number contest, namely the largest number contest of \GST\ (cf. Contest Definition \ref{lnc1}).

The obvious strategy for this largest number contest is the \emph{linguistic strategy} mentioned just before Theorem \ref{theoremlingstrat}: first describe a large cardinal $\kappa$ using \GST, and then use the brute force strategy on $\ST_\kappa$. In order for this to work we need to make sure that our description of $\kappa$ in \GST\ is in fact a valid description, and for this purpose we need to make some kind of argument. Arguments tend to be best made in some axiom system, which we can expect to be an axiomatization of \GST.

Theorem \ref{theoremselfmeta} shows that the Philosophical Largest Number Contest for gradualist set realists reduces to the Philosophical Strongest Axiom Contest. At this point it will be useful to revisit our analysis of the strongest axiom contest for set realists, this time under the gradualist assumption. Since \GST\ is selfmeta, it is legitimate to compare large cardinal axioms on the basis of whether one axiom system is more powerful than the metasystem of another axiom system. However, there is the issue that not all large cardinal axioms can be translated into \GST, since some are stated using unbounded quantifiers. Though more powerful large cardinal axioms tend to be more likely to use unbounded quantifiers, this relationship is not straightforward and the strongest large cardinal axiom not currently known to be inconsistent with \ZFC, i.e. the rank-into-rank Axiom {\bf I0}, is formulated using only bounded quantifiers.

If $\phi$ is a large cardinal axiom formulated using unbounded quantifiers, then it is still possible to translate into \GST\ the claim that $\phi$ is true in some universe: namely, this claim can be translated as `$\exists \kappa \text{ inaccessible such that $\phi^\kappa$ holds}$', where $\phi^\kappa$ denotes the result of replacing all quantifiers in $\phi$ by quantifiers over the universe $V_\kappa$. We will call this claim the ``existence of a universe satisfying $\phi$''. Thus, in addition to comparing the strength of various large cardinal axioms with bounded quantifiers, one may also compare the strength of the assumption that there exists a universe satisfying a given large cardinal axiom with unbounded quantifiers. Note that if $\phi$ is an axiom formulated in \GST, then the existence of a universe satisfying $\phi$ is a stronger claim than $\phi$; in fact, it is stronger than the metasystem of $\phi$, since \GST\ can talk about the notion of truth within any fixed universe. So to show that an axiom system $A_1$ is stronger than the metasystem of another axiom system $A_2$, it suffices to show that $A_1$ can prove the existence of a universe satisfying $A_2$.

Again, we leave the details of exactly which large cardinal axioms are the strongest according to this criterion as an exercise for the interested mathematician. (We expect that the axiom systems we give below for dealing with truth should suffice to answer this question, so the question can be attacked from a purely formalist perspective.)




\ignore{

\begin{remark}
This way of thinking makes it completely unsurprising that, in the words of Wikipedia, ``There is no generally agreed precise definition of what a large cardinal property is'' -- if there was such a precise definition, it could probably be made into a large cardinal property itself.
\end{remark}

\begin{example}
Consider the following program to output a really big number:\\
/\;\; Let $\alpha = \emptyset$;\\
/\;\; LABEL 1;\\
/\;\; If $|\alpha| < 10^{10}$(\\
/\;\;\;\; Let $\alpha = \text{``Compute $T(\alpha)$, then add $1$''}$;\\
/\;\;\;\; GOTO 1;)\\
/\;\; Else(\\
/\;\;\;\; For $n = 1,\ldots,10^{10} - 10,000$(\\
/\;\;\;\;\;\; For $\beta$ of length $n$(\\
/\;\;\;\;\;\;\;\; Allocate Graham's number of cycles to compute $T(\beta)$, and return the result as $\gamma$;\\
/\;\;\;\;\;\;\;\; If $\gamma = \alpha$(\\
/\;\;\;\;\;\;\;\;\;\; Let $\alpha = \text{``Compute $T(T(\beta))$''}$;\\
/\;\;\;\;\;\;\;\;\;\; GOTO 1;)\\
/\;\;\;\;\;\;\;\; )));\\
/\;\;Return $T(\alpha)$\\

Let's call this program ``Solomonoff induction''. It might do really well or really badly, depending on circumstances -- it's not clear. But the funny thing is that \emph{computations like \eqref{solomonoff} seem to be more or less an attempt to imitate Solomonoff induction ourselves} -- with a bit more cleverness in that we have the idea of ``a trick which worked already will work again'' that Solomonoff induction doesn't have -- nevertheless, this trick idea is not actually always true, so it makes sense that a computer should not implement it.

Also, it seems to make a big difference whether the computer thinks of the solution
\[
\text{Add $10^10$ ones, then add $10^{10} - 100$ ones, then add $10^{10} - 200$ ones, then add...}
\]
or
\[
\text{Add ones until the string length is $10^{10}$, then add ones until the string length is $10^{10}$, then add...}
\]
first, and there doesn't seem to be much opportunity to correct later: suppose that each ``add $X$ ones'' has a different way of expressing $X$, as an arbitrary consequence of how fast the code can be collapsed. E.g. $10^{10}$ can be written just ``$10^{10}$, but other numbers will need explicit decimal or binary expansions. The way we set it up, Solomonoff basically has to predict what the pattern behind these are in order to get anywhere, which is... backwards.

Are we subject to the same limitation somewhere, in a way which we do not see?

Conclusion: Although Solomonoff induction sounds good theoretically, it doesn't capture all parts of human intuition, such as: There are abstract ``numbers'', and we can write them however we want, in particular we can write them in whichever way will make there be a pattern. And: The information needs to be stored in a way that doesn't depend too sensitively on the arbitrary cutoff. Is it possible to solve these problems? It seems sort of like creating an AI. One suggestion would be to tell the computer that it has to use a specific symbol to represent the number $10^{10}$, and that when the program is executed on slightly different values of that symbol, then slightly different length strings result (describing the ``rate of change'' of $\alpha$ in the last few steps before we started the induction step). Does this work? It's a technical question... One big problem is that you may get something like:
\begin{quote}
``Iterate $(x x)$ a thousand times, then iterate $(x x)$ 1200 times, then iterate $(x x)$ 991 times, then apply $(x x)$ $(x x)$ $(x x)$ \ldots (a thousand times) $(x x)$ $xxxxxx\ldots$ (to $10^{10}$)''
\end{quote}
\end{example}

}

\ignore{
\section{Formalizing the entire paper}

Sorts used in this paper:
\begin{itemize}
\item categories
\item members of categories (objects)
\item languages
\item axiom systems
\item algorithms
\item utterances (terms, \statements, predicates, definitions, utterances)
\item numbers
\end{itemize}
Non-mathematical:
\begin{itemize}
\item humans
\item thought processes / reasoning
\item concepts
\end{itemize}
Things that they can do:
\begin{itemize}
\item terms can denote objects
\item humans can say, give, write
\end{itemize}
Things that they can be:
\begin{itemize}
\item \statements\ can be true or false, precise or imprecise, coherent or incoherent
\item languages and utterances can be precise or imprecise
\item categories can be complete or incomplete
\item reasoning can be valid/precise or invalid/imprecise
\end{itemize}
This is not very informative... More useful are the axioms:
\begin{itemize}
\item Let $C$ be a complete category and let $P$ be a property. Then the phrase `all members of $C$ have property $P$' is true if and only if all members of $C$ have property $P$.
\item Let $C$ be a category and let $P$ be a property. Then the phrase `there exists a member of $C$ with property $P$' is true if and only if there exists a member of $C$ with property $P$.
\end{itemize}
}



\newpage
\appendix

\section{Proposition/term languages}
\label{appendixPT}

In the appendices, we will assume that the reader is somewhat familiar with the techniques and terminology of mathematical logic, and is capable of filling in missing details of arguments and definitions when they are routine or obvious.\Footnote{As a convenience we recall the technical distinction between \emph{terms} and \emph{formulas}, which is most easily described in terms of their semantics: once a variable assignment has been fixed, then terms are supposed to have \emph{referents} (which are mathematical objects), and formulas are supposed to have \emph{truth-values} (which are either `true' or `false'), although not all terms have referents and not all formulas have truth-values. Note that in the case of terms and formulas with no free variables, it is not necessary to fix a variable assignment before their referents and truth-values are (in principle) well-defined.} However, the languages we consider in this paper differ in several respects from the languages which are familiar to most mathematical logicians, namely those expressed in the standard (first-order, classical) predicate calculus. We summarize the main differences as follows:
\begin{itemize}
\item We use Kleene's (weak) ``three-valued'' logic rather than the classical two-valued logic: a formula may be either true, or false, or it may be \emph{indeterminate} i.e. neither true nor false. There are three different ways for a formula to be indeterminate, which will be described in more detail below. We will think of the truth-value of a formula as a ``dynamic'' quantity: an indeterminate formula may at any time become true or false, but a true or false formula can never become indeterminate or otherwise change its truth-value. Due to our use of three-valued logic, the axiom systems we consider will not contain the law of excluded middle, but they will contain suitable replacements. Note that despite the word ``three-valued'', in this system formulas have only two possible truth-values (`true' and `false'); the ``third value'' is the label `indeterminate' which is applied to formulas that do not have truth-values, but it is not itself a truth-value. Moreover, there is no ``axiom of assignment'' asserting that every formula either has one of the truth-values or is indeterminate, although evidently no formula can fail to take either truth-value while also failing to be indeterminate (since by definition any formula that fails to take either truth-value is indeterminate).
\item Due to the need for languages with terms referring to very large numbers, we consider languages with term templates more general than those allowed by the standard predicate calculus. For example, the language of number theory will have a rule for constructing terms of the form
\begin{equation}
\label{maxtemplate}
\max\{t(m) : m\leq n, \; \phi(m)\}
\end{equation}
where $t(m)$ is a term with $m$ as a free variable and $\phi(m)$ is a formula with $m$ as a free variable. (Such a term would not be allowed in the standard predicate calculus, in which only quantifiers can bind variables, not term operators, and in which formulas can contain terms but not vice-versa.) The class of term templates we consider is motivated by mathematical practice and by the principle that increasing the expressive power of terms should not increase the expressive power of the language as a whole.
\item Due to the generality of the term templates described above, it may sometimes happen that a term does not have a referent. In this case we consider any formula containing the term to be indeterminate.
\item We do not assume that the class of formulas is closed under the application of unbounded quantifiers: that is, our definition of the class of ``formulas'' will not include the usual rule stating that if $\phi(x)$ is a formula with $x$ as a free variable, then $\LQ\exists x \; \phi(x)$' and $\LQ\forall x \; \phi(x)$' are both formulas. This may take place in an asymmetric way: for example, a language may allow unbounded existential quantifiers and bounded universal quantifiers, but no unbounded universal quantifiers.
\item To handle the possible asymmetry described above, we do not assume that the class of formulas is closed under negation: that is, the usual formula template corresponding to negation may be restricted. Thus, we draw a distinction between formulas to which the negation operation can be validly applied and those to which it cannot be applied; we call these formulas ``negatable'' and ``unnegatable'' respectively.
\item Unnegatable formulas are always considered to be either true or indeterminate, but never false.
\item Like other formulas, atomic formulas may be either negatable or unnegatable. Regardless of whether or not it is negatable, an atomic formula may be indeterminate (under a given variable assignment) even if it does not contain any terms without referents, depending on the intended meaning of the formula.
\item Due to the need to avoid using set theory as a philosophical foundation of our theory of truth and reference,\Footnote{The problem with using set theory as a foundation for a theory of truth is that the set-theoretic definition of truth only makes sense for languages whose domains of discourse are sets, and set theory itself is not such a language. For set realists, largest number contests will tend to revolve around such languages, so a more robust theory of truth is needed.} we will have a somewhat more informal model theory, which we nevertheless hope can still be considered precise.
\item Due to the fact that we are interested in the approximate lengths of proofs, the precise definition of a proof is more important than it usually is in mathematical logic, where only the concept of provability is important. We will use a deduction system similar to natural deduction, but we also introduce a couple of ``shortcuts'' that correspond to standard mathematical practice, whose introduction changes the lengths of proofs without changing the class of provable \statements.
\end{itemize}

Thus, we define the class of languages we consider as follows:\Footnote{We do not claim that all of the languages considered in the body of the paper can be put into the form considered here; indeed, some of them, such as the language \GC\ described in \6\ref{subsectionGC}, cannot be.}

\begin{definition}
A \emph{proposition/term syntax} is a list of templates for creating terms and formulas, as described below. Terms and formulas created using the templates of a proposition/term syntax will be called \emph{valid} or \emph{well-formed} with respect to that syntax.
\begin{itemize}
\item[(1)] Each template of a proposition/term syntax is a string which includes zero or more ``placeholder'' terms, formulas, and variables, with the understanding that the template corresponds to a syntax rule allowing the formation of either a term or a formula by replacing the placeholder terms, formulas, and variables of the template by terms, formulas, and variables of the proposition/term syntax, if necessary. Terms and formulas are then said to be valid with respect to a syntax if they can be constructed using finitely many applications of these syntax rules. For example, all proposition/term syntaxes include the formula templates
\begin{align*}
&\LQ\phi \text{ and }\psi\RQ,&
&\LQ\phi \text{ or }\psi\RQ,
\end{align*}
in which $\phi$ and $\psi$ are placeholder formulas (cf. (2) below). They correspond to syntax rules of the form ``if $\phi$ and $\psi$ are valid formulas, then `$\phi \text{ and }\psi$' is a valid formula'' and ``if $\phi$ and $\psi$ are valid formulas, then `$\phi \text{ or }\psi$' is a valid formula'', respectively.
\item[(2)] There are three types of formula templates:
\begin{itemize}
\item[(2a)] The two {\bf primitive logical operators} are the binary operators `and' and `or', which we sometimes denote by `$\wedge$' and `$\vee$'. Note that `not' and `implies' are \emph{not} considered primitive logical operators.
\item[(2b)] {\bf Atomic formula templates} take the usual form $R(t_1,\ldots,t_n)$, where $R$ is a string called a \emph{primitive relation} and $t_1,\ldots,t_n$ are placeholder terms. If $R$ is a negatable primitive relation, then `$\neg R(t_1,\ldots,t_n)$' is also an atomic formula template, though `$\neg R$' is not considered to be a primitive relation. Here, the background assumption is that to specify a proposition/term syntax, it is necessary to specify exactly which primitive relations of the syntax are negatable.

The binary equality relation `$=$' and the unary existence relation `$\ex$' are considered to be in the list of primitive relations of every proposition/term language. For convenience we will use the infix notations `$t_1=t_2$' and `$t\ex$' in place of the prefix notations `$=(t_1,t_2)$' and `$\ex(t)$', respectively. The unnegatable formula `$t\ex$' (``$t$ exists'' or `` `$t$' successfully denotes'') is considered to be true if the term $t$ has a referent, and to be indeterminate otherwise.

\item[(2c)] There are four possible {\bf quantifier templates}:
\begin{align}
\label{qtemplates}
&\LQ \forall x \bq t \; \phi(x)\RQ,&
&\LQ \exists x \bq t \; \phi(x)\RQ,&
&\LQ \forall x \; \phi(x)\RQ,&
&\LQ \exists x \; \phi(x)\RQ,&
\end{align}
where $t$ is a placeholder term and $\phi(x)$ is a placeholder formula that contains the placeholder variable $x$ as a free variable. The first two templates are said to contain ``bounded quantifiers'' while the last two are said to contain ``unbounded quantifiers''. A proposition/term syntax may contain any subset of these four templates, possibly leading to an asymmetry with respect to de Morgan duality. If a proposition/term syntax contains either of the first two templates, it must also contain an atomic formula template corresponding to a binary relation `$\bq$'. However, the symbol `$\bq$' is not taken to necessarily connote a partial order.
\end{itemize}
\item[(3)] There are four types of term templates:
\begin{itemize}
\item[(3a)] The term templates for {\bf variables}, {\bf constants}, and {\bf primitive functions} are the same as in the usual predicate calculus: variables and constants are terms, and $f(t_1,\ldots,t_n)$ is a term whenever $t_1,\ldots,t_n$ are terms and $f$ is a string known as a \emph{primitive function}.
\item[(3b)] The term template for a {\bf quantifier-like operator (QLO)} takes the form
\[
\Psi(\xx | \tt,\bfphi) = \Psi\big(x_1,\ldots,x_k \big| t_1,\ldots,t_n,\phi_1,\ldots,\phi_m\big)
\]
where $\tt = (t_1,\ldots,t_n)$ and $\bfphi = (\phi_1,\ldots,\phi_m)$ are lists of placeholder terms and formulas, and $\Psi$ is a string (i.e. the ``quantifier-like operator''). The operator $\Psi$ is thought of as ``quantifying'' over $x_1,\ldots,x_k$, so that while $x_1,\ldots,x_k$ may appear free in (the terms and formulas substituted for) $t_1,\ldots,t_n$ and $\phi_1,\ldots,\phi_m$, they are considered to be non-free variables with respect to the formula `$\Psi(\xx | \tt,\bfphi)$'.
\end{itemize}
\item[(4)] Although we usually deal with languages with a single domain of discourse, it is easy to modify the above definition to deal with multiple domains of discourse (i.e. so-called ``many-sorted'' languages). In this case each term template, placeholder term, and placeholder variable must be assigned a domain of discourse, and when using a template one cannot replace a placeholder term or variable by a term constructed using a template corresponding to a different domain of discourse. For the quantifier templates, a different subset of the templates may be considered valid for each domain of discourse.
\end{itemize}
\end{definition}

Although we do not treat negation as a primitive logical operator, we define the negation of a valid formula by recursion:
\begin{itemize}
\item The negation of a negatable atomic formula `$R(t_1,\ldots,t_n)$' is `$\neg R(t_1,\ldots,t_n)$', and vice-versa.
\item The negation of a conjuctive, disjunctive, or quantified formula is determined by the laws of de Morgan duality, for instance
\[
\neg \LQ\exists x \; \phi(x)\RQ = \LQ \forall x \; \neg \phi(x)\RQ,
\]
assuming that the right-hand side exists and is a valid formula. Note that the possible lack of symmetry in quantifier templates may lead to some formulas being unnegatable despite containing no unnegatable atomic formulas.
\end{itemize}
If $\phi$ is a negatable formula then we let
\begin{equation}
\label{definite}
\DD(\phi) \df \LQ \phi \text{ or } \neg \phi\RQ
\end{equation}
and we think of $\DD(\phi)$ as asserting ``$\phi$ has a definite truth-value''.

Although we can be very mathematically rigorous in the definition of the syntax of a proposition/term language, our description of its semantics may appear a little hazy. However, when we get to the business of assigning an axiom system to a proposition/term language, we will be able to be rigorous again.

\begin{definition}
\label{definitionPTlanguage}
A \emph{proposition/term language} is a proposition/term syntax together with an intuitive picture of how to interpret the valid terms and formulas of that syntax. Specifically, this intuitive picture must consist of
\begin{itemize}
\item an idea of what kind of objects the quantifiers and QLOs are supposed to be quantifying over (i.e. the domain of discourse), and of how to interpret the allowed quantifier templates;
\item for each primitive relation, an idea of what it would mean for a given list of objects to satisfy the relation;
\item for each negatable primitive relation, an idea of what it would mean for a given list of objects to fail to satisfy the relation;\Footnote{This may seem pedantic since it is intuitively plausible that if a speaker of a language knows what it would mean for a \statement\ to be true, then \he\ also knows what it would mean for the \statement\ to be false. However, we have argued against this intuitive assumption in Section \ref{sectionkripke}.}
\item an idea of which primitive relations are ``strongly negatable'' in the sense that every list of objects either satisfies or fails to satisfy the relation;
\item for each primitive function, an idea of how it would be applied to a given list of objects to produce another object;
\item for each quantifier-like operator, a \emph{reductive interpretation} of that operator as described below.
\end{itemize}
Here, a reductive interpretation of a quantifier-like operator $\Psi(\xx|\tt,\bfphi)$ is an equivalence of the form
\begin{equation}
\label{Psidef}
\Psi(\xx|\tt,\bfphi) = s
\;\;\Leftrightarrow\;\;
\Theta(\xx|\tt,s,\bfphi),
\end{equation}
where $\Theta(\xx|\tt,s,\bfphi)$ is a schematic formula (that is, a formula template constructed according to the syntax rules of the proposition/term language, but with placeholders terms, formulas, and variables allowed, as well as a primitive negation operation), written without using QLOs. The equivalence \eqref{Psidef} is understood as a conceptual equivalence, meaning that \eqref{Psidef} is either understood as giving the intuitive meaning of $\Psi$, or as following from this intuitive meaning via a straightforward argument. Note however that \eqref{Psidef} cannot be assumed to always give a coherent intuitive meaning to the operator $\Psi$, since for any given $\xx,\tt,\bfphi$ it may happen that there exist two terms $s_1,s_2$ referring to different objects such that the formulas $\Theta(\xx|\tt,s_1,\bfphi)$ and $\Theta(\xx|\tt,s_2,\bfphi)$ both hold, or else that there exist $s_1,s_2$ referring to the same object such that $\Theta(\xx|\tt,s_1,\bfphi)$ holds but $\Theta(\xx|\tt,s_2,\bfphi)$ does not. Nevertheless, in many circumstances the formula \eqref{Psidef} can indeed be understood as giving an unambiguous meaning to $\Psi$.

As an example, the quantifier-like operator `elt' (``the unique element of'') has a reductive interpretation
\begin{equation}
\label{russell}
\elt\{x:\phi(x)\} = s \;\;\Leftrightarrow\;\; \phi(s) \text{ and } \forall x \; (s=x \text{ or }  \neg\phi(x)),
\end{equation}
where $\phi$ is a placeholder formula, understood to be negatable. A bounded version of `elt', with template `$\elt\{x \bq t:\phi(x)\}$', can be treated similarly. In what follows we assume that every proposition/term language in which both of the unbounded (resp. bounded) quantifier templates are allowed also has the unbounded (resp. bounded) version of the `elt' QLO.

We take it as axiomatic that an intuitive picture of the form given above automatically leads to an intuitive picture of how to interpret any valid term or formula of the syntax of a proposition/term language, via the method of recursion. Thus, in any sufficiently precise proposition/term language we can talk about the referent of any term, and the truth-value of any formula, without ambiguity. However, we do not assume that terms always have referents or that formulas always have truth-values. Also note that according to our picture it is impossible for an unnegatable formula to be false, because a proposition/term language is not supposed to include an idea of what it would mean for such a formula to be false, but only an idea of what it would mean for such a formula to be true.

Note that from the syntactic point of view, a proposition/term language still contains two pieces of information not contained in a proposition/term syntax: a list of which primitive relations are ``strongly negatable'', and a list of ``conceptual equivalences'' between QLOs and schematic formulas. These pieces of information will both be important when we define the notion of a ``standard axiomatization''.

Also note that we can define entirely syntactically a class of \emph{strongly negatable} formulas: negatable formulas whose atomic subformulas are all strongly negatable and which contain no subterms constructed via quantifier-like operator templates. According to the intuitive picture of the proposition/term language, such formulas always have a definite truth-value.
\end{definition}

\begin{remark}
For the classical mathematician trained in working in \ZFC, we could make the following definition: a \emph{proposition/term language} is a proposition/term syntax together with a model for that language, i.e. a set representing the domain(s) of discourse, together with appropriate relations and functions on that set corresponding to the interpretation of the primitive relations and functions of the syntax, as well as a list of ``strongly negatable'' primitive relations and a list of pairings between QLOs and schematic formulas. However, this definition does not fully capture what we mean when we talk about ``proposition/term languages''.
\end{remark}

It will perhaps be helpful to give an example:

\begin{definition}
The language of \emph{gradualist number theory (\GNT)} is the proposition/term language in which
\begin{itemize}
\item the domain of discourse is the natural numbers,
\item both bounded and unbounded existential quantifiers are allowed, but only bounded universal quantifiers are allowed,
\item `$=$' and `$\bq$' are binary strongly negatable relations, while `$\ex$' is a unary unnegatable relation,
\item the numerals `$0$' and `$1$' are constants, while addition, multiplication, and exponentiation are binary functions,
\item in addition to the bounded `elt' operator, there is a quantifier-like operator $\LQ\fst$' which inputs a negatable formula $\phi(n)$ with $n$ as a free variable and outputs the expression `$\fst\{n : \phi(n)\}$', which is conceptually defined via the equivalence
\[
\fst\{n : \phi(n)\} = t
\;\;\Leftrightarrow\;\;
\phi(t) \text{ and } \forall n < t \; \neg\phi(n).
\]
(Verbally, $\fst\{n : \phi(n)\}$ is ``the first integer $n$ such that $\phi(n)$ holds''.)
\end{itemize}
Since \GNT\ has unbounded existential quantifiers but not unbounded universal quantifiers, a \statement\ of the form $\LQ\exists n \; \phi(n)$' is unnegatable, and thus can be either true or indeterminate, but never false. We can motivate this convention by an ``algorithmic'' interpretation of the semantics of \GNT. Namely, we can imagine an algorithm intended to compute the truth-value of $\LQ\exists n \; \phi(n)$' that proceeds by computing the truth-values of $\phi(0),\phi(1),\ldots$ in parallel. If it finds an $n$ such that $\phi(n)$ is true, then it returns `true' as the truth-value for `$\exists n \; \phi(n)$', and otherwise it continues the computation indefinitely without ever returning a value. This algorithm makes sense because observing finitely many $n$\,s such that $\phi(n)$ is false would never allow us to directly conclude that `$\exists n \; \phi(n)$' is false, so a brute-force algorithm should never return `false' as the truth-value for `$\exists n \; \phi(n)$'. Note that if none of the \statements\ $\phi(0),\phi(1),\ldots$ are true, then the algorithm will continue to run forever, and thus in the context of \GNT\ the formula `$\exists n \; \phi(n)$' will remain indeterminate forever.\Footnote{From a philosophical point of view, we should warn that the hypothetical ``If none of the \statements\ $\phi(0),\phi(1),\ldots$ are true'' is an outside-view way of looking at things; within the language \GNT, there is no way to talk about all of the \statements\ $\phi(0),\phi(1),\ldots$ simultaneously being true or false. The reason for this is that someone who accepts the validity of \GNT\ (for example a constructive mathematician who accepts it on the basis of the algorithmic interpretation described above) is not committed to accepting the validity of \statements\ of the form `$\forall n \; \phi(n)$'. See Section \ref{sectionnumberrealism}.}
\end{definition}

\begin{definition}
The language of \emph{classical number theory (\NT)} is the same as \GNT, except that unbounded universal quantifiers are allowed.
\end{definition}

Due to its algorithmic interpretation, the language of gradualist number theory should be acceptable to constructivist philosophers, who may be more suspicious of the language of classical number theory, which does not have such an algorithmic interpretation.

The languages \NT\ and \GNT\ are representative of two larger classes of proposition/term languages: the class of \emph{classical languages} and the class of \emph{selfmeta languages}, respectively. We proceed to define these classes and discuss their basic properties:

\subsection{Classical languages}
\label{subappendixclassical}

\begin{definition}
A proposition/term language is \emph{classical} if
\begin{itemize}
\item it allows unbounded quantification (both universal and existential), and
\item all primitive relations except for `$\ex$' are strongly negatable.
\end{itemize}
\end{definition}

Classical proposition/term languages are very similar to the languages of the standard predicate calculus, but differ in that they allow terms that do not necessarily have referents. As a consequence, classical proposition/term languages can have formulas that are neither true nor false (because they contain terms without referents). However, we will now show that every formula of a classical proposition/term language is canonically equivalent to a formula which is (in principle) either true or false. We call this equivalent formulation the \emph{Russell reformulation}, in honor of Bertrand Russell's philosophical analysis of the meaning of the word `the' \cite[Chapter XVI]{Russell}, which corresponds to defining the Russell reformulations of sentences involving the `elt' operator. The concept of Russell reformulations applies just as well to proposition/term languages that are not classical, although for non-classical languages Russell reformulations may remain indeterminate.

\begin{definition}
\label{definitionrussellreformulation}
Let $L$ be a proposition/term language allowing unbounded existential quantification. The \emph{Russell reformulation} of a formula $\phi$ of $L$, which we denote by $\lb \phi\rb$, is defined by the following recursive rules:
\begin{itemize}
\item[1.] $\lb\LQ \forall x \;\; \phi(x)\RQ\rb \;\df\; \LQ \forall x \;\;\lb \phi(x)\rb$' and the obvious analogues for `$\exists$', `and', and `or'.
\item[2.] If $R$ is a primitive relation other than `$\ex$', then
\[
\lb R(t_1,\ldots,t_n)\rb \;\;\df\;\; \LQ \exists x_1,\ldots,x_n \;\; R(x_1,\ldots,x_n)\wedge (t_1 \to x_1)\wedge\cdots\wedge (t_n\to x_n)\RQ.
\]
\item[3.] $\lb\LQ \forall x \bq t \;\; \phi(x)\RQ\rb \;\df\; \LQ \exists y \;\; (t\to y) \wedge \forall x \bq y \;\; \lb \phi(x)\rb$' and the obvious analogue for `$\exists$'.
\item[4.] $\lb \mq{t\ex} \rb \;\df\; \LQ \exists x \;\; (t \to x)$'.
\end{itemize}
The last three definitions contain expressions of the form $(t \to x)$, where $t$ is a term and $x$ is a variable. The meaning of such expressions is also given recursively, by the following rules:
\begin{itemize}
\item[1$'$.] If $f$ is a primitive function, then
\[
(f(t_1,\ldots,t_n)\to x) \;\df\; \LQ \exists y_1,\ldots,y_n \;\; f(y_1,\ldots,y_n) = x \wedge (t_1\to y_1)\wedge\cdots\wedge(t_n\to y_n)\RQ.
\]
\item[2$'$.] If $y$ is a constant or variable, then $(y\to x) \;\df\; \LQ y=x$'.
\item[3$'$.] If $\Psi$ is a quantifier-like operator with corresponding schematic formula $\Theta$, then
\[
(\Psi(\xx|\tt,\bfphi) \to y) \;\;\df\;\; \lb\Theta(\xx|\tt,y,\bfphi)\rb.
\]
\end{itemize}
\end{definition}

The idea behind the concept of Russell reformulations is that the Russell formulation of a formula $\phi$ should be equivalent to $\phi$. Similarly, if $t$ is a term and $x$ is a variable then the formula $(t\to x)$ should be equivalent to the formula `$t = x$'. However, in stating this result precisely we need to first specify what we mean by ``equivalent''; cf. Appendix \ref{appendiximplication}. In the next section we introduce the notion of a standard axiomatization of a proposition/term language, yielding the following:

\begin{lemma}
\label{lemmarussellequivalent}
The Russell reformulation of any formula $\phi$ is provably equivalent to $\phi$ relative to any standard axiomatization of $L$. Similarly, for all $t$ and $x$ the formula $(t\to x)$ is provably equivalent to \text{`$t=x$'}.
\end{lemma}

Once the notion of a standard axiomatization is defined, the proof of Lemma \ref{lemmarussellequivalent} is an easy but tedious induction argument. We will also need the following technical results:

\begin{lemma}
\label{lemmarusselldefinite}
In a classical proposition/term language, the Russell reformulation of any formula is strongly negatable.
\end{lemma}

\begin{lemma}
\label{lemmarusselllength}
For all $\phi$ and $t$ we have $|\lb \phi\rb| = O(|\phi|)$ and $|(t\to x)| = O(|t|)$. Here we recall that $|\omega|$ denotes the length of a string $\omega$.
\end{lemma}

Again, the proofs are easy induction arguments. Note that the above lemmas are all purely syntactic assertions.

\subsection{Selfmeta languages}
\label{subappendixselfmeta}

Like the definition of a proposition/term language, our definition of a selfmeta language relies upon slightly non-rigorous intuitive concepts. Again, we will be able to be more rigorous when we talk about axiomatizations.

\begin{definition}
A proposition/term language $\lang_1$ is \emph{capable of reproducing} another proposition/term language $\lang_2$ if the domains of discourse, primitive relations, and primitive functions of $\lang_2$ are all interpretable in $\lang_1$; i.e. if
\begin{itemize}
\item for each domain of discourse $D_2$ of $\lang_2$,
\begin{itemize}
\item there is domain of discourse $D_1$ of $\lang_1$ and an intuitive way of ``encoding'' objects in $D_2$ as objects in $D_1$;
\item there is a corresponding schematic formula of $\lang_1$ whose intuitive meaning is ``$x$ is an encoded version of an object in $D_2$'';
\item each allowed quantifier template of $D_2$ can be translated into a valid schematic formula of $\lang_1$ via the encoding described above; and
\end{itemize}
\item for each primitive relation or function of $\lang_2$, there is a corresponding schematic formula or term of $\lang_1$ whose intuitive meaning can be thought of as being identical via the encoding.
\end{itemize}
\end{definition}

\begin{definition}
A proposition/term language $\lang$ is \emph{capable of producing lists}, or \emph{list-compatible}, if
\begin{itemize}
\item $\lang$ has only one domain of discourse $D$;
\item $\lang$ is capable of reproducing \GNT;
\item given any tuple of objects $(a_1,\ldots,a_n)$ in $D$, one can associate a new object $[a_1,\ldots,a_n]$ in $D$ which can intuitively be thought of as the ``list'' consisting of the objects $a_1,\ldots,a_n$;
\item there are schematic formulas and terms in $\lang$ whose intuitive meanings are ``$x$ is a list'', ``the $n$th element of the list $x$'' and ``the length of the list $x$''. We will denote these formulas and terms by $\List(x)$, $x[n]$, and $|x|$, respectively.
\end{itemize}
\end{definition}

The standard languages \GNT, \NT, \GST, and \ST\ are all list-compatible.

\begin{definition}
The \emph{metalanguage} of a list-compatible proposition/term language $\lang$ is the proposition/term language $\lang+1$ given as follows:
\begin{itemize}
\item The domain of discourse $D$ is the same as the domain of discourse of $\lang$, and all templates of $\lang$ are also templates of $\lang+1$. Note that since $\lang$ is capable of reproducing \GNT, which can talk about strings via the identification of a number with its binary expansion, $\lang$ can talk about strings.
\item There are two additional templates, corresponding to a new negatable primitive relation $\TT$ and a new primitive function $\RR$. Let $\phi$ and $t$ be terms intended to represent strings, and let $\xx = (x_1,\ldots,x_k)$ (resp. $\aa = (a_1,\ldots,a_k)$) be a list of terms intended to represent strings (resp. objects in $D$).
\begin{itemize}
\item The formula $\TT(\phi\;|\;\xx=\aa)$ (resp. $\neg\TT(\phi\;|\;\xx=\aa)$) means ``$\phi$ (resp. $\neg\phi$) is a true \statement, given the variable assignment $\lv\xx=\aa\rv$''.
\item The term $\RR(t\;|\;\xx=\aa)$ means ``the referent of the term $t$, given the variable assignment $\lv\xx=\aa\rv$''.
\end{itemize}
\end{itemize}
We take it as axiomatic that the intuitive picture of a list-compatible proposition/term language $\lang$ automatically leads to an intuitive picture for the metalanguage $\lang+1$.

A language is \emph{selfmeta} if it is capable of reproducing its metalanguage.
\end{definition}

\begin{theorem}[Tarski's theorem, rephrased]
\label{theoremtarski}
No classical language can be selfmeta.
\end{theorem}
\begin{proof}[Proof sketch]
For each schematic term $T = T(t)$, let $f(T)$ denote the formula `$\neg \TT(T(\LQ T\RQ))$' (here the variable assignment is taken to be trivial). Since the map $f$ is computable, there exists a schematic term $T_f$ ``encoding'' $f$ in the sense that for every schematic term $T$, the term $T_f(\LQ T\RQ)$ denotes a formula equivalent to $f(T)$. Then the term $T_f(\LQ T_f\RQ)$ denotes a formula $\phi$ equivalent to `$\neg \TT(T_f(\LQ T_f\RQ))$'. By the definition of $\TT$, the formula $\phi$ is equivalent to its own negation. But since $\lang$ is classical, $\phi$ must be either true or false, and either assumption leads to a contradiction.
\end{proof}

To make this proof more rigorous, we would first have to make clear what kind of proof we are looking for, since the claim is not a purely syntactic one. In the next section we introduce the notion of an axiomatization, and the above proof can be made rigorous in the sense that it can be formalized in any standard axiomatization of a classical selfmeta language. We would also have to make clear what we mean by ``equivalent'': we can let it mean ``provably equivalent with respect to a fixed standard axiomatization $A$''.

\section{Axiomatizations}
\label{appendixaxiomatizations}

We now turn to the question of what sort of axiom system is appropriate for a proposition/term language. Due to the fact that `implies' is not a primitive logical operator, these axiom systems will need to depend more heavily on rules of inference than standard axiom systems do.

\begin{definition}
\label{definitionproof}
An \emph{axiomatization} of a proposition/term language $\lang$ is a list of \emph{axiom schemas} and \emph{inference rules}. An axiom schema is simply a schematic formula, and an inference rule is a string of the form $\Psi_1(\tt,\bfphi),\ldots,\Psi_k(\tt,\bfphi) \Rightarrow \Psi_0(\tt,\bfphi)$, where $\Psi_0(\tt,\bfphi),\ldots,\Psi_k(\tt,\bfphi)$ are schematic formulas containing the placeholder terms and formulas $\tt$ and $\bfphi$.

In the remainder of this definition we fix a proposition/term language $\lang$ and an axiomatization $A$. A \emph{proof step} is a sentence taking one of the following forms:
\begin{itemize}
\item[(1)] `Assume $\phi$'
\item[(2)] `Then $\phi$'
\item[(3)] `Fix $x[<t]$'
\item[(4)] `Find $x[<t]$ such that $\phi(x)$'
\item[(5)] `Case 1/2: $\phi$'
\item[(6)] `Fix $m$ such that $\phi(m)$'
\end{itemize}
In (3) and (4), the brackets mean that `$<t$' should be included or not corresponding to the allowed quantifier templates of $\lang$ (with (3) corresponding to universal quantifiers and (4) corresponding to existential quantifiers). In (5), `1/2' means that either `1' or `2' should be written here. A step of type (1) is called a \emph{global hypothesis}, while a step of type (2) is called an \emph{inference}, while a step one of the types (3)-(6) is called an \emph{intermediate hypothesis}. If $S_k$ is a steps of one of the types (1),(2),(4),(5),(6), then the formula $\phi$ is called the \emph{conclusion} of $S_k$.

A \emph{proof with global free variables $\xx = (x_1,\ldots,x_n)$} is a finite list of proof steps $S_1,\ldots,S_\ell$, together with a binary ``dependence'' relation between the steps, with the following properties:
\begin{enumerate}[(I)]
\item Each step depends only on steps preceding it, and dependence is transitive.
\item The set of steps depending on $S_j$ is of the form $\{S_{j+1},\ldots,S_k\}$ for some $k$.
\item If a variable $y$ appears free in the conclusion of of a step $S_k$, then either $y$ is a global free variable (i.e. a member of $\xx$), or $S_k$ depends on a step $S_j$ of type (3) or (4) ``introducing'' $y$.
\item All global hypotheses appear consecutively at the beginning of the proof.
\item The last step is an inference that depends only on global hypotheses and inferences. The conclusion of the last step is called the \emph{global conclusion} of the proof.
\item A step $S_k$ of type (4) always depends on a prior step $S_j$ whose conclusion is of the form $\exists x[<t] \; \phi(x)$, where $\phi(x)$ is the conclusion of $S_k$.
\item Steps of type (5) come in pairs: if $S_{k_1}$ is of the form `Case 1: $\phi_1$', then there exists $S_{k_2}$ of the form `Case 2: $\phi_2$' such that both $S_{k_1}$ and $S_{k_2}$ depend directly on a prior step $S_j$ whose conclusion is of the form $\phi_1\vee\phi_2$. Here, ``depend directly'' means that $S_{k_1}$ and $S_{k_2}$ do not depend on any steps other than $S_j$ and the steps that $S_j$ depends on.
\item A step $S_k$ is said to ``relatively depend'' on an inference $S_j$ modulo a third step $S_i$ if $S_j$ depends on $S_i$, and all other steps that $S_j$ depends on either are inferences or are also steps that $S_k$ depends on.
\item If $S_k$ is an inference with conclusion $\phi$, then one of the following is true:
\begin{itemize}
\item[(2a)] There is an axiom schema of which $\phi$ is an instance.
\item[(2b)] There is an inference rule $\Psi_1,\ldots,\Psi_k \Rightarrow \Psi_0$ such that $\phi$ is an instance of $\Psi_0$ while all of the corresponding instances of $\Psi_1,\ldots,\Psi_k$ are the conclusions of prior steps that $S_k$ depends on.
\item[(2c)] (Generalization) $S_k$ relatively depends on a step of the form `Then $\mu(x)$' modulo a step of the form `Fix $x[<t]$', where $\phi = \forall x[<t] \; \mu(x)$.
\item[(2d)] (Specification) $S_k$ relatively depends on a step of the form `Then $\phi$' modulo a step of the form `Find $x[<t]$ such that $\mu(x)$', and depends on the corresponding step with conclusion $\exists x[<t] \; \mu(x)$, and $x$ does not appear free in $\phi$.
\item[(2e)] (Division into cases) For each $i=1,2$, $S_k$ relatively depends on a step of the form `Then $\phi$' modulo a step of the form `Case $i$: $\mu_i$', and depends on the corresponding step with conclusion $\mu_1\vee\mu_2$.
\item[(2f)] (Induction) $S_k$ relatively depends on a step of the form `Then $\mu(m+1)$' modulo a step of the form `Fix $m$ such that $\mu(m)$', and depends on steps with formula $\mu(0)$ and $t\ex$, where $\phi = \mu(t)$.\Footnote{One can also consider more general forms of induction such as transfinite induction. For simplicity we consider here only induction over the natural numbers.}
\end{itemize}
\end{enumerate}
\end{definition}

\begin{definition}
If there is a proof with global free variables $\xx$, global hypotheses $\phi_1(\xx),\ldots,\phi_k(\xx)$, and global conclusion $\psi(\xx)$, then $\psi(\xx)$ is said to be \emph{relatively provable} from $\phi_1(\xx),\ldots,\phi_k(\xx)$. In this case, the string
\begin{equation}
\label{implication}
\LQ\forall \xx \;\; \phi_1(\xx),\ldots,\phi_k(\xx) \;\Rightarrow \;\psi(\xx)\RQ
\end{equation}
is said to be a \emph{theorem} of the axiom system in question, and the proof in question is said to be a proof of this theorem. In general, strings of the form \eqref{implication} are called \emph{implications} or \emph{conditional \statements}. Note that implications are \emph{not} syntactically valid formulas, since they involve the implication symbol `$\Rightarrow$', as well as unbounded universal quantification (which may or may not be allowed in the proposition/term language under consideration).

If $k=0$ and $\xx=()$ is the empty list, then $\psi$ is called \emph{absolutely provable} or just \emph{provable}.

If `$\forall \xx \;\; \phi(\xx) \;\Rightarrow \;\psi(\xx)$' and `$\forall \xx \;\; \psi(\xx) \;\Rightarrow \;\phi(\xx)$' are both theorems, then
\begin{equation}
\label{equivalence}
\LQ\forall \xx \;\; \phi(\xx) \;\Leftrightarrow \;\psi(\xx)\RQ
\end{equation}
is also said to be a theorem. In general, strings of the form \eqref{equivalence} are called \emph{equivalences}.

We will use the notation $\phi\Rightarrow_A \psi$ as shorthand for `The implication $\LQ\phi\Rightarrow\psi$' is a theorem of the axiom system $A$'.
\end{definition}

\begin{definition}
\label{definitionstandardaxiom}
An axiomatization $A$ of a proposition/term language $\lang$ is \emph{standard} if it:
\begin{itemize}
\item[(1)] $A$ contains all of the following axiom schemata and inference rules:
\begin{align*}
\phi\wedge\psi &\;\Rightarrow\;\phi;&
\phi\wedge\psi &\;\Rightarrow\;\psi;&
\phi,\, \psi &\;\Rightarrow\;\phi\wedge\psi;&
\phi &\;\Rightarrow\; \phi\vee\psi;&
\psi &\;\Rightarrow\; \phi\vee\psi;&
\phi \wedge \neg\phi &\;\Rightarrow \; \psi;
\end{align*}
\begin{align*}
&x\ex;&
&c\ex;&
f(t_1,\ldots,t_n)\ex &\;\Leftrightarrow\; t_1\ex\wedge\cdots\wedge t_n\ex;&
R(t_1,\ldots,t_n) &\;\Rightarrow\; t_i\ex;&
\end{align*}
\begin{align*}
t\ex &\;\Rightarrow\; t=t;&
t_1 = t_2 &\;\Rightarrow\; T(t_1) = T(t_2);&
t_1 = t_2,\,\phi(t_1) &\;\Rightarrow\; \phi(t_2);&
\end{align*}
\item[(2)] $A$ contains all of the following inference rules allowed by its quantifier templates:
\begin{align*}
\phi(t) &\;\;\Rightarrow\;\; \exists x \;\phi(x);&
\forall x \;\phi(x),\, t\ex &\;\;\Rightarrow\;\; \phi(t);&
\\
t_1\bq t_2 \wedge \phi(t_1) &\;\;\Rightarrow\;\; \exists x \bq t_2 \;\phi(x);&
\forall x \bq t_2 \;\phi(x),\, t_1 \bq t_2 &\;\;\Rightarrow\;\; \phi(t_1);
\end{align*}
\[
\forall x \bq t \; \phi(x) \;\;\Rightarrow\;\; t\ex;
\]
\begin{align} \label{CPQ}
\forall x \; \DD(\phi(x)) &\;\;\Rightarrow\;\; \DD(\forall x \; \phi(x));&
\forall x\bq t \; \DD(\phi(x)) &\;\;\Rightarrow\;\; \DD(\forall x\bq t \; \phi(x));
\end{align}
where in \eqref{CPQ} we use the notation \eqref{definite}.
\item[(3)] $A$ contains all reductive interpretations of quantifier-like operators, i.e. equivalences of the form $\Phi(\xx|\tt,\bfphi) = s \Leftrightarrow \Theta(\xx|\tt,s,\bfphi)$;
\item[(4)] For every strongly negatable primitive relation $R$, $A$ contains the inference rule
\[
t_1\ex,\ldots,t_n\ex \;\;\Rightarrow\;\; \DD(R(t_1,\ldots,t_n));
\]
where we use the notation \eqref{definite};
\item[(5)] If $\lang$ is list-compatible, then $A$ contains corresponding inference rules allowing it to handle lists, for example
\begin{align*}
\List(t_1), \; t_2 < |t_1| \;\;&\Rightarrow\;\; t_1[t_2]\ex
\end{align*}
as well as the standard arithmetical axioms for numbers (we do not write down a complete list of the necessary axioms);
\item[(6)] If $\lang$ is selfmeta, then $A$ contains corresponding inference rules allowing it to handle the notions of truth and reference, see Definition \ref{definitionmetasystem} below. However, a standard axiom system need not contain any inference rule of the form \eqref{provabletrue}.
\end{itemize}
Note that we could also consider a slightly weaker notion of a standard axiom system, namely one where (1)-(6) are required to be theorems but not inference rules. This technical distinction will be of no significance to us.
\end{definition}

\begin{lemma}
Suppose that the implications $\phi\Rightarrow \mu$ and $\psi\Rightarrow\nu$ are theorems of a standard axiomatization of a proposition/term language. Then the implications
\begin{itemize}
\item[(A)] $\phi\wedge\psi \Rightarrow \mu\wedge\nu$
\item[(B)] $\phi\vee\psi \Rightarrow \mu\vee\nu$
\item[(C)] $\exists x [\bq t] \;\phi(x) \Rightarrow \exists x [\bq t] \;\mu(x)$
\item[(D)] $\forall x [\bq t] \;\phi(x) \Rightarrow \forall x [\bq t] \;\mu(x)$
\end{itemize}
are also theorems. Moreover if $\phi\Rightarrow \mu$ and $\psi\Rightarrow\nu$ can be proven in $m$ and $n$ symbols respectively, then \text{(A)} and \text{(B)} can be proven in $m+n+O(1)$ symbols while \text{(C)} and \text{(D)} can be proven in $m+O(1)$ symbols.
\end{lemma}
We leave the proof as an exercise to the reader, noting that it makes crucial use of (2c)-(2e) of Definition \ref{definitionproof}.

The inference rules \eqref{CPQ}, or the rules for {\bf conservation of precision under quantification (CPQ)}, deserve further comment. They function as a partial substitute for the axiom schema of excluded middle, since they describe circumstances in which certain formulas (namely universally quantified ones) may be inferred to be either true or false. The other axiom that functions as a partial substitute for the axiom schema of excluded middle is the axiom of excluded middle for atomic formulas, i.e. (4). Together, these two axioms can be used to prove the following:

\begin{theorem}
\label{theoremSNdefinite}
Every strongly negatable statement is definite.
\end{theorem}

\subsection{Meta-axiomatizations}
Just as every proposition/term language has a metalanguage, so also every axiomatization of a proposition/term language can be extended into an axiomatization of the metalanguage.

\begin{definition}
\label{definitionmetasystem}
Let $A$ be a standard axiomatization of a list-compatible proposition/term language $\lang$. The \emph{metasystem} of $A$, denoted $A+1$, is the standard axiomatization of $\lang+1$ consisting of the axiom schemata and inference rules of $A$, together with the following inference rules concerning the primitive relation $\TT$ and the primitive function $\RR$:
\begin{align*}
\TT(\LQ \forall x[<t] \; \phi(x)\RQ) \;\;&\Leftrightarrow\;\; \forall a[<\RR(t)] \; \TT(\phi(x)|x=a),&
\TT(\phi\wedge\psi) \;\;&\Leftrightarrow\;\; \TT(\phi) \wedge \TT(\psi),\\
\TT(\LQ \exists x[<t] \; \phi(x)\RQ) \;\;&\Leftrightarrow\;\; \exists a[<\RR(t)] \; \TT(\phi(x)|x=a),&
\TT(\phi\vee\psi) \;\;&\Leftrightarrow\;\; \TT(\phi) \vee \TT(\psi),
\end{align*}
\begin{align*}
\TT([\neg]R(t_1,\ldots,t_n)) \;\;&\Leftrightarrow\;\; [\neg]R(\RR(t_1),\ldots,\RR(t_n)),\\
\RR(f(t_1,\ldots,t_n))\ex \;\;&\Rightarrow\;\; \RR(f(t_1,\ldots,t_n)) = f(\RR(t_1),\ldots,\RR(t_n)),\\
\RR(t_1)\ex,\ldots,\RR(t_n)\ex \;\;&\Rightarrow\;\; \RR(f(t_1,\ldots,t_n)) = f(\RR(t_1),\ldots,\RR(t_n)),
\end{align*}
\begin{align*}
\RR(c) &= c,&
\RR(x \; | \; x=a) &= a,
\end{align*}
\begin{equation}
\label{provabletrue}
\TT(\phi), \; \phi\Rightarrow_A \psi \;\; \Rightarrow \;\; \TT(\psi).
\end{equation}
as well as the relativizations of these inference rules to an arbitrary variable assignment $\lv\xx=\aa\rv$. Note that in \eqref{provabletrue}, the formula $\phi\Rightarrow_A \psi$ is being interpreted as a syntactic \statement\ about the strings $\phi$ and $\psi$.

We also define the \emph{weak metasystem} of $A$, denoted $(A+1)^*$, to be the system consisting of all axiom schemata and rules of inference of $A+1$ except for \eqref{provabletrue}.
\end{definition}

\subsection{Shorthand}

The definition of a ``proof'' given in Definition \ref{definitionproof} to a great extent matches what mathematicians mean when they talk about a ``proof''. However, one important difference is the common mathematical practice of allowing proofs to include their own notation and definitions. This difference is significant because allowing such shorthand may significantly decrease the size of a proof. Thus, we formalize the notion of mathematical shorthand as follows:

\begin{definition}
\label{definitionshorthand}
A \emph{proof with shorthand} in an axiom system $(\lang,A)$ consists of the following:
\begin{itemize}
\item A finite list of \emph{shorthand definitions}. A shorthand definition is an expression of the form
\[
R(\xx) \df \Phi(\xx) \;\;\;\;\; \text{ or } \;\;\;\;\; f(x) \df T(\xx),
\]
where on the left-hand side, $R$ is an undefined string (i.e. a string with no previously assigned meaning), and $\Phi$ is a schematic formula, while on the right-hand side, $f$ is an undefined string and $T$ is a schematic term. The string $R$ is subsequently interpreted as a relation and the string $f$ is subsequently interpreted as a function. The schematic formula or term appearing in a shorthand definition may contain relations or functions defined by previous shorthand definitions.
\item A proof in the sense of Definition \ref{definitionproof} for the axiom system $(\lang',A')$, where $\lang'$ is the language formed by augmenting $\lang$ with the relations and functions of the shorthand definitions, and $A'$ is the axiom system formed by augmenting $A$ with the axiom schemata $f(\xx) = T(\xx)$ and the inference rules $R(\xx) \Leftrightarrow \Phi(\xx)$.
\end{itemize}
\end{definition}

\section{Proofs}
\label{appendixproofs}

All proofs in this section should be understood as proofs in the axiom system $A+1$, where $A$ is any standard axiomatization of $\lang$, or a specific standard axiomatization if this is mentioned. If we use the inference rule \eqref{provabletrue} in a proof, then we will indicate this by writing that we assume the axiom system $A$ is ``valid''.

{\bf Note.} On a philosophical level, one can ask what the proofs in this section actually allow us to conclude. A trivial answer is that they allow us to conclude that for any proposition/term syntax $\lang$ and standard axiomatization $A$, the theorems stated in this section are theorems of the axiom system $A+1$. A more interesting answer is that whenever $\lang$ is a proposition/term \emph{language}, that is, a proposition/term syntax together with an intuitive idea of its meaning as described in Definition \ref{definitionPTlanguage}, and whenever $A$ is a \emph{valid} axiomatization of it, that is, one that formalizes accepted methods of reasoning within $\lang$, then the conclusions of the theorems of this section are \emph{true} (whenever the premises are satisfied).

\begin{definition}
\label{definitionmax}
Let $\lang$ be a proposition/term language capable of reproducing \GNT. Let $t(m)$ be a term and let $\phi(m)$ be a negatable formula, in both cases with $m$ as a free variable intended to represent a number. Then we write
\[
\LQ\max\{t(m) : m \leq n,\; \phi(m)\}\RQ \df \LQ\fst\{k : \forall m < n+1 \;\;\; \phi(m) \Rightarrow t(m) \leq k\}\RQ
\]
where $\Rightarrow$ denotes material implication (i.e. $\phi\Rightarrow\psi$ is shorthand for $\neg\phi\vee\psi$).
\end{definition}

It is easy to see that this definition captures the intuitive notion of ``the maximum of $t(m)$ over all $m\leq n$ such that $\phi(m)$ holds''.

\begin{definition}
A formula or term is \emph{quantifier-free} if it does not contain any quantifiers or quantifier-like operators. A formula is $\Sigma_1$ if it is of the form $\exists x \; \phi(x)$, where $\phi$ is a quantifier-free formula.
\end{definition}

Although the sentence `$\phi$ is true', where $\phi$ is a general formula, cannot necessarily be translated into $\lang$, we will show now how to translate this sentence in the case where $\phi$ is quantifier-free.

\begin{definition}
\label{definitionQFF}
Let $\lang$ be a list-compatible proposition/term language allowing unbounded existential quantification and let $\phi$ be a quantifier-free formula of $\lang$. A \emph{verification} of $\phi$ (relative to a given variable assignment) is an assignment of truth-values to the subformulas of $\phi$, together with an assignment of values to the subterms of $\phi$, satisfying the obvious constraints, and such that $\phi$ is assigned the truth-value `true'.

If $\Phi = \LQ \exists x \; \phi(x)\RQ$ is a $\Sigma_1$ formula, then a \emph{verification} of $\Phi$ relative to a variable assignment $\lv \yy=\bb\rv$ is a pair $(a,V)$, where $V$ is a verification of $\phi(x)$ relative to the variable assignment $\lv \yy=b,x=a\rv$.
\end{definition}

Then we can translate `$\phi$ is true' into $\lang$ as `$\phi$ has a verification'. Although the validity of this translation is intuitively clear, we note that one way to justify it is by appealing to the fact that `for all $\phi$, $\phi$ is true if and only if $\phi$ has a verification' is a theorem of $A+1$.

In the theorems below, when we say that one strategy wins against another, we mean that this is true for all sufficiently large values of $n$. In fact, numbers such as $n=10^{100}$ are sufficiently large for this purpose (given physically realistic constraints on the setup), though we will not prove this explicitly.

\begin{theorem}[Cf. Theorem \ref{theoremrecstrat}]
\label{theoremrecstratB}
Let $\lang$ be a classical list-compatible proposition/term language. Then for each number $n$, there is a term $\phi(n)$ of $\lang$ of length $O(n)$ whose referent is the largest number that can be represented in $\lang$ by a term of length $\leq n$. The term $\phi(n)$ can be written explicitly.
\end{theorem}
\begin{proof}
We leave it to the reader to check:
\begin{itemize}
\item Any term $t$ in $\lang$ is provably equivalent to a term of the form `$\elt \{x : \Phi_t(x)\}$', where $\Phi_t$ is a formula with one free variable. Here, two terms $t_1$ and $t_2$ are said to be \emph{provably equivalent} if the formulas $t_1 = x$ and $t_2 = x$ are provably equivalent, where $x$ is a variable that does not appear free in $t_1$ or $t_2$, with respect to any standard axiomatization of $\lang$.
\item For any $n\geq |t|$, the formula $\Phi_t$ can be chosen to take the form
\begin{equation}
\label{m1mn}
\Phi_t(x) = \LQ\forall y_1 \;\exists z_1 \; \cdots \; \forall y_n \;\exists z_n \;\; \phi_{t,n}(x,y_1,\ldots,y_n,z_1,\ldots,z_n)\RQ,
\end{equation}
where $\phi_t$ is a quantifier-free formula with $2n+1$ free variables. (For more details see the proof of Theorem \ref{theoremrecstrat2B} below.)
\item There is an algorithm that computes $\phi_{t,n}$ as a function of $t$ and $n$. In particular, the function $(t,n)\mapsto \phi_{t,n}$ is definable in \NT\ and thus also in $\lang$.
\end{itemize}
Using Definition \ref{definitionQFF}, we can see that the term
\[
\phi(n) \df \LQ\max\{\text{Val}_n(t) : t\in L,\; |t| \leq n,\; \lb\text{Val}_n(t)\in\N\rb\}\RQ
\]
has the desired property, where $\text{Val}_n(t)$ is shorthand for
\[
\elt\{x : \forall y_1 \;\exists z_1 \; \cdots \; \forall y_n \;\exists z_n \;\; \text{$\phi_{t,n}$ is satisfied by $x,y_1,\ldots,y_n,z_1,\ldots,z_n$}\}
\]
and $\lb \phi\rb$ denotes the Russell reformulation of a \statement\ $\phi$ (cf. Definition \ref{definitionrussellreformulation}).
\end{proof}

\begin{theorem}[Cf. Theorem \ref{theoremaxiomstrat}]
\label{theoremaxiomstratB}
Let $A$ be a valid axiomatization of a proposition/term language $\lang$ capable of reproducing \GNT. For each $n$, there is a program of length $O(\log(n))$ encoding an algorithm that succeeds at computing the largest number which is the output of some algorithm for which there is a proof in $A$ of length\Footnote{Here we define the length of a proof to be the sum of the lengths of its steps.} $\leq n$ that this algorithm halts. This program can be written explicitly.
\end{theorem}
\begin{proof}
The algorithm in question can be given as follows: For each string of $n$ symbols, check whether that string is a proof in $A$ that some algorithm $\alpha$ halts. If it is, run $\alpha$ and store the value it returns. After all strings have been checked, return the largest stored value.

This algorithm can be encoded as a program whose length is equal to a constant plus the number of symbols needed to encode the number $n$, which is $O(\log(n))$.

Finally, to prove that the algorithm halts, we need to verify that whenever there is a proof in $A$ that an algorithm $\alpha$ halts of length at most $n$, then the algorithm $\alpha$ in fact halts. But this is a special case of \eqref{provabletrue}.
\end{proof}

\begin{corollary}[Cf. Corollary \ref{corollaryaxiomstrat}]
\label{corollaryaxiomstratB}
Let $(L_1,A_1)$ and $(L_2,A_2)$ be valid axiom systems, and suppose that the axiom system $(L_1,A_1)$ is at least as powerful as the metasystem $(L_2+1,A_2+1)$ in the sense that $L_1$ is capable of reproducing $L_2+1$, and all axioms and rules of inference of $A_2+1$ are theorems of $A_1$. Then in the Busy Beaver contest, the axiomatic strategy of $(L_1,A_1)$ wins agains the axiomatic strategy of $(L_2,A_2)$.
\end{corollary}
\begin{proof}
Let $\rho_n(A_2)$ be the algorithm given by the previous theorem. Then the previous proof shows that the \statement\ `$\rho_n(A_2)$ halts and outputs a number' can be proven in $A_2+1$, and thus in $A_1$, using $O(\log(n))$ symbols. Consequently for all $m\gtrgtr \log(n)$, the output of $\rho_n(A_2)$ is a term in the maximum computed by $\rho_m(A_1)$. So the output of $\rho_m(A_1)$ is larger than the output of $\rho_n(A_2)$.
\end{proof}

\begin{theorem}[Cf. Theorem \ref{theoremselfmeta}]
\label{theoremselfmetaB}
Let $L_1$ and $L_2$ be proposition/term languages, such that $L_1$ is capable of reproducing $L_2$. Let $A$ be a valid standard axiomatization of $L_1$. Then for each $n$, there is a term $\phi(n)$ of $L_2+1$ of length $O(\log(n))$ which succeeds at referring to the largest number named by any term $t$ such that $A$ proves that $t$ succeeds at naming a number using less than $n$ symbols. This term can be written explicitly. If $L_2$ is selfmeta, the term can be translated into $L_2$.
\end{theorem}
\begin{proof}
Clearly, the formula `$A$ proves that $t$ succeeds at naming a number using less than $n$ symbols' can be translated into \GNT, and thus in $L_2+1$, as a strongly negatable formula, and by Theorem \ref{theoremSNdefinite}, such a formula is definite. This implies that the phrase `the largest number ... less than $n$ symbols' can be translated into $L_2+1$ in a straightforward way using Definition \ref{definitionmax}. The length of the resulting term will be equal to a constant times the number of symbols needed to name $n$ in $L_2+1$, which is $O(\log(n))$.

Finally, to prove that the resulting term names a number, we need to verify that whenever there is a proof in $A$ that a term $t$ names a number, then $t$ in fact names a number. Again, this is a special case of \eqref{provabletrue}.
\end{proof}

\begin{corollary}[Cf. Corollary \ref{corollaryselfmeta}]
\label{corollaryselfmetaB}
Let $\lang_*$ be a selfmeta language, and let $(L_1,A_1)$ and $(L_2,A_2)$ be valid axiom systems capable of reproducing $\lang_*$. Suppose in addition that $A_1$ is at least as powerful as the metasystem $A_2+1$. Then in the largest number contest of $\lang_*$, the axiomatic strategy of $A_1$ wins against the axiomatic strategy of $A_2$.
\end{corollary}
\begin{proof}
Let $t_n(A_2)$ be the term given by the previous theorem. Then the previous proof shows that the \statement\ `$t_n(A_2)$ names a number' can be proven in $A_2+1$, and thus in $A_1$, using $O(\log(n))$ symbols. Consequently the number named by $t_n(A_2)$ appears as a term in the maximum occuring in the term $t_m(A_1)$, whenever  $m \geq O(\log(n))$. So the number named by $t_m(A_1)$ is larger than the number named by $t_n(A_2)$.
\end{proof}

\begin{theorem}[Cf. Theorem \ref{theoremlingstrat}]
\label{theoremlingstratB}
Let $L_1$ and $L_2$ be two proposition/term languages, at least one of which is classical, and suppose that $L_1$ is capable of reproducing the metalanguage $L_2 + 1$. Then for any $n$, the phrase `the largest number that can be represented in $L_2$ using at most $n$ symbols' can be translated into $L_1$ as a term of length $O(\log(n))$ which succeeds at naming a number. In particular, in any philosophical largest number contest in which both strategies are considered valid, the linguistic strategy of $L_1$ will win against the linguistic strategy of $L_2$.
\end{theorem}

\begin{proof}
First of all, the formula `$t$ is a term of $L_2$ that succeeds at naming a number using less than $n$ symbols' can be translated into the metalanguage of $L_2$. If $L_2$ is classical, then this formula is definite and negatable in $L_2+1$, and if $L_1$ is classical, the formula is definite and negatable in $L_1$. Thus using Definition \ref{definitionmax}, the phrase given in the theorem can be translated into $L_1$. It is easy to see that the term transcribing this phrase is of length $O(\log(n))$ and succeeds at naming a number.

Thus, there is a strategy for the largest number contest of $L_1$ which will win against any entry for the largest number contest of $L_2$, and in particular against the linguistic strategy. Since the linguistic strategy of $L_1$ is nearly optimal, it will also win against the linguistic strategy of $L_2$.
\end{proof}

The following theorem is motivated by Theorem \ref{theoremrecstrat2}, but importantly different from it:

\begin{theorem}
\label{theoremrecstrat2C}
Let $A$ be a standard axiomatization of $\lang=\GNT$. Then for each $n$, the implication
\begin{equation}
\label{provabletruen}
\forall \phi,\psi,\lv\xx=\aa\rv \;\;\;\;\;\; \TT(\phi(\xx) \;|\; \xx=\aa),\; \phi \Rightarrow_{A,n} \psi \;\;\Rightarrow\;\; \TT(\psi(\xx) \;|\; \xx=\aa)
\end{equation}
can be proven in $A$ using $O(n\log(n))$ symbols. Here $\phi \Rightarrow_{A,n} \psi$ means `the implication $\phi\Rightarrow\psi$ can be proven in $A$ using $\leq n$ nested inductions'.
\end{theorem}
\begin{proof}
We begin with a definition. Let $\rho$ be a proof, and let $\lv\xx=\aa\rv$ be a variable assignment. A \emph{partial validation of $\rho$ of length $k$ with respect to $\lv\xx=\aa\rv$} is a finite list of pairings between steps and variable assignments such that for each step $S_j$, if $V_j$ denotes the set of variable assignments paired with $S_j$, then
\begin{itemize}
\item $V_{-1} = \{\lv\xx=\aa\rv\}$;
\item if $\lv\yy=\bb\rv \in V_j$, and $S_j$ is not of type (3) (cf. Definition \ref{definitionproof}), then the conclusion of $S_j$ is true relative to the assignment $\lv\yy=\bb\rv$;
\end{itemize}
and if furthermore $j\leq k$ and $S_j$ does not depend on any steps of type (6), then the following conditions hold where $S_i$ is the last step that $S_j$ depends on, unless $S_j$ does not depend on any other steps in which case we let $i=-1$:\Footnote{Note that by condition (II) of Definition \ref{definitionproof}, $S_j$ directly depends on $S_i$.}
\begin{itemize}
\item if $S_j$ is of type (1) or (2), then $V_j = V_i$;
\item if $S_j$ is of the form `Fix $z<t(\yy)$', then for each $\lv\yy=\bb\rv\in V_i$ and for each $c<\RR(t(\yy) \;|\; \yy=\bb)$, we have $\lv\yy=\bb,z=c\rv \in V_j$;
\item if $S_j$ is of the form `Find $z$', then  for each $\lv\yy=\bb\rv\in V_{j-1}$ there exists $c$ such that $\lv\yy=\bb,z=c\rv \in V_j$;
\item if $S_j$ is of the form `Case 1: $\phi_1$', then $V_j\cup V_{j'} = V_i$, where $j'$ is the corresponding step of the form `Case 2: $\phi_2$'.
\end{itemize}
We now proceed with the proof by induction on $n$. We will show that \eqref{provabletruen} can be proven in $\leq C f(n)$ symbols, where
\[
f(n) = (n+2) \log(n+3)
\]
and $C$ is a large constant. Note that when $n=-1$, \eqref{provabletruen} is trivially provable because it is vacuously true. Thus, fix $n\geq -1$ and suppose that \eqref{provabletruen} can be proven in $A$ using $\leq C f(n)$ symbols. Now $A$ can reason as follows:

``Let $\phi,\psi,\lv\xx=\aa\rv$ satisfy $\TT(\phi(\xx) \;|\; \xx=\aa)$ and $\phi \Rightarrow_{A,n+1} \psi$. Then there is a proof $\rho$ in $A$ of $\phi\Rightarrow\psi$ using $\leq n+1$ nested inductions. We will prove by induction that for all $k\leq\ell$, there exists a partial validation of $\rho$ of length $k$ for the variable assignment $\lv\xx=\aa\rv$.

``Fix $k < \ell$ and suppose that there exists a partial validation of $\rho$ of length $k$. Let $j=k+1$. We need to show that the partial validation can be extended to a partial validation of length $j$. If $S_j$ is not of type (2f), then the proof is tedious and left as an exercise to the reader.

``If $S_j$ is of type (2f), then fix a variable assignment $\lv\yy=\bb\rv\in V_j = V_i$. We need to show that the conclusion of $S_j$, which we denote by $\mu(\yy,t(\yy))$, is true relative to $\lv\yy=\bb\rv$. For ease of notation we assume that $\lv\yy=\bb\rv = \lv\rv$ is the trivial assignment.

``Fix $m$ and suppose that $\phi' = \mu(m)$ is true. By the structure of $\rho$, the inference $\mu(m) \;\Rightarrow\;\mu(m+1)$ is provable in $A$ using $\leq n$ nested inductions. So [insert proof of \eqref{provabletruen} using less than $\leq C f(n)$ symbols], and thus $\psi' = \mu(m+1)$ is true. Thus by induction, $\mu(\RR(t))$ is true, and thus $\mu(t)$ is true.

``Thus there exists a partial validation of $\rho$ of length $j=k+1$. So by induction, there exists a partial validation of $\rho$ of length $\ell$. Since by assumption $S_\ell$ depends only on global hypotheses and inferences, we have $V_\ell = V_{-1} = \{\lv\xx=\aa\rv\}$ and thus the conclusion of $S_\ell$ is true relative to $\lv\xx=\aa\rv$.''

The length of this proof is $\leq C f(n) + O(\log(n))$,\Footnote{Note that the length $O(\log(n))$ is required not only to keep track of the numeral $n$ appearing in the proof, but also to distinguish between the different variables that appear at different levels of nesting.} so if $C$ is chosen appropriately then the length is $\leq C f(n+1)$. This completes the inductive step.
\end{proof}

\begin{theorem}[Cf. Theorem \ref{theoremrecstrat2}]
\label{theoremrecstrat2B}
Let $\lang$ be a classical list-compatible proposition/term language, and let $A$ be a standard axiomatization of $\lang$. Then for each $n$, the \statement\ `every $\Sigma_1$ \statement\ provable in $A$ using $\leq n$ nested quantifiers is true' (where by `is true', we mean `has a verification' as per Definition \ref{definitionQFF}) can be proven in $A$ using $O(n^2\log(n))$ symbols, or $O(n\log(n))$ if we allow shorthand.
\end{theorem}
\begin{proof}
First we consider the case where shorthand is allowed. For each $i\in\N$, we will call a formula $\Phi$ an \emph{$i$-formula} if it is of the form
\begin{equation}
\label{prefix}
\Phi(\xx) = \LQ\forall y_i \;\exists z_i \; \cdots \; \forall y_1 \;\exists z_1 \;\; \phi(\xx,\yy,\zz)\RQ,
\end{equation}
where $\phi$ is a quantifier-free formula. (The reason for the decreasing indices will be apparent later.) In particular, a formula is a $0$-formula if and only if it is quantifier free.

The following description of shorthand can be translated into $\lang$: ``A $0$-formula is said to be \emph{$0$-true} relative to a variable assignment $\lv\xx=\aa\rv$ if it has a verification (cf. Definition \ref{definitionQFF}). A $1$-formula $\Phi = \LQ\forall y_1 \;\exists z_1 \; \phi(\xx,y_1,z_1)\RQ$ is said to be \emph{$1$-true} relative to a variable assignment $\lv\xx=\aa\rv$ if $\phi$ is a $0$-formula such that for all $b$, there exists $c$ such that $\phi$ is true relative to the variable assignment $\lv\xx=\aa,y_1 = b,z_1 = c\rv$. [Repeat the previous sentence $n$ times, each time incrementing all numerical values by 1.]'' We remark that an $i$-formula is $i$-true if and only if it is true, though this fact will not be needed for the proof.

To compute the length of the description of shorthand in the previous paragraph, we observe that it consists of $n$ definitions of shorthand, each of which  must refer to some natural number between $0$ and $n$ (e.g. via its binary expansion). Thus each definition is of length $O(\log(n))$, and the length of the sequence of definitions is $O(n\log(n))$.

\begin{remark*}
The shorthand description ``For all $i < n$, an $(i+1)$-formula $\Phi$ is said to be \emph{$(i+1)$-true} if [etc]'' cannot be translated into $\lang$, because our shorthand conventions (cf. Definition \ref{definitionshorthand}) do not allow recursive definitions. Allowing shorthand to include recursive definitions changes the situation, and we do not analyze the modified version here.
\end{remark*}

Now for each $i,j \leq n$ such that $i \leq j$, we define the \emph{$j$-extension} of an $i$-formula $\Phi$ to be the formula `$\forall y_j \;\exists z_j \; \cdots \; \forall y_{i+1} \;\exists z_{i+1} \;\Phi(\xx)$'. Note that since the notion of a $j$-extension is purely syntactic, it can be defined in $\lang$ simultaneously for all $j$. An $i$-formula $\Phi$ is said to be \emph{$j$-true} relative to a variable assignment $\lv\xx=\aa\rv$ if the $j$-extension of $\Phi$ is $j$-true relative to $\lv\xx=\aa\rv$.

Now for each $j < n$, the $\lang$-formula $\Xi(j)=\;$`For every $i\leq j$, for every $i$-formula $\Phi$, and for every variable assignment $\lv\xx=\aa\rv$, $\Phi$ is $j$-true relative to $\lv\xx=\aa\rv$ if and only if $\Phi$ is $(j+1)$-true relative to $\lv\xx=\aa\rv$' can be proven in $A$ using $O(\log(n))$ symbols. Now let $\Theta(j)$ denote the $\lang$-formula `For every $i \leq j$, for every $i$-formula $\Phi$, and for every variable assignment $\lv\xx=\aa\rv$, $\Phi$ is $j$-true relative to $\lv\xx=\aa\rv$ if and only if $\Phi$ is $n$-true relative to $\lv\xx=\aa\rv$'. Note that both $\Xi(j)$ and $\Theta(j)$ are \statements\ rather than merely implications due to our assumption that the language $\lang$ is classical.  Now for all $j<n$, the implication $\Theta(j+1) \wedge \Xi(j)\Rightarrow\Theta(j)$ can be proven in $A$ using $O(\log(n))$ symbols. Since $\Theta(n)$ is trivially provable, it follows that $\Theta(1),\ldots,\Theta(n)$ can be simultaneously proven in $A$ using $O(n\log(n))$ symbols.

A division into cases of the form $i=i_0$ shows that $A$ can prove the following \statement\ using $O(n\log(n))$ symbols:
\begin{itemize}
\item[$(*)$] For every $i < n$, for every $(i+1)$-formula $\Phi = \LQ \forall y_{i+1} \;\exists z_{i+1}\; \phi(\xx,y_{i+1},z_{i+1})$', and for every variable assignment $\lv\xx=\aa\rv$, $\Phi$ is $n$-true relative to $\lv\xx=\aa\rv$ if and only if for all $b$, there exists $c$ such that $\phi$ is $n$-true relative to the variable assignment $\lv\xx=\aa,y_{i+1} = b,z_{i + 1} = c\rv$.
\end{itemize}
This completes the part of the proof where we need to keep track of the length of the proofs in $A$. In the remainder of the proof, we will simply use $(*)$ as a black box and use reasoning that can entirely be formalized in $A$ (using $O(1)$ symbols). The following lemma shows that the conjunction operator `and' interacts nicely with the notion of $n$-truth:

\begin{lemma}
Fix $i,j\leq n$ with $i + j \leq n$. Let $\Phi$ be an $i$-formula, and let $\Psi$ be a $j$-formula. Then there exists an $(i+j)$-formula $\Theta$ (computed explicitly from $\Phi$ and $\Psi$) such that for any variable assignment $\lv\xx=\aa\rv$, $\Theta$ is $n$-true relative to $\lv\xx=\aa\rv$ if and only if $\Phi$ and $\Psi$ are both $n$-true relative to $\lv\xx=\aa\rv$.
\end{lemma}
\begin{proof}
If
\begin{align*}
\Phi(\xx) &= \LQ\forall y_i \;\exists z_i \; \cdots \; \forall y_1 \;\exists z_1 \; \phi(\xx,\yy,\zz)\RQ\\
\Psi(\xx) &= \LQ\forall y_j \;\exists z_j \; \cdots \; \forall y_1 \;\exists z_1 \; \psi(\xx,\yy,\zz)\RQ
\end{align*}
then we let
\[
\Theta(\xx) = \LQ\forall y_{i+j} \; \exists z_{i+j}\; \cdots \; \forall y_1 \; \exists z_1 \; \phi(\xx,\yy_1^i,\zz_1^i) \wedge \psi(\xx,\yy_{i+1}^{i+j},\zz_{i+1}^{i+j})\RQ.
\]
To prove that $\Theta$ has the desired property, we apply induction on $j$ and use $(*)$.
\end{proof}


Similar lemmas can be proved for the `or' operator and for both types of quantifiers. These lemmas can be used to define the \emph{prefix version} of any formula that can be written using `and', `or', quantifiers, and quantifier-free formulas, i.e. any formula that does not contain quantifier-like operators. Since the Russell reformulation of a formula does not contain quantifier-like operators, this motivates the following definitions:
\begin{itemize}
\item a formula  $\phi$ is $n$-true relative to a variable assignment $\lv\xx=\aa\rv$ if the prefix version of the Russell reformulation $\lb \phi\rb$ is $n$-true relative to $\lv\xx=\aa\rv$;
\item the $n$-referent of a term $t$ relative to a variable assignment $\lv\xx=\aa\rv$ is the unique object $b$ (if one exists) such that the prefix version of the formula $(t\to y)$ is $n$-true relative to the assignment $\lv\xx=\aa,y=b\rv$.
\end{itemize}
By Lemma \ref{lemmarusselllength}, there exists a constant $C$ such that if $|\phi|\leq n/C$ (resp. $|t|\leq n/C$), then the prefix version of $\lb \phi\rb$ (resp. $(t\to y)$) is an $i$-formula for some $i < n$. Thus, these definitions provide a way for $\lang$ to partially reproduce the metalanguage $\lang+1$, subject to the restriction that the notions of truth and reference are only used for formulas and terms of length $\leq n/C$. It is routine to check that all axiom schemata and inference rules of the metasystem $A+1$ except for \eqref{provabletrue} are theorems of $A+\text{(*)}$.\Footnote{In order to check the forward direction of the equivalence on the third line, one first needs to prove by induction on $|t|$ that there can be at most one $b$ such that the prefix version of $(t\to y)$ is $n$-true relative to the variable assignment $\lv\xx=\aa,y=b\rv$.} Thus since the proof of Theorem \ref{theoremprovabletrue} can be formalized in $(A+1)^*$, its relativized version can be proven in $A+\text{(*)}$.\Footnote{Note that there is something to check here, namely that the proof of Theorem \ref{theoremprovabletrue} does not rely on the notions of truth and reference for formulas and terms that are significantly longer than the proof whose conclusion we wish to show is true.} But this relativized version says that every \statement\ provable in $A$ using less than $n/C$ symbols is $n$-true, and it is easy to check that for $\Sigma_1$ \statements, $n$-truth coincides with the notion of truth given by Definition \ref{definitionQFF}. This completes the proof.

If we are not allowed to use shorthand, then each time the notion of $i$-truth appears in the proof it must be expanded in full, which costs $O(n\log(n))$ symbols instead of $O(\log(n))$. So the final length would be $O(n^2\log(n))$ rather than $O(n\log(n))$.
\end{proof}

\begin{theorem}[In $(A+1)^*$]
\label{theoremprovabletrue}
Let $A$ be a valid axiomatization of a classical list-compatible proposition/term language $\lang$. Then for each $n$, the \statement\ \eqref{provabletruen} is true, where $\phi\Rightarrow_{A,n}\psi$ means `the implication $\phi\Rightarrow \psi$ can be proven using $\leq n$ nested steps'.
\end{theorem}
Note that \eqref{provabletruen} is a \statement\ rather than an implication because of the assumption that $\lang$ is classical.
\begin{proof}
Clearly, \eqref{provabletruen} is true for $n=-1$. Suppose that \eqref{provabletruen} is true, and fix $\phi,\psi,\lv\xx=\aa\rv$ such that $\phi(\xx)$ is true relative to $\lb\xx=\aa\rv$ and $\phi \Rightarrow_{A,n+1}\psi$. Let $\rho$ be the proof of $\phi\Rightarrow\psi$. We will prove by induction that if $S_j$ is an inference that does not depend on any intermediate hypotheses, then the conclusion of $S_j$ is true relative to $\lv\xx=\aa\rv$. Indeed, suppose that $S_j$ is of type (2c), with conclusion $\forall y \; \mu(y)$. Let $\nu_1,\ldots,\nu_r$ be the conclusions of the steps that $S_j$ depends on. Then the implication $\nu_1,\ldots,\nu_r\Rightarrow \mu(y)$ is provable in $A$ using $\leq n$ nested steps. Since $\nu_1,\ldots,\nu_r$ are all true relative to $\lv\xx=\aa\rv$, it follows that for all $b$, $\mu(y)$ is true relative to $\lv\xx=\aa,y=b\rv$ and thus by the induction hypothesis $\forall y \; \mu(y)$ is true. The other cases are similar. This completes the induction with respect to $S_j$. Letting $S_j$ be the last step of $\rho$, we see that $\psi$ is true relative to $\lv\xx=\aa\rv$, which completes the induction with respect to $n$.
\end{proof}

\ignore{
Old proof:

It turns out to be easiest to prove (ii) first and then prove (i) as a variant.
\begin{proof}[Proof of \text{(ii)}]
\begin{lemma}
\label{lemmaAinference}
Let $A$ be a standard axiomatization of a selfmeta language $\lang$. Then the implication `Every formula of $\lang$ that can be inferred from true formulas using a single application of an inference rule of $A$ is true' is provable in $A$.
\end{lemma}
\begin{proof}
Since $A$ has only finitely many inference rules, it suffices to show that if $R$ is an inference rule of $A$, then `Every formula of $\lang$ that can be inferred from true formulas using a single application of rule $R$ is true' is a theorem of $A$. So let $R$ be an inference rule of $A$. Then $R$ is encoded as a schema $\Psi_1(\tt,\bfphi),\ldots,\Psi_k(\tt,\bfphi) \Rightarrow \Psi_0(\tt,\bfphi)$, where $\tt$ and $\bfphi$ are lists of placeholder terms and formulas. Let $\wbar\tt$ (resp. $\wbar\bfphi$) denote the result of replacing each term $t_i$ in $\tt$ by the term `the referent of $t_i$'
 (resp. of replacing each formula $\phi_j$ in $\bfphi$ by the formula `$\phi_j$ is true'). Since $A$ is standard, for each $i = 0,\ldots,k$ the equivalence $\TT(\Psi_i(\tt,\bfphi)) \Leftrightarrow \Psi_i(\wbar\tt,\wbar\bfphi)$ is a theorem of $A$.

Thus the following reasoning can be translated into $A$: ``Suppose that $\phi$ is a formula of $\lang$ that can be inferred from true formulas using a single application of rule $R$. By definition, this means that there exist $\tt,\bfphi$ such that the formulas $\Psi_1(\tt,\bfphi),\ldots,\Psi_k(\tt,\bfphi)$ are all true, and in addition $\phi = \Psi_0(\tt,\bfphi)$. Thus [insert $\Psi_1(\wbar\tt,\wbar\bfphi),\ldots,\Psi_k(\wbar\tt,\wbar\bfphi)$ here], so by rule $R$, [insert $\Psi_0(\wbar\tt,\wbar\bfphi)$ here]. Thus $\phi = \Psi_0(\tt,\bfphi)$ is true.'' Here, the editorial mark `[insert $\psi$ here]' is a note for the person tasked with translating this reasoning into $\lang$, indicating that they should insert the formula $\psi$ (which is a well-formed formula of $\lang$, but not a legitimate English phrase) rather than translating any English phrase. We are forced to include this awkward notation in the English version of the reasoning because the formulas $\Psi_i(\wbar\tt,\wbar\bfphi)$ cannot be straightforwardly translated into English without more information about their content. (The nearest substitute, `$\Psi_i(\wbar\tt,\wbar\bfphi)$ is true', would be translated back into $\lang$ as `$\TT(\Psi_i(\wbar\tt,\wbar\bfphi))$', which is a different $\lang$-formula from $\Psi_i(\wbar\tt,\wbar\bfphi)$. However, in the sequel we will be less worried about this distinction.)
\end{proof}

Now for each $n$, let $\Phi(n)$ denote the \statement\ `Every sentence of $\lang$ provable in $A$ using at most $n$ symbols is true'. Note that here we are treating $\Phi(n)$ as a \statement\ rather than as an implication; this is possible because we can write $\Phi(n) = \LQ \forall |\phi| \leq n \;\; P(\phi,A,n) \Rightarrow \TT(\phi)$', where $P(\phi,A,n)$ is short for ``$\phi$ is provable in $A$ using at most $n$ symbols'', and $|\phi|$ denotes the length of $\phi$. The quantifier is bounded, and the $\Rightarrow$ symbol can be treated as a material implication.

We claim that for each $n$, the implication $\Phi(n) \Rightarrow \Phi(n+1)$ is a theorem of $A$ which can be proven using $O(\log(n))$ symbols. Indeed, the following reasoning can be formalized in $A$: ``Suppose that $\Phi(n)$ is true. Let $\phi$ be a sentence of $\lang$ provable in at most $n+1$ symbols. Then $\phi$ can be inferred from true sentences using a single application of an inference rule of $A$. Thus, $\phi$ is true. Since $\phi$ was arbitrary, $\Phi(n+1)$ is true.'' In the second-to-last sentence, we have used Lemma \ref{lemmaAinference}. Since the preceding reasoning requires a representation of $n$, its length is $O(\log(n))$.

But now by concatenation, for each $n$, $A$ can prove the \statement\ $\Phi(n)$ using $O(n\log(n))$ symbols.
\end{proof}

\begin{remark}
The previous proof can be formalized in \GNT$_1$; the ``by concatenation'' part means that we must go up a level.
\end{remark}

\begin{proof}[More complicated and more correct proof of \text{(ii)}]
Let $\rho$ be a proof. The \emph{aggregation} of $\rho$ is the formula constructed as follows:
\begin{itemize}
\item Start with the empty string. Move the counter to the first step.
\item Wherever the counter is now, if that step is `Assume $\phi$' or `Then $\phi$', then add `$\wedge\phi$' to the current string, omitting the wedge if we currently have the empty string.
\item If that step is `Fix $x[<t]$', then let $\phi(x)$ be the aggregation of the list of steps that depend on this step. Then add `$\wedge\forall x[<t] \;\phi(x)$' to the current string, while removing the dependencies from circulation.
\item If that step is `Find $x[<t]$ such that $\phi(x)$', then let $\psi(x)$ be the aggregation of dependencies. Then add `$\wedge\exists x[<t] \;\phi(x)\wedge\psi(x)$'.
\item If that step is `Case 1: $\mu_1$', then find the corresponding step `Case 2: $\mu_2$', and if none exists then pretend it is at the end with nothing depending on it. Let $\psi_1$ (resp. $\psi_2$) be the aggregation of dependencies, and then add `$\wedge ((\mu_1\wedge\psi_1)\vee (\mu_2\wedge\psi_2))$'.
\end{itemize}
Then the inductive claim is that the aggregation of the first $n$ steps is a true \statement. [Continue\internal]

Now let $\rho$ be a proof and let $\phi$ be a step with no hypotheses. A \emph{partial validation} of $\rho$ consists of the following objects:
\begin{itemize}
\item A subset $S$ of the steps of $\rho$, with the property that every hypothesis used by a member of $S$ is an element of $S$.
\item A variable assignment with respect to which all members of $S$ are true.
\end{itemize}
Now for each $n$, let $\Phi(n,\rho,\phi)$ denote the \statement\ `Either $\phi$ is true, or there exists a partial validation of $\rho$ of cardinality at least $n$''. As before, we claim that for each $n$, the implication $\forall \rho \; \Phi(n,\rho,\phi) \Rightarrow \Phi(n+1,\rho,\phi)$ is a theorem of $A$ which can be proven using $O(\log(n))$ symbols. Indeed, the following reasoning can be formalized in $A$:

``Suppose that $\Phi(n,\rho,\phi)$ is true, i.e. either $\phi$ is true, or there exists a partial validation $(S,\LQ\xx = \aa\RQ)$ of $\rho$ of cardinality $n$. In the first case, we are done. In the second case, since $\phi$ has no hypotheses, we may choose $\phi'$ to be the first step of $\rho$ which is not a member of $S$ but whose hypotheses are all elements of $S$. If $\phi'$ is the instantiation of an axiom schema or an inference rule, then Lemma \ref{lemmaAinference} completes the proof. Otherwise, we are in case 3 or 4.

``Suppose we are in case 3, and let $\mu$ and $\nu$ be as there. Then since `$\mu$ or $\nu$' is in $S$, it is true under the current variable assignment, so either $\mu$ is true or $\nu$ is true. Without loss of generality suppose $\mu$ is true. Then we can add $\mu$ to $S$ if it is not already in $S$. If it is already in $S$, then by the minimality assumption on $\phi'$, the earlier occurrence of $\phi'$ is also in $S$. Thus we can add $\phi'$ to $S$.

``Suppose we are in case 4, and let $\mu(x)$ be as there. Then since `$\exists x \; \mu(x)$' is in $S$, it is true under the current variable assignment, so we may select $a$ such that $\mu(a)$ is true. Then $\mu(x)$ is true under the variable assignment $\{x = a\}$ (augmented with the current variable assignment). So we can add $\mu(x)$ to $S$ if it is not already in $S$. If it is already in $S$, then by the minimality assumption on $\phi'$, the earlier occurrence of $\phi'$ is also in $S$. Thus we can add $\phi'$ to $S$.''

[Cases 5 and 6 needed\internal]

As before, since the preceding reasoning requires a representation of $n$, its length is $O(\log(n))$, and thus by concatenation, for each $n$, $A$ can prove the \statement\ $\Phi(n)$ using $O(n\log(n))$ symbols.
\end{proof}

}

\ignore{
\section{A framework for math}
\label{appendixframework}

In this section we describe a proposition/term language and corresponding axiom system capable of transcribing the arguments of the previous section. This section is not strictly necessary [etc]

\begin{definition}
The language of \emph{abstract gradualist mathematics} is the proposition/term language in which
\begin{itemize}
\item there are five domains of discourse: strings/numbers, objects, categories, functions, and relations
\item quantification: is allowed over objects and tuples only within a fixed category; unbounded universal quantification is only allowed in a complete category.
\item there are several constants: the category of numbers, the arithmetic operations (thought of as members of the category of functions), `$=$', `$\leq$', `$\ex$' (thought of as members of the category of relations)
\item there are two primitive relations: membership (between objects and categories) and satisfaction (between objects and relations)
\item there are several primitive function: application (which inputs objects and functions and outputs objects), arity (relations, functions $\to \N$), 
\item in addition, there is a unary relation for categories called ``completeness''; intuitively, a category is complete if the principle of conservation of precision under unbounded quantification is satisfied.
\end{itemize}
This language is capable of describing both the syntax and semantics of proposition/term languages: a proposition/term language $\lang$ is a tuple consisting of a collection of syntax rules (thought of as a string) together with a list of categories, relations, and functions, corresponding to the variable types and primitive relations/functions for the syntax rules.
\end{definition}
}

\section{Implication and induction}
\label{appendiximplication}


{\it The meaning of implication.} A fundamental tool in mathematical and philosophical thought is \emph{hypothetical reasoning}, the mode of reasoning in which a reasoner ``pretends'' that a certain hypothesis is true and then observes what kind of conclusions his mind generates based on this ``information''. For example, a mathematician may pretend that the symbol `$n$' denotes a number and then perform manipulations under this pretense, for example concluding that the \statement\ `$n + n = 2n$' is true. The implicit claim is that knowing the results of hypothetical reasoning can be useful when encountering new situations. For example, if the mathematician decides that he will temporarily let the symbol `$n$' denote the number $10^{100}$, then the existence of his previous hypothetical reasoning gives him good reason to believe that the \statement\ `$10^{100} + 10^{100} = 2\cdot 10^{100}$' is true, and thus that $10^{100} + 10^{100} = 2\cdot 10^{100}$.

It is worth analyzing more closely the thought process of a reasoner who applies previously performed hypothetical reasoning to a current situation. The thought process must go something like ``I was reasoning well then, and if I reasoned in the same way regarding the current situation, then I would come to the same conclusion as I did when I was reasoning hypothetically. Thus, I can save myself the time and effort of re-reasoning by simply accepting that the conclusion of my previous reasoning holds in the new situation.'' We should point out that this kind of reasoning is not always valid. First of all, it is not necessarily true that if the reasoner reasoned in the same way in the new situation, he would reach the same conclusion. He might notice flaws in the reasoning that were not noticed when he was just pretending. Secondly, even if he would reason in the same way, the resulting reasoning could be invalid even though he would accept it. In this case, the reasoner is still correct that he can save himself time and effort by accepting the conclusion of the hypothetical reasoning in the current situation, but he is simply getting to nowhere faster.

The more serious difficulty is that the notion of ``reasoning well'' is not very precise. Indeed, hypothetical reasoning may be based on false premises, in which case it is hard to see what it would mean to say that it was ``reasoned well''.\Footnote{Philosophers distinguish between ``sound'' and ``valid'' reasoning; the difference is that ``sound'' reasoning is required to have true premises but ``valid'' reasoning isn't. What we are doing is contesting the preciseness of the concept of ``valid'' reasoning.} One may mean merely that it conforms to standard methods of reasoning, but we shall see the difficulties with this kind of response below.

Due in part to these difficulties, mathematicians have developed various notions of \emph{formal implication}, with the intention that after performing hypothetical reasoning, there should be some \emph{fact} that can be concluded simply from the existence of the hypothetical reasoning. If such a fact has the additional property that there are standard modes of reasoning via which a person, believing the fact and being confronted with a new situation, may conclude that the conclusion of their previous hypothetical reasoning is valid in the new situation, then perhaps these standard modes of reasoning may be found to be less problematic than the reasoning based on ``saving time and effort'' given above.

Two canonical examples are \emph{material implication} and \emph{universal quantification}. If $P$ and $Q$ are two claims, then $P$ is said to \emph{materially imply} $Q$ if either $Q$ is false or $P$ is true. Now suppose that $P$ and $Q$ both have definite truth-values, and that there is reasoning that yields $Q$ as a conclusion using $P$ as a hypothesis. If $P$ is true and $Q$ is false, then the reasoning has true premises but a false conclusion, meaning that it must be flawed. Since $P$ and $Q$ have definite truth-values, in all other cases $P$ materially implies $Q$. So either the reasoning is flawed, or $P$ materially implies $Q$. If we hold the reasoning in our mind and believe that there is no flaw in it, then we appear to be justified in concluding that $P$ materially implies $Q$, a fact which can then be used later if we discover that $P$ is true without reference to our previous hypothetical reasoning.

If the hypothetical reasoning can be formalized in an axiom system, then there is an even better way to reach the conclusion that $P$ materially implies $Q$. 
Namely, each claim $X$ that we make in our hypothetical reasoning may be replaced by the unconditional claim ``$P$ materially implies $X$''. The hypothesis, $P$, can be justified because the fact that $P$ has a definite truth-value means that it materially implies itself. Reproducing the hypothetical reasoning as ``material-conditional reasoning'' yields the conclusion that $P$ materially implies $Q$.

In contrast to material implication, which is well-known for confusing beginning students in logic, the concept of \emph{universal quantification} appears to be relatively intuitive. A universally quantified \statement\ is simply one of the form `Every member of category $C$ has property $P$'. In other words, a universally quantified \statement\ is one that talks about all members of a given category simultaneously. As we have argued above, any school of thought except for ultrafinitism eventually yields the conclusion that there are categories for which it does not make sense to talk about all members of the category simultaneously.

Let us call a category \emph{complete} if it makes sense to talk about all members of that category simultaneously. Suppose that $C$ is a complete category, and that there is reasoning yielding the conclusion that the object denoted by `$x$' has property $P$, using the hypothesis that `$x$' denotes a member of $C$. Then for every member $z$ of $C$, it seems that the symbol `$x$' \emph{could} denote the object $z$, in which case our reasoning would yield the conclusion that $z$ has property $P$. Moreover, if every member of $C$ is describable then we can strengthen the case by saying that \emph{we could have made the convention} that `$x$' denotes $z$, rather than talking about the symbol `$x$'\,'s ability to denote $z$ in the abstract. This appears to yield the conclusion that all members of $C$ have property $P$, as long as there was nothing wrong with our reasoning. If we later talk about a concrete member $z$ of $C$, then we can use this fact to conclude that $z$ has property $P$ without reference to our hypothetical reasoning.

Analogously to material implication, the case can be strengthened if the reasoning can be formalized in an axiom system. This time, every claim that we make while hypothetically reasoning is replaced by the claim that the previous claim holds for all members of the given category.

There is another important way to make the notion of ``reasoning well'' mathematically precise: one may say that one is reasoning well if one's reasoning can be formalized in some fixed axiom system. To be precise, a claim $Q$ is said to be \emph{relatively provable from $P$} with respect to an axiom system $A$ if $A$ can prove $Q$ using $P$ as a hypothesis (see Appendix \ref{appendixaxiomatizations} for a precise definition in the framework of this paper). The claim that $Q$ is relatively provable from $P$ is equivalent to the assertion that a certain algorithm halts, namely the algorithm that enumerates proofs of $A$ using $P$ as a hypothesis, and halts if it finds a proof of $Q$. Thus it is uncontroversial (ignoring ultrafinitism) that claims of relative provability are precise claims.

However, it is easy to show that no notion of relative provability can fully capture what we mean by `reasoning well'. Namely, if we think that the axiom system $A$ represents valid reasoning, then we tend to also think that the axiom system $A$ augmented with the inference rule `if $P$ is relatively provable from $Q$ in $A$, and $P$ is true, then $Q$ is true' also represents valid reasoning. But if $A$ is at least as strong as Peano arithmetic, then by G\"odel's incompleteness theorem this new axiom system must be strictly stronger than $A$.

Let us conclude with a simple argument against those who would say that the intuitive notion of ``valid'', ``logical'', or ``unproblematic'' reasoning is precise. Namely, one who holds such a position must answer the question of whether the reasoning process ``$P$ is true and $P$ logically implies $Q$, so $Q$ must be true'' is an example of logical reasoning. If the notion of ``logical'' reasoning is coherent and precise, then it is hard to see what is wrong with this reasoning process, and consequently there is no obvious reason why it should not be called ``logical''. But if we allow this type of reasoning, then we encounter a contradiction upon considering the following variation on the liar's paradoxical sentence: `This sentence logically implies that $0 = 1$'. If the sentence is true, and if logical reasoning is valid, then it appears that $0 = 1$. Moreover, it appears that the previous sentence constitutes an instance of logical reasoning, and thus the sentence logically implies that $0 = 1$. But this is what the sentence itself asserts, a contradiction. It is hard to see what the flaw in this argument could be if it is not the assumption that the notion of logical reasoning is precise.\\

{\it Induction.} The ambiguity inherent in the intuitive notion of ``implication'' becomes especially relevant when considering the mathematical principle of \emph{induction}, which we formulate as follows. Note that we are careful to avoid universally quantifying over the category of numbers, since we do not assume that this category is complete. Even if it is complete, any philosophical position will have to deal with similar issues for its own incomplete categories.

\begin{principle}[Principle of Induction for Numbers\Footnote{There are similar principles of induction for ordinals (for gradualist set realists), truth-values (for logicist set and number realists) and for complete categories (for philosophers who follow \GC). The general rule of thumb is that the most powerful principle of induction is the principle that discusses members of an incomplete category.}]
Let $P$ be a property of numbers, and suppose that there exists valid and precise reasoning yielding that the number $n$ denoted by `$n$' has property $P$, using the hypothesis that all numbers less than $n$ have property $P$, as well as the hypothesis that `$n$' denotes a number. Finally, let $N$ be a number. Then $N$ has property $P$.
\end{principle}
The intuitive explanation for why the principle is legitimate is as follows. Suppose that the hypotheses are true, i.e. that $P$ is a property of numbers, and there is valid and precise reasoning yielding that $n$ has property $P$, using the hypothesis that all numbers less than $n$ have property $P$, where $n$ is a free variable. Then we could make the convention that $n$ denotes $0$; if we do this, then the hypotheses of the reasoning will be satisfied because there are no numbers less than $0$. Since the reasoning is valid and precise, it appears that this implies that $0$ has property $P$. Next we could change conventions so that $n$ denotes $1$; the hypotheses of the reasoning are still satisfied because $0$, the only number less than $1$, has property $P$. We could continue this process indefinitely, eventually getting to the number $N$ and showing that $N$ has property $P$.

As formulated above, the most serious objection to the principle of induction is the fact that the phrase `valid and precise reasoning' cannot be made precise, as we demonstrated above. This objection can be dealt with by fixing a precise notion of ``valid and precise reasoning'' and formalizing the principle of induction for this precise notion. For example, the special case of the principle of induction for material implication and universal quantification suffices to formalize its applications to classical mathematics, since all \statements\ in classical mathematics have definite truth-values. 

However, such an approach is insufficient when dealing with largest number contests, since the (vague) general principle of induction can be used to justify the validity of descriptions of larger numbers than can be justified with only the principle of induction for material implication.

\ignore{
\section{Formalization}
\label{sectionformalization}

Mathematical readers may be getting impatient, not only with the philosophical nature of the reasoning used in the previous sections, but also its lack of \emph{formality}. To \emph{formalize} reasoning is to encode it as a string of symbols in a precise language, as well as to specify an axiom system in that language which endorses (the formalized version of) the reasoning. In this section, we address these concerns by defining precise languages for dealing with the philosophical concepts of meaning and reference.

We begin by defining \emph{gradualist number theory}, the basic framework which we will use for our investigations. Gradualist number theory should be philosophically acceptable to anyone who is not an ultrafinitist.

\begin{definition}
The \emph{language of gradualist number theory \GNT} has three basic types of utterances: \emph{negatable formulas}, \emph{unnegatable formulas}, and \emph{terms}. A negatable formula should be understood as one for which the assertion that the formula is false is precise. Nevertheless, a negatable formula does not necessarily have a definite truth-value.

\GNT\ allows the following syntactical constructions for creating formulas:
\begin{itemize}
\item Bounded quantifiers of both types (existential and universal) are allowed. Applying a bounded quantifier to a negatable formula yields a negatable formula; applying a bounded quantifier to an unnegatable formula yields an unnegatable formula.
\item The `and' and `or' operations are allowed, with the same restrictions as above.
\item The `not' operation can be applied to a negatable formula, yielding another negatable formula.
\item Unbounded existential quantifiers are allowed, but the resulting formula is unnegatable.
\item If $x$ is a term, then `$x$ exists' is an unnegatable formula.
\end{itemize}
\GNT\ allows the following syntactical constructions for creating terms:
\begin{itemize}
\item The numerals `$0\RQ,\LQ 1$' and the standard arithmetic operations of addition, multiplication, and exponentiation.\Footnote{The inclusion of exponentiation makes it much easier to formalize other areas of ``finite'' mathematics such as finite set theory and computability theory within number theory. \ego\ am not sure whether \GNT\ without exponentiation is capable of formalizing these fields.}
\item The term `$\fst\{n : \phi(n)\}$', where $\phi$ is a negatable formula with $n$ as a free variable. This term should be read ``the first $n$ such that $\phi(n)$ is true''.
\end{itemize}
The semantics of \GNT\ can be understood algorithmically in the sense that there is a single algorithm that inputs a \statement\ or term of \GNT\ and returns its truth-value or referent, respectively, if it exists, and otherwise fails to halt. The algorithm is defined recursively as follows:
\begin{itemize}
\item To compute the truth-value of a \statement\ of the form $\Phi = \LQ \forall k \leq t \; \phi(k)$', first compute the referent of the term $t$, call it $\ell$, and then compute the truth-values of the \statements\ $\phi(1),\ldots,\phi(\ell)$ in parallel. If any of the computations return the value `false', return `false' as the truth-value of $\Phi$. If all of the computations return the value `true', return `true' as the truth-value of $\Phi$.
\item Apply similar rules with `and', `or', and `not'.
\item To compute the truth-value of a \statement\ of the form $\Phi = \LQ \exists k \; \phi(k)$', compute the truth-values of the \statements\ $\phi(1),\phi(2),\ldots$ in parallel. (The number of \statements\ that we attempt to compute simultaneously increases throughout time via standard techniques.) If any of the computations returns the value `true', return `true' as the truth-value of $\Phi$. There are no circumstances under which we return `false' as the truth-value.
\item To compute the truth-value of a \statement\ of the form $\Phi=$`$t$ exists', compute the referent of the term $t$. If this computation halts, return `true' as the truth-value of $\Phi$. There are no circumstances under which we return `false' as the truth-value.
\item The numerals and standard arithmetic operations are computed normally. To compute the referent of the term `$\fst\{k : \phi(k)\}$', compute the truth-values of the \statements\ $\phi(1),\phi(2),\ldots$ in succession (not in parallel). Return the first $k$ such that the truth-value of $\phi(k)$ is computed as `true' as the referent of the term. Note that the truth-value of $\phi(k)$ will never be computed unless the truth-values of $\phi(1),\ldots,\phi(k - 1)$ have all been computed as `false'. 
\end{itemize}
Note that according to these semantics, a \statement\ containing a term without a referent does not have a truth-value unless the truth-value can be computed without reference to the referent of the term. In particular, this phenomenon may prevent a negatable \statement\ from having a definite truth-value.
\end{definition}

\begin{remark*}
\GNT\ is similar to the language $\Sigma_1(\NT)$ (``the $\Sigma_1$ sentences of number theory'') which is well-known to mathematical logicians. However, it differs from it in that $\Sigma_1(\NT)$ has no term construction rule analogous to the $\LQ\fst$' term construction rule. This means that \GNT\ has a larger class of negatable formulas than $\Sigma_1(\NT)$.
\end{remark*}

To illustrate the power of \GNT, we prove the following result:
\begin{theorem}
\label{theoremGNTalgorithm}
\GNT\ has a term formalizing the phrase `the output of the algorithm whose G\"odel number is $\alpha$', where $\alpha$ is a free variable.
\end{theorem}
\begin{proof}
This is mostly standard, but we highlight the first step in the proof where exponentiation is needed, as well as the last two steps where the $\LQ\fst$' operator is needed. First of all, note that the predicate
\[
\phi(n,m) = \LQ \exists k \leq m \; \exists 0 \leq \ell < 2^n \;\;\; m = k\cdot 2^{n + 1} + 2^n + \ell\RQ
\]
formalizes the phrase `the $n$th binary digit of $m$ is a 1'. This predicate can be thought of as a ``membership'' relation, allowing gradualist finite set theory to be encoded into gradualist number theory. Standard techniques now show that the phrase ``$\alpha$ is the G\"odel number of an algorithm that terminates in $t$ steps or fewer and returns the number $k$'' can be encoded as a negatable predicate $\phi(\alpha,k,t)$. Moreover, $\phi(\alpha,k,t)$ is always false unless $k < t$. Thus the term
\[
T(\alpha) = \LQ \fst\{t : \exists k < t \; \phi(\alpha,k,t)\}\RQ
\]
encodes the phrase `the amount of time it takes for the algorithm $\alpha$ to halt'. Finally, the term
\[
\LQ\fst\{k : \phi(\alpha,k,T(\alpha))\}\RQ
\]
encodes the phrase `the output of the algorithm whose G\"odel number is $\alpha$'.
\end{proof}

Combining Theorem \ref{theoremGNTalgorithm} with our algorithmic description of the semantics of \GNT\ shows that \GNT\ can formalize its own truth predicate; if $\alpha$ is the algorithm computes the truth-value of any \statement\ of \GNT, then the phrase `$\phi$ is a true \statement\ of \GNT' can be formalized in \GNT\ as the phrase `the algorithm $\alpha$ returns `true' on the input $\phi$'.

Let us now address the question of how to axiomatize \GNT. We will denote by \GNT$_0$ the axiomatization of \GNT\ consisting of the following axioms and inference rules:
\begin{itemize}
\item There are axiom schemas corresponding to the standard arithmetic identities, such as $x(y+z)=x\cdot y + x\cdot z$. Each of these axiom schemas should be understood as an inference rule: if $x,y,z$ are terms such that `$x$ exists', `$y$ exists', and `$z$ exists' are believed, we are allowed to infer the corresponding identity.
\item If $\phi(1),\ldots,\phi(\ell - 1)$ are all believed to be true, we are allowed to infer $\forall k < \ell \; \phi(k)$. Conversely, if $\forall k < \ell \; \phi(k)$ has been proven we are allowed to infer $\phi(k)$ for any $k < \ell$.
\item If $\phi(k)$ is believed to be true, we are allowed to infer $\exists k \; \phi(k)$.
\item Apply similar rules with `and' and `or'.
\item De Morgan's laws constitute rules of inference for `not'.
\item If $\forall k < \ell \;\neg \phi(k)$ and $\phi(\ell)$ are believed to be true, we are allowed to infer `$\fst\{k : \phi(k)\} = \ell$' and `$\fst\{k : \phi(k)\} \text{ exists}$'.
\end{itemize}

Note that this axiom system is very weak: it does not have any induction rule. The reason for this is that the ``material implication'' induction can't be used here because there's no universal quantification over \N. Nevertheless, \GNT$_0$\ is capable of proving all true formulas of \GNT.

To add an induction axiom, we observe that the concept of relative provability in \GNT$_0$\ can be formalized in \GNT. We denote the \statement\ `$\phi$ is provable relative to $\psi$ in \GNT$_0$' by $\LQ\phi\Rightarrow_0 \psi$'. Now \GNT$_1$ is \GNT$_0$\ together with the following induction rule:
\begin{itemize}
\item If `$\phi(n) \Rightarrow_0 \phi(n + 1)$' is believed to be true, where $n$ is a numeral representing a free variable, then we can infer $\phi(n)$ for any number $n$.
\end{itemize}
Obviously, if we have any notion of inference we can replace the induction rule by the rule corresponding to that notion of inference. For example, if we denote relative provability in \GNT$_1$ by $\Rightarrow_1$, then we will denote the corresponding axiom system by \GNT$_2$.

}
\ignore{

\section{The relation between \GDST\ and \C(\TC)}
\label{appendixKP}

In this appendix we prove the following claim made in \6\ref{subsectionDST}:
\begin{theorem}
\label{theoremGDSTCTC}
Suppose that there exists a model of \KPC\ in which there exists an uncountable (relative to the model) set. Then \GDST\ is capable of reproducing $\C(\TC)$.
\end{theorem}

Here \KPC\ denotes Kripke--Platek set theory with choice.

\begin{lemma}[\KPC]
Suppose that there exists an uncountable set. Then if $\aleph_1$ is the first uncountable ordinal, then $\LL \df L_{\aleph_1}$ is a model of \KPC.
\end{lemma}
\begin{proof}
The main thing we need to check is restricted replacement. So let $a\in \LL$, and let $\phi$ be a formula with bounded quantifiers including constants $c_1,\ldots,c_n\in \LL$ such that for all $x\in a$, there exists a unique $f(x)\in \LL$ such that $\phi(x,f(x),c_1,\ldots,c_n)$ holds. We need to show that there exists $b\in \LL$ such that $b = \{f(x) : x\in a\}$.

First, by the axiom of restricted replacement, there exists a set $b$ such that $b = \{f(x) : x\in a\}$. Here, we can apply restricted replacement because the notion of truth for the sentence $\phi(x,f(x),c_1,\ldots,c_n)$ with respect to the model $\LL$ can be defined using bounded quantifiers. To complete the proof, we need to show that $b\in \LL$.

Since $a\in \LL$, $a$ is countable, and thus $b$ is countable as well. Moreover, every element of $b$ is a member of $\LL$. For each $y\in b$, let $\gamma(y)$ be the smallest ordinal such that $y\in L_{\gamma(y)}$. By the axioms of union and restricted replacement, the ordinal $\alpha = \bigcup_{y\in b} \gamma(y)$ exists. Since $\alpha$ is the union of countably many countable ordinals, by the axiom of choice $\alpha$ is countable. Thus $\alpha < \aleph_1$ and in particular $\alpha+1 < \aleph_1$. Now since $b \subset L_\alpha$, we have $b = \{y\in L_\alpha : \exists x\in a \; f(x) = y\} \in L_{\alpha+1} \subset \LL$. This completes the proof.
\end{proof}

Let $\SS$ be the sentence ``there exists an uncountable set'', or equivalently ``$\LL$ exists''.
\begin{lemma}[\KP]
Suppose that there exists a model of \KPC\ satisfying $\SS$. Then there exists a model of \KPC\ satisfying ``$\LL$ exists but does not satisfy $\SS$''. This model can be chosen in a canonical manner.
\end{lemma}
\begin{proof}
First we prove the lemma in \KPC\ under the hypothesis that $\LL$ exists and satisfies $\SS$. In this case, there exists an admissible ordinal $\alpha$ such that $L_\alpha$ satisfies $\SS$. Let $\alpha_0$ be the least such ordinal. Then $M = L_{\alpha_0}$ is a model of \KPC. Since $M$ satisfies $\SS$, there exists an ordinal $\beta < \alpha_0$ such that $M$ satisfies ``$\LL = L_\beta$''. But by the minimality of $\alpha$, $L_\beta$ does not satisfy $\SS$, and thus $M$ satisfies ``$\LL$ does not satisfy $\SS$''.

Now suppose that $M_0$ is a model of \KPC\ satisfying $\SS$. If $M_0$ satisfies ``$\LL$ exists but does not satisfy $\SS$'', then the conclusion of the lemma holds automatically. Otherwise, $M_0$ satisfies ``$\LL$ exists and satisfies $\SS$'', so the previous paragraph shows that the conclusion of the lemma holds. This completes the proof.
\end{proof}

\begin{lemma}[\KPC]
\label{lemmaclubset}
Suppose that $\LL$ exists. Then there exists a club (closed and unbounded) set $F\subset \aleph_1$ such that for all $\beta\in F$, $\ST_\beta = \ST_{\aleph_1}$. In particular, there exist $\beta,\gamma\in \aleph_1$ distinct such that $\ST_\beta = \ST_\gamma$.
\end{lemma}
Here we use the notation
\[
\ST_\beta = \{\phi \in \ST : \text{$\phi^\beta$ is true}\},
\]
where $\phi^\beta$ denotes the result of replacing all unbounded quantifiers in $\phi$ with quantifiers that range over $L_\beta$. Equivalently, $\phi\in\ST_\beta$ if $\phi$ is true of $L_\beta$.
\begin{proof}
One proves by induction that for every formula $\phi\in \ST_{\aleph_1}$ with constants $c_1,\ldots,c_n\in \LL$, there exists a club set $F_\phi \subset \aleph_1$ such that for all $\beta\in F_\phi$, $\phi^\beta(c_1,\ldots,c_n)$ holds. The most difficult case is where $\phi = \forall x \; \psi^\beta(c_1,\ldots,c_n,x)$ where $\psi$ satisfies the inductive hypothesis. Then for each $x\in \LL$ there exists a club set $F_x\subset \aleph_1$ such that for all $\beta\in F_x$, $\psi^\beta(c_1,\ldots,c_n,x)$ holds. Now let $G_x = F_x \cup\bigcup\{\beta\in\aleph_1 : x\notin L_\beta\}$. Then $G_x$ is closed and thus so is $F \df \bigcap_{x\in \LL} G_x$. To show that $F$ is unbounded, fix $\beta_0 < \aleph_1$ and for each $n$, if $\beta_n$ is defined then let $F_n = \CO{\beta_n}{\infty}\cap \bigcap_{x\in L_{\beta_n}} G_x$ and $\beta_{n+1} = \inf(F_n)$. Since the collection of club sets is closed under countable intersection, $F_n$ is a club set and thus $\beta_{n+1}$ is well-defined. Now let $F_\omega = \bigcap_n F_n$ and $\beta = \lim_{n\to\infty} \beta_n$. For all $m,n$ with $m\geq n$, we have $\beta_m \in F_n$. Since $F_n$ is closed, taking the limit yields $\beta\in F_n$, and taking the intersection we have $\beta\in F_\omega = \CO\beta\infty\cap \bigcap_{x\in L_\beta} G_x$. Now for all $x\in \LL\butnot L_\beta$, by the definition of $G_x$ we have $\beta\in G_x$,so $\beta\in G_x$ for all $x\in\LL$. Thus $\beta_0 \leq \beta \in F$, so $F$ is unbounded and thus club.
\end{proof}

\begin{lemma}[\KPC]
Suppose that $\LL$ exists. Then every transfinite algorithm either halts or loops by a countable ordinal.
\end{lemma}
\begin{proof}
The current state of a transfinite algorithm at time $\alpha$ can be defined in $\ST_\alpha$ in a uniform way. Since by Lemma \ref{lemmaclubset} there exist distinct countable ordinals $\beta,\gamma$ such that $\ST_\beta = \ST_\gamma$, it follows that the state of any transfinite algorithm must be equal at those two ordinals. If the algorithm has not halted yet, this means that it will loop forever.
\end{proof}

\begin{proof}[Proof of Theorem \ref{theoremGDSTCTC}]
Let $M$ be a model of \KP\ in which there exists an uncountable (relative to the model) ordinal. Without loss of generality we may suppose that $M = L_\alpha$ for some $\alpha$. By the previous lemma, in $M$ every transfinite algorithm either halts or loops by an ordinal $<\alpha$. So we can tell whether or not any given transfininte computation halts by asking whether or not it halts by stage $\alpha$.
\end{proof}

}

\ignore{

\begin{theorem}[\GKP]
\label{theoremtransdescr}
Suppose that there is a model of \KPC\ satisfying $\SS$. Then transfinite computation is descriptionalist in the following loose sense: If $\alpha$ is a transfinite computation, then there is a canonical way of labelling the stages of the computation.
\end{theorem}
\begin{proof}
By previous lemmas, there is a canonical minimal model of \KPC\ satisfying ``$\LL$ exists but does not satisfy $\SS$''. In this model, every transfinite algorithm either halts or loops by a countable ordinal, and also every countable ordinal is canonically countable (since $\LL$ satisfies that it is countable). So we can take a canonical counting of the stages of the computation.
\end{proof}
}

\section{More on the Busy Beaver contest}
\label{appendixmoreBBC}

In the Introduction we saw that the axiomatic strategy for the Busy Beaver contest is relatively epistemologically optimal. But is there a truly optimal strategy? Of course, if we fix a bound on the size of the entries and assume the Principle of Limited Omniscience, then one of the possible entries must in fact be the highest scoring one. Can we know what it is? The following theorem shows that we can know some nontrivial facts about the winning entry, while at the same time proving that the entry itself is elusive:

\begin{theorem}
\label{theoremBBCincompressible}
Let $\omega$ be the highest scoring entry to the Busy Beaver contest among all entries of length at most $n$. Then the Kolmogorov complexity of $\omega$ is at least $n - C$, where $C$ is an absolute constant. In particular, the length of $\omega$ is at least $n - C$, and $\omega$ is $C$-incompressible.
\end{theorem}
Here, we recall that the \emph{Kolmogorov complexity} of a finite string $\omega$ is the length of the shortest program that computes $\omega$. A string is called \emph{$C$-incompressible} if its length is at most $C$ plus its Kolmogorov complexity.
\begin{proof}
Let $\alpha$ be the shortest program that computes $\omega$, and consider the algorithm $\beta=$`Let $\omega$ be the output of $\alpha$, then let $N$ be the output of $\omega$, then return $N+1$'. Then the length of $\beta$ is equal to the length of $\alpha$ plus some (small) constant $C$. Now, $\beta$ is a valid entry to the Busy Beaver contest whose score is greater than the score of $\omega$. If the length of $\beta$ is at most $n$, this contradicts the definition of $\omega$, so the length of $\beta$ is at least $n$. Thus, the length of $\alpha$ is at least $n - C$. Since $\alpha$ was arbitrary, this shows that the Kolmogorov complexity of $\omega$ is at least $n - C$. The remaining two assertions follow from the fact that the length of $\omega$ is at least its Kolmogorov complexity and at most $n$.
\end{proof}

Thus, the Busy Beaver Contest shares with the contests $\LNC(\ST)$ and $\LNC(\NT)$ the property that the highest scoring entry has close to the maximum allowed length. However, while Theorem \ref{theoremrecstrat} shows that there is an algorithm for computing nearly optimal entries to $\LNC(\ST)$ and $\LNC(\NT)$ as a function of the maximum allowed length, Theorem \ref{theoremBBCincompressible} shows that there is no such algorithm in the case of the Busy Beaver contest.

One might ask: even if there is no algorithm for computing the winner of the Busy Beaver contest, maybe we can still figure it out using some other method of deduction? Perhaps as the allowed length of the entries increases, we can spend more time thinking about how to determine what is the highest scoring entry. However, this idea seems to be based on a misunderstanding of the nature of deductive reasoning. If we specify an axiom system in which our reasoning can be encoded, then any entry we can prove is valid will be dominated by the axiomatic strategy corresponding to our axiom system, and is therefore not optimal.

\section{Contra the intersective definition}
\label{appendixintersectivedefinition}

It is worth taking some time to correct a common misunderstanding regarding the conceptual status of the notion of a category generated by production rules. This misunderstanding arises from the way in which this notion is formalized in the language of set theory. Namely, if $S$ is a set, then the subset of $S$ generated by a given list of production rules $R$ is defined to be the intersection of all subsets of $S$ that are closed under the production rules of $R$.\Footnote{A set is \emph{closed} under a list of production rules if it is not possible to add new members to the set by applying the production rules.} It seems to be a common belief that this ``intersective'' definition is somehow ``more rigorous'' or ``more precise'' than the intuitive notion of a category generated by production rules. Even worse, some people may even think that the intersective definition is \emph{conceptually accurate}, i.e. that it somehow describes \emph{what we really mean} when we talk intuitively about categories generated by production rules. A few points should suffice to counter this view:
\begin{itemize}
\item Regarding one definition being ``more precise'' than another, it appears that the only reason people think this is because the concept that the intersective definition depends on, i.e. the concept of a set, has a well-known standard axiomatization, whereas the natural axiomatizations (in various contexts) of the concept of a definition by production rules are less well-known. But there is nothing unrigorous about these other axiomatizations, nor are they at all arbitrary.
\item A serious threat to the conceptual accuracy of the intersective definition comes from the proof that the intersective definition is equivalent to the intuitive notion (in cases where the intersective definition applies). Namely, this proof depends on the conceptual validity of the intuitive notion in a crucial way: to show that every element of the set given by the intersective definition is in fact eventually generated by the production rules, one must use the fact that there is a \emph{set} consisting of all objects eventually generated by the production rules. (This set is a member of the collection of subsets being intersected, so it contains the set given by the intersective definition.) It seems to \me\ that most other proofs showing that two given notions are the same do not use the conceptual validity of the notions being proven equivalent in the same way (for example, the above proof does not appear to \me\ to depend on the conceptual validity of the intersective definition).
\item An obvious threat to the conceptual accuracy of the intersective definition is that it seems to be perfectly coherent to deny the validity of talking about sets at all while still accepting the legitimacy of talking about categories generated by production rules. In fact, all number realists seem to do this, since the category of natural numbers is ordinarily understood as being generated by production rules (namely $0$ and the successor operation).
\item On the other hand, even set realists appear to accept the validity of at least some recursive definitions that cannot be restated in terms of the intersective definition. Namely, many set realists accept the language of classical set theory as precise, even though the recursive definition of truth for classical set theory does not make sense according to an intersective concept of recursion.\Footnote{That is, unless we accept the notion of ``proper classes'' as legitimate and on a par with the notion of sets. But few set realists appear willing to do this, for reasons discussed in Section \ref{sectionsetrealism}.} While there may be good reasons for set realists to deny that the language of classical set theory is precise, it seems that the recursive nature of its notion of truth is not one of them.
\end{itemize}
Thus, \ego\ will take it as a given that definitions of categories in terms of production rules are precise, and \ego\ will not argue the point further.

\ignore{
\section{The theory of sets and classes}

A naive argument in favor of going farther is based on the theory of sets and classes, which we will denote by \SCT, also known as the second-order theory of sets. In this theory, one can quantify not only over sets but over ``classes'' of sets. The standard axiomatization of \SCT\ is \NBG, von Neumann--Bernays--Godel set theory.

To many mathematical logicians, \SCT\ may at first appear to have greater expressive power than \K(\ST). This is because the category of true \statements\ of \K(\ST)\Footnote{More precisely, the category of pairs $(\phi,\lv\xx=\aa\rv)$ where $\phi$ is a \statement\ of $\K(\ST)$ which is true relative to the variable assignment $\lv\xx=\aa\rv$.} is a subset of the category of sets that can be defined in terms of production rules, so it seems that the notion of truth in \K(\ST) can be expressed in \SCT\ via the ``intersective definition'' of categories defined by production rules (see Appendix \ref{appendixintersectivedefinition}). However, whether this is true or not depends on which of two possible semantics we give to the language \SCT.

{\bf Option 2:} to articulate a principled difference between the notion of a class and the notion of a set that justifies the conclusion that the language of set theory has different semantics when applied to classes than when applied to sets. One way to try to do this is to claim that the classes are ``logical'' objects while sets are ``combinatorial'' ones, see e.g. \cite{Reinhardt}. In other words, a class is just a ``logical set'', i.e. a description of a property abstracted away from the details of the description. But in order to talk about properties of sets, we need a notion of truth in \ST, and the need to define truth is exactly the reason why we felt the need to resort to classes in the first place. So it appears that this option is first of all incoherent, and in the end circles back to Option 1.

}

\ignore{
\section{Miscellaneous}

[Note: The category-creating operator in \GC\ is not quantifier-like, it is more powerful (because it includes a second-order free variable). Similarly for the standard definition of recursion. This means that the resulting languages are \emph{NOT} proposition/term lanugages.\internal]

[From after Principle \ref{principleGC}\internal]

From a set realist perspective, it is worth asking whether this principle justifies the conclusion that \ST\ is precise in the same way as for \NT. In order for this to be the case, the universe of sets would have to be generated by production rules, but the standard set realist view is that it is not possible to generate the universe of sets in this way. The same holds for the earlier phrasing of the intuition in terms of ``precisely delimited'' categories, since the universe of sets is not supposed to be delimited either. This is a hint that perhaps a strong view of the largeness or indescribability of the universe of sets should lead one to believe that universal quantification over the whole universe of sets does not make sense. We will explain this view further in Section \ref{sectionsetrealism}.

As a side remark, an interesting consequence of Principle \ref{principleGC} is that it implies that every ordering of \N\ either is or is not a well-ordering. Here, an \emph{ordering} of \N\ is a convention for labelling some members of \N\ as ``larger'' than others in a way that does not necessarily correspond to their magnitude, but nevertheless obeys certain natural algebraic rules. An ordering is called a \emph{well-ordering} if the following production rule completely generates \N: whenever all numbers labelled ``strictly less than'' a given number $k$ have been added to a set $S$, then add $k$ to $S$. For example, the standard ordering (i.e. calling ``larger'' whichever numbers are larger in terms of magnitude) is a well-ordering, since first $0$ would be added, then $1$, and so on. It appears difficult or impossible to justify the principle that every ordering either is or is not a well-ordering using only the special case of Principle \ref{principleGC} that occurs when $C=\N$, even for the case of computable orderings.

[From Set Realism\internal]

It is true that we can get some set theory out of the concept of lists by simply using the word `set' instead of `list' and `equal' instead of `equivalent'. This idea can be generalized to infinite sets by considering a description of a category of objects to be a ``set'', and considering two such descriptions to be ``equal'' if they identify the same objects. However, this view of set theory is very different from the classical view, since it implies that

The issue of whether collections are pre-existing objects or are defined in terms of more primitive objects such as lists is important for several reasons. First of all, according the classical view there exist sets which no description can ever be precise enough to uniquely identify.\Footnote{The standard heuristic used to support this claim is that we can uniquely identify only ``countably many'' sets, whereas there are ``uncountably many'' sets overall. There are many problems with this heuristic, in particular its implicit claim that it makes sense to talk simultaneously about everything that we will ever be able to uniquely identify, and its reliance on the standard view that the intuitive notion of ``size'' is best generalized to infinite sets by Cantor's embeddability criterion. The argument cannot be formalized in \ZFC\ due to \ZFC's inability to formalize the notion of an identifiable set, but many set realists see \ZFC\ as ``morally'' implying the existence of unidentifiable sets, in the sense that the intuitions that would lead one to accept \ZFC\ would also lead one to accept that there are undescribable sets.} If a set is nothing more than a description of a criterion or standard for membership, abstracted away from the details of the description (an entity which we will call a \emph{logical set}), then this classical view cannot be correct.

Secondly, the issue is relevant to the question of whether it makes sense to quantify over all sets, or even over all subsets of a given set. If sets are defined in terms of descriptions of criteria, then the question of what sets there are depends on what language the descriptions are supposed to be written in. But this means that we cannot precisely talk about sets until we specify a precise language for writing the descriptions. This conflicts with the classical view that it is possible to identify a set by giving a membership criterion that quantifies over all sets. Namely, applying this view to logical sets would yield the conclusion that the descriptions are supposed to be written in the language of set theory, but this would mean that we cannot precisely talk about sets until the language of set theory is precise, which will not happen until we can precisely talk about sets. This ``vicious circle'' is the reason that Henri Poincar\'e called classical definitions that quantify over all sets ``impredicative'' and therefore invalid.

...under the intuitive notion of a set.\Footnote{By contrast, if a set is conceived of as a description of a standard for membership, then the resolution of the paradox is that the standard of admitting as members ``all criteria that are not satisfied by themselves'' is not a very precise standard, because some criteria are not precise enough for there to be a well-defined answer to the question of whether or not they are satisfied by themselves. If we restrict our attention to those criteria that are precise enough, then the problem becomes the fact that the question of whether a criterion is precise enough or not is not itself a precise question. This reasoning will be formalized in Section \ref{sectionmetasq}.}

Let us consider one last criticism, not of the idea that \ST\ has precise semantics, but rather of the idea that \ST\ is capable of formalizing all mathematical thought. Namely, \ego\ claim that \emph{\ST\ has no conceptually accurate way of defining the set of natural numbers $\N$}. This may seem surprising to many mathematicians, since \ST\ can define the set of natural numbers as the intersection of all sets that contain $0$ and are closed under the successor operation $k\mapsto k + 1$. ($0$ and the successor operation are conventionally identified with the empty set and a certain definable operation on sets, respectively.) However, \ego\ argue that this definition does not adequately formalize the intuitive idea of ``the set that you get by starting with $0$ and applying the rule that whenever $k$ is in the set, you add $k+1$ to the set''. In other words, defining a set as an \emph{intersection} of other sets has a very different flavor from defining it in terms of \emph{production rules}. Ordinarily this incongruity is justified by saying that the two definitions both define the same set of natural numbers. This would be fine in most circumstances, but the issue here is that in order to prove that the two definitions both define the same set, \emph{the conceptual validity of the definition in terms of production rules is used in a crucial way}. Namely, the proof that the set defined in terms of intersection is contained in the set defined in terms of production rules proceeds via the observation that the set defined in terms of production rules \emph{exists as a set} (and in particular is a member of the class being intersected), and therefore contains the set defined in terms of intersection. It seems to \me\ that most other proofs showing that two given definitions define the same concept do not use the conceptual validity of the definitions being proven equivalent in the same way (for example, the above proof does not appear to us to depend on the conceptual validity of the definition in terms of intersection), so \ego\ say that \ST\ has no \emph{conceptually} accurate way of defining the set of natural numbers.\Footnote{For readers coming from mathematical logic, \ego\ should make clear that this criticism has nothing to do with the fact that \ST\ has no way of defining the natural numbers which is accurate in every model of \ZFC; the same is true for any axiomatization of the definition in terms of production rules. Our point is at the level of languages and not at the level of axiom systems.}

{\bf Option 4:} to point out that the new language is different from the old language in the sense that the new language has a predicate expressing the concept of being a member of the original universe, whereas the old language has no such predicate. Thus the fact that the brute force strategy of the new language yields a larger number than the brute force strategy of the old language is not necessarily a contradiction. The predicate expressing the concept of being a member of the original universe can be neatly written as $\LQ\in V_{\kappa_0}$', where $\kappa_0$ denotes the cardinality of the original universe (which is a member of the new univese but not of the old universe). But then we can apply the same logic to get a new predicate describing a universe of cardinality $\kappa_1 > \kappa_0$, and so on. Since we can continue to define larger and larger universes, it is not clear how this new picture differs from the gradualist picture except that there is a new primitive notion of ``being the cardinality of a universe''. Since it is not clear what the intuitive meaning of this primitive notion is, it is difficult to justify the claim that it can be the subject of a precise language, and thus it seems that this option circles back to Option 3.\\

{\bf Alternatives.} There may be other philosophically viable options as well, but they will have to be articulated by their supporters. It seems to \me\ that after ``pinning down'' clearly enough what any given philosophy believes, one will always end up with a Philosophical Largest Number Contest that can be formalized. This can only be illustrated via examples, and hopefully the above examples as well as the ones which will appear later in the paper constitute an at least somewhat plausible proof-of-principle.

}

\section{Glossary}
\label{sectionglossary}

As the subject matter of this paper has been the concepts of language, truth, reference, and preciseness, it has been particularly important to use precise language regarding these concepts. Below, I describe to some extent how I have been using various words.

\begin{itemize}
\item A \emph{language} is a convention specifying methods to be used for expressing and communicating thought. Languages can be either spoken or written (or they may use some other medium), but in this paper we mostly concentrate on written languages, which stipulate that thoughts are to be expressed by writing down strings. A \emph{string} consists of \emph{symbols} (marks which can be reliably distinguished from each other) arranged sequentially. Words, phrases, sentences, and mathematical formulas and expressions are all examples of strings.\Footnote{To guarantee that the symbols occurring in mathematical formulas have a canonical sequential arrangement, it may be best to think of them as being specified by a sequence of \LaTeX\ commands.}

This non-technical use of the word `language' should be contrasted with a different usage common in mathematical logic: the word `language' is often used to mean ``a set of strings''. While languages in the sense of this paper are usually associated with sets of strings of symbols, they are not fully described by them in the sense that a language always includes conventions for how it should be \emph{used} (and not merely mentioned). These conventions are typically of the form ``A certain mental concept should be expressed in the language in a certain way''. Obviously, the notion of a language in the sense of this paper cannot be made mathematically precise.

It is worth noting that although axiom systems are in some sense conventions for how a language should be used, a set of strings together with a formal axiom system still does not count as a language in the sense of this paper, since the conventions do not specify methods for communicating thought, but rather methods for manipulating symbols. On the other hand, if we start with a language and then add an axiom system, then the axiom system can sometimes be thought of as being a part of the language in the sense that it corresponds to the convention that the language should be interpreted in a way such that the axiom system is valid, i.e. we must determine whether or not any given way of interpreting the language is compatible with the axiom system, and reject those interpretations that are incompatible. Such a convention may increase the preciseness of a language which was previously imprecise, though it may also simply be an expression of intuition about concepts that were already precise.
\item Languages generally specify that their users should communicate thought by making \emph{utterances}, which can then be thought of as objects in their own right. In written languages, the utterances consist of strings of symbols. For the purposes of this paper we do not distinguish between ``complete'' utterances (such as sentences and mathematical formulae) and ``partial'' utterances (such as words, phrases, and mathematical expressions). We do, however, group complete and partial utterances into the following categories:
\begin{itemize}
\item A \emph{term}, \emph{name}, or \emph{representation} is the kind of utterance that \emph{refers to}, \emph{represents}, \emph{denotes}, or \emph{names} an \emph{object}, which is called its \emph{referent} or \emph{value}. In English, terms are generally singular noun phrases including determiners or proper nouns, such as `the largest number' or `Alice'. I do not consider plural nouns and noun phrases (such as `numbers' and `Alice and Bob') to be terms, nor singular nouns without determiners such as `number'. Mathematical terms are often called ``expressions'', but I will avoid this word both because it is often used to mean ``figure of speech'', and because it suggests a connection with the word `express', a connection which I find misleading.
\item A \emph{\statement} or \emph{claim} is the kind of utterance that can be \emph{true} or \emph{false}. The \emph{truth-value} of a claim is the answer to the question of whether the claim is true, if such an answer exists.
\item A \emph{predicate} is the kind of utterance that can be naturally combined with a term to form a \statement.
\item A \emph{definition} is the kind of utterance that \emph{defines} a term or class of terms in the sense of either specifying a new convention or describing a pre-existing convention for interpreting those terms. A definition is said to be \emph{conceptually valid} to the extent that it accurately describes a pre-existing convention.
\item A \emph{description} is the kind of utterance that \emph{describes}, or expresses thoughts about, an object. Descriptions may be more or less precise; particularly precise descriptions are said to \emph{identify} the object they describe. Terms are one kind of description but not the only kind. In particular, for a general description it does not make sense to talk about whether it refers to, represents, denotes, or names an object.

Note: any description can be turned into a name via the phrase `the [entity] described by [that description]'.
\item ``talk about'', ``call'', ``is called'', ``is said to be''

\end{itemize}
\item A \emph{notion} or \emph{concept} is a mental image or intuitive picture. Languages can \emph{express} concepts but cannot fully communicate them.
\item A \emph{category} is a mental ``box'' into which objects can be imaginarily placed. One can \emph{define} a category by specifying conventions regarding which objects should be placed in such a box, though this use of the word `define' should be carefully distinguished from the use of the word `define' for terms. For example, the ``category of natural numbers'' is defined by the convention that natural numbers should be placed in the box. An object which users of the language should place in the box (according to convention) is said to be a \emph{member} of the category.
\item A \emph{collection} is a category whose members are required to be members of a second ``reference'' category (such as the category of natural numbers).
\item The word `set' will be used in this paper in two different ways:
\begin{itemize}
\item (a) a set is a pre-existing mind-independent object which may contain other objects as members;
\item (b) a set is a collection (as defined above) which is defined precisely enough that every member of the reference category either is or is not a member of the collection.
\end{itemize}
\item A \emph{class} is a collection of sets according to usage (a), i.e. a collection whose reference category is the category of sets. This matches the standard meaning of the word `class' in mathematical logic.

Note that categories, collections, classes, and sets according to usage (b) are all mind-dependent objects, whereas sets according to usage (a) are mind-independent objects.
\item A person may \emph{translate} the utterances of one language into another language, by making new utterances which express the same concepts. In this case, the new utterance is called a \emph{translation} or \emph{transcription} of the original utterance, and is said to \emph{transcribe} and \emph{reproduce} (but not translate) the original utterance. Metonymically, the language of the new utterance is also said to transcribe and reproduce the language of the original utterance.

\item An utterance is \emph{precise} to the extent that it is incapable of being misinterpreted or reinterpreted by thinkers following the conventions of the language it is embedded in. This definition is somewhat circular since the answer to the question of whether thinkers are in fact following the conventions of a language depends on how well-defined or precise those conventions are. Nevertheless, a fundamental human drive is to clarify conventions in order to make them more precise, and preciseness can be understood as the extent to which we have succeeded at this goal.

The concept of preciseness is in general too vague to be used without an additional qualifier, such as a qualifier describing what purpose the utterance is precise enough for. Some of the qualified notions of preciseness we have used in this paper are `precise enough to have a definite truth-value' (also called \emph{definite}), `precise enough to apply induction', `precise enough that we can say exactly \emph{what it would mean} for the claim to be true' (also called simply \emph{precise}).
\item A \statement\ is called \emph{intelligible} or \emph{coherent} (or \emph{meaningful}?) if it is precise enough that we can say at least to some extent what it would mean for the \statement\ to be true, even if we cannot do so exactly.
\item A thinker \emph{justifies} a claim by describing a reasoning process that might lead a reader to reason in that way and thus \emph{endorse} the justification and consequently believe the claim.
\item A \emph{method} of reasoning is a strategy for reasoning, while a \emph{mode} of reasoning is only a way in which people reason, regardless of whether or not it follows a definite strategy.
\item A thinker \emph{gives}, \emph{writes}, \emph{formulates}, or \emph{states} a claim or chain of reasoning by uttering it, while he \emph{formalizes} or \emph{codifies} a claim or chain of reasoning by formulating a claim or chain of reasoning in a precise language which captures the intuitive meaning of the original claim or chain of reasoning. Claims and chains of reasoning in precise languages can be translated into other precise languages. However, we will not call this `formalization' because it is not making it more formal, it is preserving the degree of formality. Instead, we will call it \emph{reproduction} or \emph{transcription}.

\item A person \emph{gives}, \emph{writes}, or \emph{encodes} an \emph{algorithm} by describing it in sufficient detail that it could be implemented.
\item I follow the standard convention that if $\phi$ is a string, then the string that results from enclosing $\phi$ in single quotes (with whitespace padding if necessary) is a proper noun referring to $\phi$. For example, if we let $\phi = \LQ \text{red}$', then the string which would be printed as
\[
\LQ \text{red}\RQ
\]
(i.e. the string ` `red' ') is a proper noun referring to the word `red'. So for example, `the word `red' has three letters' is a grammatically valid sentence.

If the string $\phi$ contains a substring $\psi$ which is a term referring to a string $\theta$, then I may modify the above convention to stipulate that the string ` `$\phi$' ', instead of referring to $\phi$, refers to the string that results from replacing $\psi$ with $\theta$ in $\phi$. For example, if we let $\psi = \LQ \text{red}$' and $\phi = \LQ \psi\;\psi$', then the string ` `$\psi\;\psi$' ' may refer to the string `red red' instead of to the string `$\psi \;\psi$'. This abuse of notation can hopefully be determined easily from context. Note that I have used it already once in this paragraph.

I use double quotes much more loosely and do not follow any precise convention regarding them.

\item A language is \emph{classical} if it follows the conventions of classical predicate calculus. We take one of these conventions to be the assumption that every syntactically valid \statement\ of the language is either true or false. A language can be classical even if this assumption is not valid, as long as making the assumption is a requirement for using the language.
\item A language is \emph{proposition/term} if it follows the conventions described in Appendix \ref{appendixPT}. Every classical language is proposition/term, but not conversely.
\item A proposition/term language is \emph{selfmeta} if it is capable of transcribing its own metalanguage (see \6\ref{subappendixselfmeta}).
\item A \emph{program} is a description of an algorithm written in a programming language.
\end{itemize}

\ignore{
TO DO list:
\begin{itemize}
\item Relative countability is much weaker than relative expressibility... we can talk about inaccessible, Mahlo etc. \emph{cardinals} (not ordinals) in a descriptionalist sense as long as we relativize it to some fixed ordinal... this may allow a proof that ZF is consistent using descriptivist methods...
\item Talk about ``reflective consistency''
\item Epistemological status of $V=L$ and ``every set is countable'' -- both of them are sharpening of concepts towards the descriptionalist view
\item The sharpening $V=L$ implies CH (even in IKP) but does not seem to ``contradict what the notion of a set is about'', contrary to Feferman16
\item $\phi$ and $\psi$ are terms, while $\alpha$ is an algorithm and $\rho$ is a proof. $P$ and $Q$ are \statements, $A$ is an axiom system, and $X$ is a school of thought.
\item ``standard'' vs. ``criterion'' -- essentially the same but linguistically a little different
\item ``optimal'' vs. ``best'': optimality is always a specific claim, while saying that a strategy is the ``best'' is merely a way of recommending it (according to emotivism).
\item to say that an algorithm \emph{returns} a value is a weaker claim than to say that it \emph{computes} the value; computation must be intentional.
\item humans express thought through languages
\item a \emph{strategy} is a procedure for creating an entry
\item a \emph{principle} is a description of a method of reasoning.
\item ``sound'' vs ``valid'': do we distinguish these well?
\end{itemize}
\begin{itemize}
\item Surprising theorem in \GST: any fact true at an inaccessible cardinal is true in the larger universe. (The converse cannot be formulated but the existence of counterexamples can and is true.)
\item An interesting ``maximality'' hypothesis: There exists a set of \statements\ in \GST, all of them true, which cannot be consistently extended to a larger set. Here we have to be careful what we mean by the word `consistently'.
\item note: part of the liar's paradox analysis is due to Dmitri Bochvar (Bochvar's internal logic)
\item however, our logic is crucially different in that it can't be described by a truth table
\item ``logical implication'' vs ``implication via unproblematic reasoning'' (the notion of unproblematic reasoning is problematically imprecise!) vs material implication
\item an important ``escape hatch'': the \emph{preciseness} of the concept ``reasoning whose soundness is grounded the soundness of $A$ together with the soundness of sound reasoning'', where $A$ is an axiom system
\item Note: There \emph{are} precise notions of precision capable of encompassing full mathematical ontologies. However, the natural-language concept of precision is not itself a precise concept, as is the case for most natural-language concepts.
\item Note: It makes no sense to attempt to define a language using arguments that purport to show that it is plausible that it is definable! (e.g. abductive and inductive reasoning)
\item New axiom: Any physically \emph{instantiated} (not instantiatable) reasoning is sound. So if reasoning will be instantiated, then it will be true, but that does not imply that it is true. This requires the introduction of a ``temporal'' truth predicate. (In a separate document, maybe?)
\item Introduce ordinals where appropriate. Describe in detail the notion of encoding ordinals and using ordinalhood (or ``well-foundedness'') as a predicate.
\item Introduce ``the iterative conception'' for (hyper)predicativism
\item Introduce the concepts of collections and sets where appropriate. (A set is a definite collection.) [NOTE: I would rather define a collection to be a string of the form ``The collection of all $x$ such that $\phi(x)$'' or ``$\{x:\phi(x)\}$'' rather than the bare formula $\phi(x)$.]
\item The old document talked a lot about ``intuition/introspection'' and the question of when/whether it is valid, as well as the notion of making ``philosophical commitments''. It would be nice to add this kind of language where appropriate.
\item Similarly, we should add language about ``the question of which questions are precise / have definite answers is not precise / does not have a definite answer''.
\item Similarly with Lewis Carroll's ``What the Tortoise said to Achilles''.
\item In particular, we should be clear about the strategy used in this paper: make some arguments, then use introspection to figure out why we think those arguments are valid, then make arguments based on those new premises, then repeat.
\item Also emphasize that no axiom system can cover all reasoning of a given school of thought (since then ``this axiom system is valid'' would also be acceptable)
\item Move the discussion of Kolmogorov complexity to a more prominent location.
\item Add a discussion about the $\omega$-conservativity LNC and why the contest is the same as the LNC of $\reclang$
\item Add a discussion about the formalizability of mathematical intuitions. Major point: This paper has succeeded at formalizing many of our intuitions.
\item Play the Philosophical Largest Number Contest for constructivists (using the ideas from the Indefinite Languages section).
\item Add the Conclusion: Four possible maximal ontological frameworks, non-classical, linearly ordered. There are no maximal axiom systems, and the question ``what is the smallest impredicative ordinal?'' does not make sense. Each framework can attempt to comment on what sort of axiom systems the framework below it should use, but any axiom system that tries to transfer philosophical intuitions too directly will be rejected by the adherents of the smaller framework.
\item Can \ZFC\ yield larger numbers than \ZF? Maybe conservativity results are relevant?
\item Primitive recursive algorithms yield a degenerate LNC
\item Analysis of set realists who accept truth as a fundamental notion but deny the existence of classes
\item Pitfalls of the axiomatic strategy: possibility of losing ``by coincidence''; why meta-ing is better than choosing a large constant
\item Criticism of set realists (in particular ``one step back from disaster''): fails to engage Russell's paradox seriously as a philosophical argument (cf. writeup for LessWrong)
\item Citations to philosophical and mathematical literature
\item Clarification: should $L^{(1)}$ be written in \NT\ or \ST?
\item Version of Theorem \ref{theoremrecstrat} for non-classical logic?
\item Add remark: ``Godelization'' requires analyzing the philosophical underpinnings behind the reasoning that you use.
\item Add material from Largest\_number\_addendum.tex
\item Any definite \statement\ in $\reclang$ is equivalent to a \statement\ from a language coming from an ordinal
\item Any ordinal is equivalent to a computable ordinal?

\end{itemize}

}

\bibliographystyle{amsplain}

\bibliography{bibliography}

\providecommand{\bysame}{\leavevmode\hbox to3em{\hrulefill}\thinspace}
\providecommand{\MR}{\relax\ifhmode\unskip\space\fi MR }
\providecommand{\MRhref}[2]{%
  \href{http://www.ams.org/mathscinet-getitem?mr=#1}{#2}
}
\providecommand{\href}[2]{#2}
\begin{thebibliography}{1}

\bibitem{Feferman5}
Solomon Feferman, \emph{Why a little goes a long way: Logical foundations of
  scientifically applicable mathematics}, PSA 1992, vol.~2, Philosophy of
  Science Association, East Lansing, Michigan, 1993, pp.~442--455.

\bibitem{GBCST}
Scott Garrabrant, Tsvi Benson-Tilsen, Andrew Critch, Nate Soares, and Jessica
  Taylor, \emph{Logical induction}, https://arxiv.org/abs/1609.03543, preprint
  2016.

\bibitem{Hellman}
Geoffrey Hellman, \emph{On the scope and force of indispensability arguments},
  PSA 1992, vol.~2, Philosophy of Science Association, East Lansing, Michigan,
  1993, pp.~456--464.

\bibitem{Maddy1}
Penelope Maddy, \emph{Believing the axioms. {I}}, J. Symbolic Logic \textbf{53}
  (1988), no.~2, 481--511. \MR{947855}

\bibitem{Maddy2}
\bysame, \emph{Believing the axioms. {II}}, J. Symbolic Logic \textbf{53}
  (1988), no.~3, 736--764. \MR{960996}

\bibitem{McCallum}
Rupert McCallum, \emph{A proposed characterisation of the intrinsically
  justified reflection principles}, \url{http://arxiv.org/abs/1403.8058},
  preprint 2014.

\bibitem{Russell}
Bertrand Russell, \emph{Introduction to mathematical philosophy}, second ed.,
  Dover Publications, Inc., New York, 1993. \MR{1243640}

\bibitem{Spencer}
Joel Spencer, \emph{Large numbers and unprovable theorems}, Amer. Math. Monthly
  \textbf{90} (1983), no.~10, 669--675. \MR{723939}

\bibitem{Tarski}
Alfred Tarski, \emph{The concept of truth in formalized languages}, Logic,
  Semantics, Metamathematics, Hackett, 1983.

\end{thebibliography}

\end{document}